\documentclass[11pt, twoside, leqno]{article}

\usepackage{amssymb}
\usepackage{amsmath}
\usepackage{amsthm}
\usepackage{color}
\usepackage{mathrsfs}

\usepackage{indentfirst}

\usepackage{txfonts}

\allowdisplaybreaks

\def\red{\color{red}}

\def\rr{{\mathbb R}}
\def\rn{{\mathbb{R}^n}}

\def\zz{{\mathbb Z}}
\def\cc{{\mathbb C}}

\def\nn{{\mathbb N}}

\def\ca{{\mathcal A}}

\def\cf{{\mathcal F}}

\def\cl{{\mathcal L}}
\def\cm{{\mathcal M}}
\def\cn{{\mathcal N}}
\def\cp{{\mathcal P}}
\def\cq{{\mathcal Q}}

\def\cs{{\mathcal S}}

\def\fz{\infty }

\def\lf{\left}
\def\r{\right}

\def\hs{\hspace{0.25cm}}
\def\ls{\lesssim}

\def\noz{\nonumber}

\def\wh{\widehat}

\def\gfz{\genfrac{}{}{0pt}{}}

\def\dist{\mathop\mathrm{\,dist\,}}

\def\BMO{\mathop\mathrm{\,BMO\,}}

\def\loc{{\mathop\mathrm{\,loc\,}}}
\def\supp{\mathop\mathrm{\,supp\,}}
\def\atom{\mathop\mathrm{\,atom\,}}

\def\Xint#1{\mathchoice
{\XXint\displaystyle\textstyle{#1}}%
{\XXint\textstyle\scriptstyle{#1}}%
{\XXint\scriptstyle\scriptscriptstyle{#1}}%
{\XXint\scriptscriptstyle\scriptscriptstyle{#1}}%
\!\int}
\def\XXint#1#2#3{{\setbox0=\hbox{$#1{#2#3}{\int}$ }
\vcenter{\hbox{$#2#3$ }}\kern-.6\wd0}}

\def\dashint{\Xint-}

\def\({\left(}
\def \){ \right)}
\def\BB{{\mathbb B}}

\newtheorem{theorem}{Theorem}[section]
\newtheorem{lemma}[theorem]{Lemma}

\newtheorem{assumption}[theorem]{Assumption}
\newtheorem{proposition}[theorem]{Proposition}

\theoremstyle{definition}
\newtheorem{remark}[theorem]{Remark}
\newtheorem{definition}[theorem]{Definition}
\renewcommand{\appendix}{\par
   \setcounter{section}{0}%
   \setcounter{subsection}{0}%
   \setcounter{subsubsection}{0}%
   \gdef\thesection{\@Alph\c@section}%
   \gdef\thesubsection{\@Alph\c@section.\@arabic\c@subsection}%
   \gdef\theHsection{\@Alph\c@section.}%
   \gdef\theHsubsection{\@Alph\c@section.\@arabic\c@subsection}%
   \csname appendixmore\endcsname
 }

\numberwithin{equation}{section}

\begin{document}

\title{\bf\Large Weak Hardy-Type Spaces Associated with Ball Quasi-Banach Function Spaces
I: Decompositions with Applications to Boundedness of Calder\'on--Zygmund Operators
\footnotetext{\hspace{-0.35cm} 2010 \emph{Mathematics Subject Classification}. Primary 42B30;
Secondary 42B25, 42B20, 42B35, 46E30.
\endgraf \emph{Key words and phrases}. ball quasi-Banach function space, weak Hardy space,
Orlicz-slice space, maximal function, atom, molecule, Calder\'on--Zygmund operator.
\endgraf This project is supported
by the National Natural Science Foundation of China (Grant Nos.
11571039, 11761131002, 11671185, 11726622, 11726621 and 11871100).}}
\author{Yangyang Zhang, Songbai Wang, Dachun Yang\,\footnote{Corresponding author,
E-mail: \texttt{dcyang@bnu.edu.cn}/{\red June 28, 2019}/Final version.}
\  and Wen Yuan}
\date{}
\maketitle

\vspace{-0.7cm}

\begin{center}

\begin{minipage}{13cm}
{\small {\bf Abstract}\quad Let $X$ be a ball quasi-Banach function space on ${\mathbb R}^n$.
In this article, the authors introduce the weak Hardy-type space $WH_X({\mathbb R}^n)$, associated with $X$, via the radial
maximal function. Assuming that the powered Hardy--Littlewood maximal operator
satisfies some Fefferman--Stein vector-valued maximal inequality on $X$ as well as it is
bounded on both the weak ball quasi-Banach function space $WX$ and the associated space,
the authors then establish several real-variable characterizations of $WH_X({\mathbb R}^n)$, respectively, in terms of
various maximal functions, atoms and molecules. As an application,
the authors obtain the boundedness of Calder\'on--Zygmund operators
from the Hardy space $H_X({\mathbb R}^n)$ to $WH_X({\mathbb R}^n)$, which includes the critical
case. All these results are of wide applications. Particularly,
when $X:=M_q^p({\mathbb R}^n)$ (the Morrey space), $X:=L^{\vec{p}}({\mathbb R}^n)$ (the mixed-norm Lebesgue space)
and $X:=(E_\Phi^q)_t({\mathbb R}^n)$ (the Orlicz-slice space), which are all ball quasi-Banach function spaces
but not quasi-Banach function spaces, all these results are even new.
Due to the generality, more applications of these results are predictable.
}
\end{minipage}
\end{center}

\tableofcontents

\section{Introduction\label{s1}}

It is well known that the classical Hardy space $H^p(\rn)$ with $p\in (0,1]$, which was introduced
by Stein and Weiss \cite{SW} and further developed by Fefferman and Stein \cite{FS0}, plays
a key role in harmonic analysis and partial differential equations. These works
\cite{FS0,SW} inspire many new ideas for the real-variable theory of function spaces. It is worth
to pointing out that the real-variable characterizations of classical Hardy spaces reveal the
intrinsic connections among some important notions in harmonic analysis, such as harmonic
functions, maximal functions and square functions. In
recent decades, various variants of classical Hardy spaces have been introduced
and their real-variable theories have been well developed; these variants include weighted Hardy
spaces (see \cite{ST}), (weighted) Herz--Hardy spaces (see, for instance,
\cite{CL,GC1,GCH,LY1,LY2}),  (weighted) Hardy--Morrey spaces (see, for instance
\cite{JW,Sa,H}), Hardy--Orlicz spaces (see, for instance, \cite{IV,Se,V,NS1,YY}),
Lorentz Hardy spaces (see, for instance, \cite{AST}), Musielak--Orlicz Hardy spaces
(see, for instance, \cite{K,YLK}) and variable Hardy spaces (see, for instance, \cite{CUW,NS, YZN}).
Observe that these elementary spaces on which the aforementioned Hardy spaces were built,
such as (weighted) Lebesgue spaces, (weighted) Herz spaces, (weighted) Morrey spaces,
mixed-norm Lebesgue spaces, Orlicz spaces, Lorentz spaces, Musielak--Orlicz spaces and variable Lebesgue spaces,
are all included in a generalized framework called ball quasi-Banach function spaces
which were introduced, very recently, by Sawano et al. \cite{SHYY}.
Moreover, Sawano et al. \cite{SHYY} and Wang et al. \cite{wyy19} established a unified real-variable theory
for Hardy spaces associated with ball quasi-Banach function spaces on ${\mathbb R}^n$
and gave some applications of these Hardy-type spaces to the boundedness of
Calder\'on--Zygmund operators and pseudo-differential operators.

Recall that ball quasi-Banach function spaces are a generalization of
quasi-Banach function spaces. Compared with quasi-Banach function spaces,
ball quasi-Banach function spaces contain more function spaces.
For instance, the Morrey spaces are ball quasi-Banach function spaces,
which are not quasi-Banach function spaces and hence the class
of quasi-Banach function spaces is a proper subclass of ball quasi-Banach function spaces;
see \cite{SHYY} for more details. Let $X$ be a ball quasi-Banach function space
(see \cite{SHYY} or Definition \ref{Debqfs} below). Sawano et al. \cite{SHYY} introduced
the Hardy space $H_X(\rn)$ via the grand maximal function
(see \cite{SHYY} or Definition \ref{DeHX} below). Assuming that the Hardy--Littlewood maximal function is bounded on
the $p$-convexification of $X$, Sawano et al. \cite{SHYY} established
several different  maximal function characterizations of $H_X(\rn)$.
On another hand, Coifman \cite{C} and Latter \cite{La} found the most useful atomic characterization of
classical Hardy spaces $H^p(\rn)$, which plays an important role in
developing the real-variable theory of Hardy spaces. Sawano et al. \cite{SHYY}
found that these atomic characterizations strongly depend on
the Fefferman--Stein vector-valued maximal inequality and the boundedness on the associated
space of the powered Hardy--Littlewood maximal operator. Another key tool used
in \cite{SHYY} to deal with the convergence of the atomic decomposition is the
Herz space $K_{p,q}^{-n/p}(\rn)$ with $p,\ q\in(0,1)$.
Roughly speaking, one can embed $X$ into $K_{p,q}^{-n/p}(\rn)$ and, moreover, $K_{p,q}^{-n/p}(\rn)$
does not contain the constant function $1$, which is crucial to the application of the Whitney
decomposition theorem in the proof of \cite[Proposition 4.9]{SHYY}.

Recall that, to find the biggest function space $\ca$ such that Calder\'on--Zygmund operators
are bounded from $\ca$ to $WL^1(\rn)$, Fefferman and Soria \cite{FSo} originally introduced the weak Hardy space
$WH^1(\rn)$ and they did obtain the boundedness of the convolutional Calder\'on--Zygmund operator
with kernel satisfying the Dini condition
from $WH^1(\rn)$ to $WL^1(\rn)$ by using the $\infty$-atomic characterization of $WH^1(\rn)$.
It is well known that the classic Hardy spaces  $H^p(\rn)$, with $p\in(0,1]$,
are good substitutes of Lebesgue spaces $L^p(\rn)$  when studying the boundedness
of some Calder\'on--Zygmund operators. For instance, if $\delta\in(0,1]$ and $T$ is
a convolutional $\delta$-type Calder\'on--Zygmund operator, then $T$ is bounded on
$H^p(\rn)$ for any given $p\in(n/(n+\delta),1]$ (see \cite{AM1986}).  However, this is not
true when
$$p=n/(n+\delta)$$
which is called the \emph{critical case} or the \emph{endpoint case}.
Liu \cite{L} introduced the weak Hardy spaces $WH^p(\rn)$ with $p\in(0,1]$ and
proved that the aforementioned operator $T$ is bounded from $H^{n/(n+\delta)}(\rn)$ to $WH^{n/(n+\delta)}(\rn)$
via first establishing the $\infty$-atomic characterization of the weak Hardy space $WH^p(\rn)$.
Thus, the classical weak Hardy spaces $WH^p(\rn)$ play an irreplaceable role in the study
of the boundedness of operators in the critical case. Recently, He \cite{He} and Grafakos
and He \cite{GH} further studied vector-valued weak Hardy space
$H^{p,\infty}(\rn,\ell^2)$ with $p\in(0,\infty)$. In 2016, Liang et al. \cite{LYJ} (see also \cite{YLK})
considered the weak Musielak--Orlicz type Hardy space $WH^\varphi(\rn)$, which covers
both the weak Hardy space $WH^p(\rn)$ and the weighted weak Hardy space
$WH^p_\omega(\rn)$ from \cite{QY},  and obtained various equivalent characterizations of
$WH^\varphi(\rn)$, respectively, in terms of maximal functions, atoms, molecules and Littlewood--Paley functions,
as well as the boundedness of Calder\'on--Zygmund operators in the critical case. Meanwhile,
Yan et al. \cite{YYYZ} developed a real-variable theory of variable weak Hardy spaces
$WH^{p(\cdot)}(\rn)$ with $p(\cdot)\in C^{\log}(\rn)$.

Let $X$ be a ball quasi-Banach function space on ${\mathbb R}^n$ introduced by Sawano et al. in \cite{SHYY}.
In this article, we introduce the weak Hardy-type space $WH_X({\mathbb R}^n)$, via the radial
maximal function, associated with $X$. Assuming that the powered Hardy--Littlewood maximal operator
satisfies some Fefferman--Stein vector-valued maximal inequality on $X$ as well as it is
bounded on both the weak ball quasi-Banach function space $WX$ and the associated space,
we then establish some real variable characterizations of $WH_X({\mathbb R}^n)$, respectively,
in terms of various maximal functions, atoms and molecules. Using the atomic
characterization of $H_X({\mathbb R}^n)$, we further obtain the boundedness of Calder\'on--Zygmund operators
from the Hardy space $H_X({\mathbb R}^n)$ to $WH_X({\mathbb R}^n)$, which includes the critical
case. All these results are of wide applications and, particularly,
when $X:=M_q^p({\mathbb R}^n)$ (the Morrey space) introduced by Morrey \cite{m38}
(or see Definition \ref{mmus} below),
$X:=L^{\vec{p}}(\rn)$ (the mixed-norm Lebesgue space) (see, for instance, \cite{bp,H60}
or Definition \ref{mix} below) and $X:=(E_\Phi^q)_t(\rn)$ (the Orlicz-slice space)
introduced in \cite{ZYYW} (or see Definition \ref{DeOs} below),
all these results are even new.

To establish the atomic characterization of the weak Hardy-type space $WH_X(\rn)$,
similarly to \cite{SHYY}, we find that it strongly depends on the Fefferman--Stein
vector-valued maximal inequality (see Assumption \ref{a2.15} below)
and the boundedness on the associate space of the powered Hardy--Littlewood
maximal operator [see \eqref{Eqdm}]. Then, using the atomic characterization of
$WH_X(\rn)$, we further establish the molecular characterization
of $WH_X(\rn)$. As applications,
when $X$ further satisfies \eqref{Eqfs11} or \eqref{Eqfs12} (the Fefferman--Stein
vector-valued maximal inequality from $X$ to $WX$), we prove that the convolutional
$\delta$-type and the non-convolutional $\gamma$-order Calder\'on--Zygmund
operators are bounded form $H_X(\rn)$ to $WH_X(\rn)$ including the critical case
$p_-= n/(n+\delta)$ or $p_-=n/(n+\gamma)$, with $p_-$ as in Assumption \ref{a2.15} below.
Moreover, when $X$ is the Morrey space $M_q^p({\mathbb R}^n)$, the mixed-norm Lebesgue
space $L^{\vec{p}}({\mathbb R}^n)$ or the Orlicz-slice space $(E_\Phi^q)_t(\rn)$,
we find that all assumptions on $X$ of this article in these cases hold true
and hence all results obtained in this article hold true and new even for these spaces.

Also, to limit the length of this article, applying these characterizations of
$WH_X({\mathbb R}^n)$ in this article, we in \cite{wyyz} establish various Littlewood--Paley
function characterizations of $WH_X(\rn)$ and prove that the real interpolation
intermediate space $(H_{X}({\mathbb R}^n),L^\infty({\mathbb R}^n))_{\theta,\infty}$,
between $H_{X}({\mathbb R}^n)$ and $L^\infty({\mathbb R}^n)$, is
$WH_{X^{{1}/{(1-\theta)}}}({\mathbb R}^n)$, where $\theta\in (0, 1)$.
These results in \cite{wyyz} are also of wide applications; particularly,
when $X:=M_q^p({\mathbb R}^n)$ (the Morrey space), $X:=L^{\vec{p}}({\mathbb R}^n)$ (the mixed-norm Lebesgue space)
and $X:=(E_\Phi^q)_t({\mathbb R}^n)$ (the Orlicz-slice space), all these results are even new;
when $X:=L_\omega^\Phi({\mathbb R}^n)$ (the weighted Orlicz space),
the result on the real interpolation is new
and, when  $X:=L^{p(\cdot)}({\mathbb R}^n)$ (the variable Lebesgue space) and
$X:=L_\omega^\Phi({\mathbb R}^n)$, the Littlewood--Paley function characterizations
of $WH_X({\mathbb R}^n)$ obtained in \cite{wyyz} improve the existing results
via weakening the assumptions on the Littlewood--Paley
functions; see \cite{wyyz} for more details. It is easy to see
that, due to the generality, more applications of these results obtained
in both the present paper and \cite{wyyz} are predictable.

To be precise, this article is organized as follows.

In Section \ref{s2}, we recall some notions concerning the ball (quasi)-Banach function space $X$ and
the weak ball (quasi)-Banach function space $WX$. Then we state the assumptions of the Fefferman--Stein vector-valued
maximal inequality on $X$ (see Assumption \ref{a2.15} below)
and the boundedness on the $p$-convexification of $WX$ for the Hardy--Littlewood maximal operator
(see Assumption \ref{xinm}). Finally, in Definition \ref{DewSH} below, we introduce the weak
Hardy space $WH_X(\rn)$ via the radial grand maximal function.

Under the assumption about the boundedness on the $p$-convexification of $WX$ for the Hardy--Littlewood
maximal operator [see \eqref{Eqdm}], we establish various real-variable characterizations of $WH_X(\rn)$ in
Theorem \ref{Thmc} below of Section \ref{s3},
respectively, in terms of the radial maximal function, the grand maximal function, the non-tangential maximal function,
the maximal function of Peetre type and the grand maximal function of Peetre type
(see Definition \ref{Dem} below).
If $WX$ satisfies an additional assumption \eqref{Eqinf} (namely, the $WX$-norm of the
characteristic function of any unit ball of $\rn$ has a low bound),
we then characterize $WH_X(\rn)$ by means of the non-tangential maximal function
with respect to Poisson kernels in Theorem \ref{Thmp} below. Moreover, the relations
between $WX$ and $WH_X(\rn)$ are also clarified in this section.

Section \ref{s4} is devoted to establishing the atomic characterization of $WH_X(\rn)$.
Under the assumption that $X$ satisfies the Fefferman--Stein vector-valued inequality
and is $\vartheta$-concave for some $\vartheta\in(1,\infty)$, we show that any $f\in WH_X(\rn)$
has an atomic decomposition in terms of $(X,\infty,d)$-atoms in Theorem \ref{Thad} below.
Recall that the atomic decomposition of $H^p(\rn)$ with $p\in (0,1]$ was obtained
via a dense argument which does not work for the atomic decomposition of $WH^p(\rn)$
due to the lack of a suitable dense subset of $WH^p(\rn)$. We have the same problem for
$WH_X(\rn)$. To overcome this difficulty, we obtain the atomic decomposition
of $WH_X(\rn)$ via using some ideas from \cite{C1977,LYJ,YYYZ}, namely, in the proof of Theorem \ref{Thad},
we need to use the global Calder\'on reproducing formula in $\mathcal{S}'(\rn)$
(see Lemma \ref{Le47} below),
the generalized Campanato space,
and the Alaoglu theorem. To obtain the reconstruction theorem in terms of $(X,q,d)$-atoms
(see Theorem \ref{Thar}), we need to further assume that $X$ is strictly $r$-convex for any $r\in(0,p_-)$,
where $p_-$ is as in Assumption \ref{a2.15}, and the boundedness on the associate space of the powered Hardy--Littlewood
maximal operator \eqref{Eqdm}, besides the Fefferman--Stein vector-valued inequality.

In Section \ref{s5}, we establish the molecular characterization of $WH_X(\rn)$
in Theorems \ref{Thmolcha0} and \ref{Thmolcha} below with all the same assumptions as
in the atomic decomposition theorem (Theorem \ref{Thad}) and
the reconstruction theorem (Theorem \ref{Thar}). Since each atom of $WH_X(\rn)$ is also
a molecule of $WH_X(\rn)$, to prove
Theorem \ref{Thmolcha}, it suffices to show that the weak molecular Hardy
space $WH_{\mathrm{mol}}^{X,q,d,\epsilon}(\rn)$
is continuously embedded into $WH_X(\rn)$ due to Theorems \ref{Thad} and \ref{Thar}.
To this end, a key step is to prove that an $(X,q,d,\epsilon)$-molecule can be divided
into an infinite linear combination of $(X,q,d)$-atoms. We show this via borrowing some ideas
from the proof of \cite[Theorem 5.3]{YYYZ}.

Section \ref{s6} is devoted to proving that both the convolutional $\delta$-type Calder\'on--Zygmund operator
and the non-convolutional
$\gamma$-order Calder\'on--Zygmund operator are bounded from $H_X(\rn)$ to $WH_X(\rn)$ in the critical case
when $p_-=\frac n{n+\delta}$ or when $p_-=\frac n{n+\gamma}$
(see Theorems \ref{Thcz} and \ref{Thcz1} below). In this case,
any convolutional $\delta$-type or any non-convolutional $\gamma$-order Calder\'on--Zygmund operator may
not be bounded on $H_X(\rn)$ even when $X=L^p(\rn)$ with $p\in(0,1]$. In this sense, the space $WH_X(\rn)$
is a proper substitution of $H_X(\rn)$ in the critical case for the study on the boundedness of
some operators.

In Section \ref{s7}, we apply the above results to
the Morrey space, the mixed-norm Lebesgue space and the
Orlicz-slice space, respectively, in Subsections
\ref{s7.1}, \ref{s7.2} and \ref{s7.3}.

Recall that, due to the applications in elliptic partial differential equations,
the Morrey space $M_q^p(\rn)$ with $0<q\leq p<\infty$
was introduced by Morrey \cite{m38} in 1938.
In recent decades, there exists an increasing interest in applications
of Morrey spaces to various areas of analysis,
such as partial differential equations, potential theory and
harmonic analysis (see, for instance, \cite{a15,a04,cf,JW,l16,lyy,ysy}).
Particularly,  Jia and Wang \cite{JW} introduced the Hardy--Morrey spaces and established their atomic
characterizations.
Later, based on the Morrey space, various variants of
Hardy--Morrey spaces have been introduced and
developed, such as weak Hardy--Morrey spaces (see Ho \cite{H17}),  variable Hardy--Morrey spaces
(see \cite{H15}) and
Besov--Morrey spaces and Triebel--Lizorkin--Morrey spaces (see \cite{Sa}).
Observe that, as was pointed out in \cite[p.\,86]{SHYY}, $M_q^p(\rn)$ with $1\le q<p<\infty$, which violates
\eqref{Eqbv} below (see \cite[Example 3.3]{ST}), is not a Banach function space
as in Definition \ref{bn}, but it does be a ball
Banach function space as in Definition \ref{Debqfs}.
In Subsection \ref{s7.1}, We first recall some of the useful properties of Morrey spaces.
Borrowing some ideas from \cite{tx}, we establish a weak-type vector-valued inequality of the Hardy--Littlewood
maximal operator $\cm$ from the Morrey space $M_1^p(\rn)$ to the weak Morrey space $WM_1^p(\rn)$
with $p\in[1,\infty)$ (see Proposition \ref{pro822} below).
From this and the results in \cite{cf,H15,H17}, we can easily show that
all the assumptions of main theorems in Sections \ref{s3} through \ref{s6} are satisfied.
Thus, applying these theorems, we obtain the atomic and the molecular characterizations of weak Hardy--Morrey spaces
and the boundedness of Calder\'on--Zygmund
operators from the Hardy--Morrey spaces
to the weak Hardy--Morrey spaces including the critical case.

The study of mixed-norm Lebesgue spaces $L^{\vec{p}}(\rn)$ with
$\vec{p}\in (0,\infty]^n$ originated from Benedek and Panzone
\cite{bp} in the early 1960's, which can be traced back to H\"ormander \cite{H60}.
Later on, in 1970, Lizorkin \cite{l70} further developed both the theory of
multipliers of Fourier integrals and estimates of convolutions
in the mixed-norm Lebesgue spaces.
Particularly,
in order to meet the requirements arising in the study of the boundedness of operators,
partial differential equations and some other fields, the real-variable theory of mixed-norm
function spaces, including mixed-norm Morrey spaces,
mixed-norm Hardy spaces, mixed-norm Besov
spaces and mixed-norm Triebel--Lizorkin spaces, has rapidly been developed
in recent years (see, for instance,
\cite{cgn17bs,gjn17,tn,hy,hlyy,hlyy19}).
Observe that $L^{\vec{p}}(\rn)$ when
$\vec{p}\in (0,\infty]^n$ is a ball quasi-Banach function space, but, it
is not a quasi-Banach function space (see Remark \ref{mbfs-not} below).
In Subsection \ref{s7.2}, to establish a vector-valued inequality of the Hardy--Littlewood
maximal operator $\cm$ on the weak mixed-norm Lebesgue space $WL^{\vec{p}}(\rn)$ with
$\vec{p}\in (1,\infty)^n$ (see Theorem \ref{mixProfs} below), we first establish
an interpolation theorem of sublinear operators on the
space $WL^{\vec{p}}(\rn)$. Then, via an extrapolation theorem (see Lemma \ref{waicha} below)
which is a slight variant of a special case of \cite[Theorem 4.6]{CMP},
we establish a vector-valued inequality of the Hardy--Littlewood
maximal operator $\cm$ from $L^{\vec{p}}(\rn)$
to $WL^{\vec{p}}(\rn)$ with
$\vec{p}\in [1,\infty)^n$ (see Proposition \ref{mixPro} below). Since
all the assumptions of main theorems in Sections \ref{s3} through \ref{s6} are satisfied,
applying these theorems,
we obtain the atomic and the molecular characterizations of weak Hardy--Morrey spaces
and the boundedness of Calder\'on--Zygmund
operators from the mixed-norm Hardy spaces
to the weak mixed-norm Hardy spaces including the critical case.

In Subsection \ref{s7.3}, let $q,\ t\in(0,\infty)$ and $\Phi$ be an Orlicz function. Recall that the
Orlicz-slice space $(E_\Phi^q)_t(\rn)$ introduced in \cite{ZYYW}
generalizes both the slice space $E_t^p(\rn)$ [in this case, $\Phi(\tau):=\tau^2$ for any $\tau\in[0,\infty)$], which
was originally introduced by Auscher and
Mourgoglou \cite{AM2014} and has been applied to study the classification of weak solutions in the natural classes
for the boundary value problems of a $t$-independent elliptic system in the upper plane,
and $(E_r^p)_t(\rn)$ [in this case, $\Phi(\tau):=\tau^r$ for any $\tau\in[0,\infty)$ with $r\in(0,\infty)$],
which was originally introduced by Auscher and Prisuelos-Arribas \cite{APA}
and has been applied to study the boundedness of operators such
as the Hardy--Littlewood maximal operator, the Calder\'on--Zygmund operator and the Riesz potential.
The Orlicz-slice space $(E_\Phi^q)_t(\rn)$ is a ball quasi-Banach function space, however,
they may not be a quasi-Banach function space [see Remark \ref{Re25}(i) for more details].
Moreover, Zhang et al. \cite{ZYYW} introduced the Orlicz-slice Hardy space $(HE_\Phi^q)_t(\rn)$
and obtained real-variable characterizations of $(HE_\Phi^q)_t(\rn)$, respectively,
in terms of various maximal functions, atoms, molecules and Littlewood--Paley functions, and the boundedness
on $(HE_\Phi^q)_t(\rn)$ for convolutional $\delta$-order and non-convolutional $\gamma$-order Calder\'on--Zygmund
operators. Naturally, this new scale of Orlicz-slice Hardy spaces contains the variant of the Hardy-amalgam space
[in this case, $t=1$ and $\Phi(\tau):=\tau^p$ for any $\tau\in [0,\fz)$ with $p\in (0,\fz)$] of Abl\'e and Feuto \cite{AF}
as a special case. Moreover, the results in \cite{ZYYW} indicate that, similarly to the classical
Hardy space $H^p(\rn)$ with $p\in (0,1]$, $(HE_\Phi^q)_t(\rn)$ is a good substitute of
$(E_\Phi^q)_t(\rn)$ in the study on the boundedness of operators.
On another hand, observe that $(E_\Phi^p)_t(\rn)$ when $p=t=1$ goes back to the amalgam space $(L^\Phi,\ell^1)(\rn)$
introduced by Bonami and Feuto \cite{BF}, where
$$\Phi(t):=\frac{t}{\log(e+t)}$$
for any $t\in[0,\infty),$
and the Hardy space $H_*^\Phi(\rn)$
associated with the amalgam space $(L^\Phi,\ell^1)(\rn)$
was applied by Bonami and Feuto \cite{BF} to study the linear decomposition of the product of
the Hardy space $H^1(\rn)$ and its dual space $\BMO(\rn)$. Another main
motivation to introduce $(HE_\Phi^q)_t(\rn)$ in \cite{ZYYW} exists in that it is
a natural generalization of $H_*^\Phi(\rn)$ in \cite{BF}.
In the last part of this section, we focus on the weak Orlicz-slice Hardy space $(WHE_\Phi^q)_t(\rn)$
built on the Orlicz-slice space $(E_\Phi^q)_t(\rn)$, which is actually the starting point of this article.
We first recall some of the useful properties of Orlicz-slice spaces.
To obtain the atomic characterization of  $(WHE_\Phi^q)_t(\rn)$, we only need
to show that the powered Hardy--Littlewood maximal operator is bounded on the weak Orlicz-slice space
$(WE_\Phi^q)_t(\rn)$ (see Definition \ref{DewOs} below),
because $(E_\Phi^q)_t(\rn)$, as a ball quasi-Banach space, has been proved, in \cite{ZYYW},
to satisfy all the other assumptions appeared in Theorems \ref{Thad} and \ref{Thar}.
To this end, we first establish an interpolation theorem
of Marcinkiewicz type for sublinear operators on $(WE_\Phi^q)_t(\rn)$
(see Theorem \ref{Th1} below). As a corollary, we immediately obtain
the vector-valued inequality of the Hardy--Littlewood maximal operator $\cm$ on $(WE_\Phi^q)_t(\rn)$.
To prove Theorem \ref{Th1}, differently from the proofs of \cite[Theorem 2.5]{LYJ} and \cite[Theorem 3.1]{YYYZ},
we cannot directly apply the Fubini theorem. We overcome this difficulty by establishing a Minkowski type
inequality mixed with the norms of both the Lebesgue space $L^1(\rn)$
and the Orlicz space $L^\Phi(\rn)$ with the lower type $p_\Phi^-\in(1,\infty)$ (see Lemma \ref{LeMI} below).
As an application, we obtain the boundedness of Calder\'on--Zygmund
operators from the Orlicz-slice Hardy space $(HE_\Phi^q)_t(\rn)$
to $(WHE_\Phi^q)_t(\rn)$ in the critical case. To this end, applying Theorems \ref{Thcz}
and \ref{Thcz1}, we only need to establish the Fefferman--Stein vector-valued inequality for the Hardy--Littlewood
maximal operator from $(E_\Phi^q)_t(\rn)$ to $(WE_\Phi^q)_t(\rn)$. We do this by borrowing
some ideas from \cite{ZYYW}.

Finally, we make some conventions on notation. Let $\nn:=\{1,2,\ldots\}$, $\zz_+:=\nn\cup\{0\}$
and $\zz_+^n:=(\zz_+)^n$.
We always denote by $C$ a \emph{positive constant} which is independent of the main parameters,
but it may vary from line to line. We also use $C_{(\alpha,\beta,\ldots)}$ to denote a positive constant depending
on the indicated parameters $\alpha,\beta,\ldots.$ The \emph{symbol} $f\lesssim g$ means that $f\le Cg$.
If $f\lesssim g$ and $g\lesssim f$, then we write $f\sim g$. We also use the following
convention: If $f\le Cg$ and $g=h$ or $g\le h$, we then write $f\ls g\sim h$
or $f\ls g\ls h$, \emph{rather than} $f\ls g=h$
or $f\ls g\le h$. The \emph{symbol} $\lfloor s\rfloor$ (resp., $\lceil s\rceil$) for any $s\in\mathbb{R}$
denotes the maximal (resp., minimal) integer not greater (resp., less)
than $s$. We use $\vec0_n$ to denote the \emph{origin} of $\rn$ and let
$\mathbb{R}^{n+1}_+:=\rn\times(0,\infty)$.
If $E$ is a subset of $\rn$, we denote by $\mathbf{1}_E$ its
characteristic function and by $E^\complement$ the set $\rn\setminus E$.
For any cube $Q:=Q(x_Q,l_Q)\subset\rn$,
with center $x_Q\in\rn$ and side length $l_Q\in(0,\infty)$,
and $\alpha\in(0,\infty)$, let $\alpha Q:=Q(x_Q,\alpha l_Q)$.
Denote by $\cq$ the set of all cubes having their edges parallel to the coordinate axes. For any
$\theta:=(\theta_1,\ldots,\theta_n)\in\zz_+^n$, let $|\theta|:=\theta_1+\cdots+\theta_n$. Furthermore,
for any cube $Q$ in $\rn$ and $j\in\zz_+$, let $S_j(Q):=(2^{j+1}Q)\setminus(2^jQ)$ with $j\in\nn$
and $S_0(Q):=2Q$. Finally, for any $q\in[1,\infty]$, we denote by $q'$ its \emph{conjugate exponent},
namely, $1/q+1/q'=1$.

\section{Preliminaries\label{s2}}

In this section, we present some notions and preliminary facts on ball quasi-Banach function spaces.

\subsection{Ball quasi-Banach function spaces}
Denote by the \emph{symbol} $\mathscr M(\rn)$ the set of
all measurable functions on $\rn$. Let us first recall
the notion of Banach function spaces;  see, for instance, \cite[Chapter 1, Definitions 1.1 and 1.3]{BS}.

\begin{definition}\label{bn}
A Banach space $Y\subset\mathscr M(\rn)$ is called a \emph{Banach function space} if the norm $\|\cdot\|_Y$
is a \emph{Banach function norm}, that is, for all measurable functions $f,\ g$ and $\{f_m\}_{m\in\nn}$, the following
properties hold true:
\begin{itemize}
\item[(i)] $\|f\|_Y=0$ if and only if $f=0$ almost everywhere;
\item[(ii)] $|g|\le |f|$ almost everywhere implies that $\|g\|_Y\le\|f\|_Y$;
\item[(iii)] $0\le f_m\uparrow f$ almost everywhere implies that $\|f_m\|_Y\uparrow\|f\|_Y$;
\item[(iv)] $\mathbf{1}_E\in Y$ for any measurable set $E\subset\rn$ with finite measure;
\item[(v)] for any measurable set $E\subset\rn$ with finite measure, there exists a positive constant $C_{(E)}$,
depending on $E$, such that, for any $f\in Y$,
\begin{equation}\label{Eqbv}
\int_E|f(x)|\,dx\le C_{(E)}\|f\|_Y.
\end{equation}
\end{itemize}
\end{definition}
\begin{remark}
It was pointed out in \cite[p.\,9]{SHYY} that we sometimes describe the quality of
functions via some function spaces beyond Banach function spaces,
for instance, Morrey spaces $M_q^p(\rn)$ with $1\le q<p<\infty$, which violates
\eqref{Eqbv} (see \cite[Example 3.3]{ST}).
It is the point which motivated Sawano et al. \cite{SHYY} to introduce a more general
framework than Banach function spaces, ball quasi-Banach function spaces.
\end{remark}

For any $x\in\rn$ and $r\in(0,\infty)$, let $B(x,r):=\{y\in\rn:\ |x-y|<r\}$ and
\begin{equation}\label{Eqball}
\BB:=\lf\{B(x,r):\ x\in\rn\quad\text{and}\quad r\in(0,\infty)\r\}.
\end{equation}

\begin{definition}\label{Debqfs}
A quasi-Banach space $X\subset\mathscr M(\rn)$ is called a \emph{ball quasi-Banach function space} if it satisfies
\begin{itemize}
\item[(i)] $\|f\|_X=0$ implies that $f=0$ almost everywhere;
\item[(ii)] $|g|\le |f|$ almost everywhere implies that $\|g\|_X\le\|f\|_X$;
\item[(iii)] $0\le f_m\uparrow f$ almost everywhere implies that $\|f_m\|_X\uparrow\|f\|_X$;
\item[(iv)] $B\in\BB$ implies that $\mathbf{1}_B\in X$, where $\BB$ is as in \eqref{Eqball}.
\end{itemize}

Moreover, a ball quasi-Banach function space $X$ is called a
\emph{ball Banach function space} if the norm of $X$
satisfies the triangle inequality: for any $f,\ g\in X$,
\begin{equation}\label{eq22x}
\|f+g\|_X\le \|f\|_X+\|g\|_X
\end{equation}
and, for any $B\in \BB$, there exists a positive constant $C_{(B)}$, depending on $B$, such that, for any $f\in X$,
\begin{equation}\label{eq2.3}
\int_B|f(x)|\,dx\le C_{(B)}\|f\|_X.
\end{equation}
\end{definition}

Recall that a quasi-Banach space $X\subset\mathscr M(\rn)$ is called a \emph{quasi-Banach function space} if it
is a ball quasi-Banach function space and it satisfies Definition \ref{Debqfs}(iv) with ball
replaced by any measurable set of finite measure.

It is easy to see that every Banach function space is a ball Banach function space.
As was mentioned in \cite[p.\,9]{SHYY}, the family of ball Banach function spaces includes Morrey type spaces,
which are not necessarily Banach function spaces.

For any ball Banach function space $X$, the \emph{associate space} (\emph{K\"othe dual}) $X'$ is defined by setting
\begin{equation}\label{asso}
X':=\lf\{f\in\mathscr M(\rn):\ \|f\|_{X'}:=\sup\lf\{\|fg\|_{L^1(\rn)}:\ g\in X,\ \|g\|_X=1\r\}<\infty\r\},
\end{equation}
where $\|\cdot\|_{X'}$ is called the \emph{associate norm} of $\|\cdot\|_X$
(see, for instance, \cite[Chapter 1, Definitions 2.1 and 2.3]{BS}).

\begin{remark}
\begin{itemize}
\item[(i)] By \cite[Proposition 2.3]{SHYY}, we know that, if $X$ is a ball Banach function space,
then its associate space $X'$ is also a ball Banach function space.
\item[(ii)] A ball quasi-Banach function space $Y\subset\mathscr M(\rn)$ is called a \emph{quasi-Banach
function space} (see, for instance, \cite[Definition 2.4.7]{SHYY}) if,
for any measurable set $E\subset\rn$ with finite measure, $\mathbf{1}_E\in Y$.
\end{itemize}
\end{remark}

The following H\"older inequality is a direct corollary of both Definition \ref{Debqfs}(i)
and \eqref{asso} (see also \cite[Theorem 2.4]{BS}); we omit the details.

\begin{lemma}[{\bf the H\"older inequality}]\label{LeHolder}
Let $X$ be a ball Banach function space with the associate space $X'$.
If $f\in X$ and $g\in X'$, then $fg$ is integrable and
\begin{equation}\label{12}
\int_\rn|f(x)g(x)|\,dx\le \|f\|_X\|g\|_{X'}.
\end{equation}
\end{lemma}

Similarly to \cite[Theorem 2.7]{BS}, we have the following conclusion, whose
proof is a slight modification of the one of \cite[Theorem 2.7]{BS}.

\begin{lemma}[{\bf G. G. Lorentz, W. A. J. Luxembourg}]\label{Lesdual}
Every ball Banach function space $X$ coincides with its second associate space $X''$.
In other words, a function $f$ belongs to $X$ if and only if it belongs to $X''$ and,
in that case,
$$
\|f\|_X=\|f\|_{X''}.
$$
\end{lemma}

\begin{proof} Let $X$ be a ball Banach function space.
From this and \cite[Proposition 2.3]{SHYY}, we deduce that
$X'$ and $X''$ are both ball Banach function spaces.
Using this and Lemma \ref{LeHolder} and repeating the proof of \cite[Theorem 2.7]{BS} via
replacing Definition \ref{bn}(iv) by Definition \ref{Debqfs}(iv), we then complete
the proof of Lemma \ref{Lesdual}.
\end{proof}

We still need to recall the notions of the convexity and the concavity of ball quasi-Banach
function spaces,
which come from, for instance, \cite[Definition 1.d.3]{LT}.
\begin{definition}\label{Debf}
Let $X$ be a ball quasi-Banach function space and $p\in(0,\infty)$.
\begin{itemize}
\item[(i)] The $p$-\emph{convexification} $X^p$ of $X$ is defined by setting $X^p:=\{f\in\mathscr M(\rn):\ |f|^p\in X\}$
equipped with the quasi-norm $\|f\|_{X^p}:=\||f|^p\|_X^{1/p}$.
\item[(ii)] The space $X$ is said to be $p$-\emph{concave} if there exists a positive constant $C$ such that,
for any sequence $\{f_j\}_{j\in\nn}$ of $X^{1/p}$,
$$
\sum_{j\in\nn}\|f_{j}\|_{X^{1/p}}\le C\lf\|\sum_{j\in\nn}|f_j|\r\|_{X^{1/p}}.
$$
Particularly, $X$ is said to be \emph{strictly $p$-concave} when $C=1$.
\end{itemize}
\end{definition}

Now we introduce the notion of weak ball quasi-Banach function spaces as follows.

\begin{definition}\label{2.8}
Let $X$ be a ball quasi-Banach function space. The \emph{weak ball quasi-Banach function
space} $WX$ is defined to be the set of all measurable functions $f$ satisfying
\begin{equation}\label{Eqde1}
\|f\|_{WX}:=\sup_{\alpha\in(0,\infty)}
\lf\{\alpha\lf\|\mathbf{1}_{\{x\in\rn:\ |f(x)|>\alpha\}}\r\|_X\r\}<\infty.
\end{equation}
\end{definition}

\begin{remark}\label{Rews}
\begin{itemize}
\item[(i)] Let $X$ be a ball quasi-Banach function space. For any $f\in X$ and $\alpha\in(0,\infty)$, we have
$\mathbf{1}_{\{x\in\rn:\ |f(x)|>\alpha\}}(x)\leq|f(x)|/\alpha$ for any $x\in\rn$, which, together with Definition \ref{Debqfs}(ii),
further implies that $\sup_{\alpha\in(0,\infty)}\lf\{\alpha\|\mathbf{1}_{\{x\in\rn:
\ |f(x)|>\alpha\}}\|_X\r\}\leq\|f\|_X$. This shows
that $X\subset WX$.

\item[(ii)] Let $f,\ g\in WX$ with $|f|\le|g|$. By Definition \ref{Debqfs}(ii), we conclude that $\|f\|_{WX}\le\|g\|_{WX}$.
\end{itemize}
\end{remark}

\begin{lemma}\label{Leqn}
Let $X$ be a ball quasi-Banach function space. Then $\|\cdot\|_{WX}$ is a quasi-norm on $WX$, namely,
\begin{itemize}
\item[\rm(i)] $\|f\|_{WX}=0$ if and only if $f=0$ almost everywhere;
\item[\rm(ii)] for any $\lambda\in\cc$ and
$f\in WX,\ \|\lambda f\|_{WX}=|\lambda|\|f\|_{WX}$;
\item[\rm(iii)] for any $f,\ g\in WX$, there exists a positive constant $C$
such that $\|f+g\|_{WX}\le C[\|f\|_{WX}+\|g\|_{WX}]$.
Moreover, if $p\in(0,\infty)$ and $X^{1/p}$
is a ball Banach function space, then
$$
\|f+g\|_{WX}^{1/p}\le 2^{\max\{1/p,1\}}\lf[\|f\|_{WX}^{1/p}+\|g\|_{WX}^{1/p}\r].
$$
\end{itemize}
\end{lemma}

\begin{proof}
It is easy to show (i) and (ii) and the details are omitted. We now show (iii).
We first assume that $X^{1/p}$ is a ball Banach function space for some given $p\in(0,\infty)$.
Then, for any $f,\ g\in WX$ and $\alpha\in(0,\infty)$, by Definition \ref{Debf}(i), \eqref{eq22x}
with $X$ replaced by $X^{1/p}$
and the well-known inequality that $(a+b)^{1/p}\le2^{\max\{1/p-1,0\}}(a^{1/p}+b^{1/p})$ for any $a,\ b\in(0,\infty)$, we have
\begin{align*}
\|f+g\|_{WX}&\leq\sup_{\alpha\in(0,\infty)}\lf\{
\alpha\lf\|\mathbf{1}_{\{x\in\rn:\ |f(x)|+|g(x)|>\alpha\}}\r\|_X\r\}
=\sup_{\alpha\in(0,\infty)}\lf\{
\alpha\lf\|\mathbf{1}_{\{x\in\rn:\ |f(x)|+|g(x)|>\alpha\}}\r\|_{X^{1/p}}^{1/p}\r\}\\
&\le\sup_{\alpha\in(0,\infty)}\lf\{\alpha\lf[\lf\|\mathbf{1}_{\{x\in\rn:\ |f(x)|>\alpha/2\}}\r\|_{X^{1/p}}
+\lf\|\mathbf{1}_{\{x\in\rn:\ |g(x)|>\alpha/2\}}\r\|_{X^{1/p}}\r]^{1/p}\r\}\\
&\le2^{\max\{1/p-1,0\}}\sup_{\alpha\in(0,\infty)}
\lf\{\alpha\lf[\lf\|\mathbf{1}_{\{x\in\rn:\ |f(x)|>\alpha/2\}}
\r\|_{X^{1/p}}^{1/p}+\lf\|\mathbf{1}_{\{x\in\rn:\ |g(x)|>\alpha/2\}}\r\|_{X^{1/p}}^{1/p}\r]\r\}\\
&\leq2^{\max\{1/p,1\}}\lf[\sup_{\alpha\in(0,\infty)}\lf\{\alpha\lf\|\mathbf{1}_{\{x\in\rn:\ |f(x)|>\alpha\}}
\r\|_{X^{1/p}}^{1/p}\r\}+\sup_{\alpha\in(0,\infty)}\lf\{\alpha\lf\|\mathbf{1}_{\{x\in\rn:\ |g(x)|>\alpha\}}\r\|_{X^{1/p}}^{1/p}\r\}\r]\\
&=2^{\max\{1/p,1\}}\lf[\|f\|_{WX}+\|g\|_{WX}\r].
\end{align*}
For the ball quasi-Banach function space, the same procedure as above leads us to
the desired estimate with the positive constant
$C$ depending on the positive constant appearing in the quasi-triangular inequality of
the quasi-norm $\|\cdot\|_X$. This finishes the proof of Lemma \ref{Leqn}.
\end{proof}

\begin{remark}\label{Re213}
Let $X$ be a ball quasi-Banach function space. Then, by the Aoki--Rolewicz theorem
(see, for instance, \cite[Exercise 1.4.6]{G1}), one finds a positive constant $\nu\in(0,1)$ such that,
for any $N\in\nn$ and $\{f_j\}_{j=1}^N\subset\mathscr M(\rn)$,
\begin{align*}
\lf\|\sum_{j=1}^N|f_j|\r\|_{WX(\rn)}^\nu\le4\sum_{j=1}^N\lf[\lf\||f_j|\r\|_{WX}\r]^\nu.
\end{align*}
\end{remark}

\begin{lemma}\label{LeFatou}
Let $X$ be a ball quasi-Banach function space and $\{f_m\}_{m\in\nn}\subset WX$. If $f_m\rightarrow f$ as $m\rightarrow\infty$
almost everywhere in $\rn$ and if $\liminf_{m\rightarrow\infty}\|f_m\|_{WX}<\infty$, then
$f\in WX$ and
$$
\lf\|f\r\|_{WX}\le\liminf_{m\rightarrow\infty}\|f_m\|_{WX}.
$$
\end{lemma}

\begin{proof}
For any $k\in\nn$, letting $h_k:=\inf_{m\geq k}|f_m|$, then $0\le h_k\uparrow|f|$,
$k\rightarrow\infty$, almost everywhere in $\rn$ and hence, for any $\alpha\in(0,\infty)$,
$\mathbf{1}_{\{x\in\rn:\ |h_k(x)|>\alpha\}}\uparrow\mathbf{1}_{\{x\in\rn:\ |f(x)|>\alpha\}}$.
Moreover, by Definition \ref{Debqfs}(iii) and the definition of $h_k$, for any
$\alpha\in(0,\infty)$, we have
$$
\lf\|\mathbf{1}_{\{x\in\rn:\ |f(x)|>\alpha\}}\r\|_X
=\lim_{k\rightarrow\infty}\lf\|\mathbf{1}_{\{x\in\rn:\ |h_k(x)|>\alpha\}}\r\|_X
\le\liminf_{m\rightarrow\infty}\lf\|\mathbf{1}_{\{x\in\rn:\ |f_m(x)|>\alpha\}}\r\|_X.
$$
This further implies that, for any $\alpha\in(0,\infty)$,
\begin{align*}
\alpha\lf\|\mathbf{1}_{\{x\in\rn:\ |f(x)|>\alpha\}}\r\|_X&\le
\alpha\liminf_{m\rightarrow\infty}\lf\|\mathbf{1}_{\{x\in\rn:\ |f_m(x)|>\alpha\}}\r\|_X\\
&\le\liminf_{m\rightarrow\infty}\sup_{\alpha\in(0,\infty)}\lf\{\alpha\lf\|\mathbf{1}_{\{x\in\rn:\ |f_m(x)|>\alpha\}}\r\|_X\r\}
=\liminf_{m\rightarrow\infty}\|f_m\|_{WX},
\end{align*}
which completes the proof of Lemma \ref{LeFatou}.
\end{proof}

From the definition
of $WX$, Remark \ref{Re213}, Lemmas \ref{Leqn} and \ref{LeFatou}, it is easy to deduce
the following lemma and we omit the details.

\begin{lemma}\label{bal}
Let $X$ be a ball quasi-Banach function space.
Then the
space $WX$ is also a ball quasi-Banach function space.
\end{lemma}

\begin{remark}\label{Rews-1} Let $X$ be a ball quasi-Banach function space.
By Lemma \ref{bal}, we know that $WX$ is also a ball quasi-Banach function space.
For any given $s\in(0,\infty)$, it is easy to show that $X^s$ is also
a ball quasi-Banach function space. Thus, $(WX)^s$ and $W(X^s)$ make sense and coincide
with equal quasi-norms. Indeed, for any $f\in (WX)^s$, by Definitions \ref{Debf}(i) and \ref{2.8},
we have
$$
\|f\|_{(WX)^s}^s=\||f|^s\|_{WX}=\|f\|_{W(X^s)}^s.
$$
\end{remark}

\subsection{Assumptions on the Hardy--Littlewood maximal operator}

Denote by the \emph{symbol $L_{\loc}^1(\rn)$} the set of all locally integrable functions on $\rn$.
The \emph{Hardy--Littlewood maximal operator} $\cm$
is defined by setting, for any $f\in L_{\loc}^1(\rn)$ and $x\in\rn$,
\begin{equation}\label{mm}
\cm(f)(x):=\sup_{B\ni x}\frac1{|B|}\int_B|f(y)|\,dy,
\end{equation}
where the supremum is taken over all balls $B\in\BB$ containing $x$.

For any $\theta\in(0,\infty)$, the \emph{powered Hardy--Littlewood
maximal operator} $\cm^{(\theta)}$ is defined by setting,
for any $f\in L_{\loc}^1(\rn)$ and $x\in\rn$,
\begin{equation}\label{mmx}
\cm^{(\theta)}(f)(x):=\lf\{\cm\lf(|f|^\theta\r)(x)\r\}^{1/\theta}.
\end{equation}

To establish atomic characterizations of weak Hardy spaces associated with ball quasi-Banach
function spaces $X$,
the approach used in this article heavily depends on the following assumptions
on the boundedness of the Hardy--Littlewood maximal function on $X^{1/p}$,
which is stronger than \cite[(2.8)]{SHYY}.
\begin{assumption}\label{a2.15}
Let $X$ be a ball quasi-Banach function space and there exists a $p_-\in(0,\infty)$ such that,
for any given $p\in(0,p_-)$ and $s\in(1,\infty)$, there exists a positive constant $C$ such that,
for any $\{f_j\}_{j=1}^\infty\subset\mathscr M(\rn)$,
\begin{equation}\label{m215}
\lf\|\lf\{\sum_{j\in\nn}\lf[\cm(f_j)\r]^s\r\}^{1/s}\r\|_{X^{1/p}}
\le C\lf\|\lf\{\sum_{j\in\nn}|f_j|^s\r\}^{1/s}\r\|_{X^{1/p}}.
\end{equation}
\end{assumption}

\begin{remark}\label{Refs}
\begin{itemize}

\item[(i)] Let $X$ and $p_-$ be the same as in Assumption \ref{a2.15}. Let
\begin{equation}\label{Eqpll}
\underline{p}:=\min\{p_-,\ 1\}.
\end{equation}
Then, for any given $r\in(0,\underline{p})$ and for any sequence $\{B_j\}_{j\in\nn}\subset\BB$
and $\beta\in[1,\infty)$, by Definition \ref{Debqfs}(ii),
the fact that $\mathbf{1}_{\beta B_j}\le[\beta^n\cm(\mathbf{1}_{B_j})]^{1/r}$
almost everywhere on $\rn$ for any $j\in\nn$, Definition \ref{Debf}(i) and
Assumption \ref{a2.15}, we have
\begin{align}\label{EqHLMS}
\lf\|\sum_{j\in\nn}\mathbf{1}_{\beta B_j}\r\|_{X}&\le\lf\|\sum_{j\in\nn}
\lf[\beta^n\cm(\mathbf{1}_{B_j})\r]^\frac1r\r\|_{X}
=\beta^\frac{n}{r}\lf\|\lf\{\sum_{j\in\nn}\lf[\cm(\mathbf{1}_{B_j})\r]^\frac1r\r\}^r\r\|_{X^{1/r}}^{1/r}\\\noz
&\le C\beta^\frac{n}{r}\lf\|\lf[\sum_{j\in\nn}\mathbf{1}_{B_j}\r]^r\r\|_{X^{1/r}}^{1/r}
=C\beta^\frac{n}{r}\lf\|\sum_{j\in\nn}\mathbf{1}_{B_j}\r\|_{X},
\end{align}
where the positive constant $C$ is independent of $\{B_j\}_{j\in\nn}$ and $\beta$.
\item[(ii)] In Assumption \ref{a2.15}, let
$X:=L^{\widetilde{p}}(\rn)$ with any given $\widetilde{p}\in(0,\infty)$. In this case,
$p_-=\widetilde{p}$ and the inequality \eqref{m215} becomes the well-known
Fefferman--Stein vector-valued maximal inequality,
which was originally established by Fefferman and Stein
in \cite[Theorem 1(a)]{FS}.
\end{itemize}
\end{remark}

\begin{assumption}\label{xinm}
Let $X$ be a ball quasi-Banach function space.
Assume that there exists $r\in(0,\infty)$
such that $\cm$ in \eqref{mm} is bounded on $(WX)^{1/r}$.
\end{assumption}

\subsection{Weak Hardy type spaces }

In what follows, we denote by $\cs(\rn)$ the \emph{space of all Schwartz functions},
equipped with the well-known topology determined by a countable family of
seminorms, and by $\cs'(\rn)$ its \emph{topological dual space},
equipped with the weak-$\ast$ topology. For any $N\in\nn$, let
\begin{equation}\label{EqcfN}
\cf_N(\rn):=\lf\{\varphi\in\cs(\rn):\sum_{\beta\in\zz_+^n,|\beta|\le N}
\sup_{x\in\rn}\lf[\lf(1+|x|\r)^{N+n}\lf|\partial_x^\beta\varphi(x)\r|\r]\le1\r\},
\end{equation}
here and hereafter, for any $\beta:=(\beta_1,\ldots,\beta_n)\in\zz_+^n$
and $x\in\rn$,
$|\beta|:=\beta_1+\cdots+\beta_n$ and
$\partial_x^\beta:=(\frac{\partial}{\partial x_1})^{\beta_1}
\cdots(\frac{\partial}{\partial x_n})^{\beta_n}$.
For any given $f\in\cs'(\rn)$, the \emph{radial grand maximal function} $M_N^0(f)$
of $f$ is defined by setting, for any $x\in\rn$,
\begin{equation}\label{EqMN0}
M_N^0(f)(x):=\sup\lf\{|f\ast\varphi_t(x)|:\ t\in(0,\infty)\ \text{and}\ \varphi\in\cf_N(\rn)\r\},
\end{equation}
where, for any $t\in(0,\infty)$ and $\xi\in\mathbb R^n,\varphi_t(\xi):=t^{-n}\varphi(\xi/t)$.

\begin{definition}\label{DewSH}
Let $X$ be a ball quasi-Banach function space.
Then the \emph{weak Hardy-type space}
$WH_X(\rn)$ associated with $X$ is defined by setting
$$
WH_X(\rn):=\lf\{f\in\cs'(\rn):\ \|f\|_{WH_X(\rn)}:=\lf\|M_N^0(f)\r\|_{WX}<\infty\r\},
$$
where $M_N^0(f)$ is as in \eqref{EqMN0} with $N\in\nn$ sufficiently large.
\end{definition}

\begin{remark}
\begin{itemize}
\item[(i)] When $X:=L^p(\rn)$ with $p\in(0,1]$, the Hardy-type space $WH_X(\rn)$
coincides with the classical weak Hardy space $WH^p(\rn)$ (see, for instance, \cite[p.\,114]{L}).
\item[(ii)] By Theorem \ref{Thmc}(ii) below, we know that, if
the Hardy-Littlewood maximal operator $\cm$ in \eqref{mm} is bounded on $(WX)^{1/r}$ and
$N\in[\lfloor\frac nr\rfloor+1,\infty)\cap\nn$,
then $WH_X(\rn)$ in Definition \ref{DewSH} is independent of
the choice of $N$.
\end{itemize}
\end{remark}

\section{Maximal function characterizations and
relations between $WX$ and $WH_X(\rn)$\label{s3}}

The aim of this section is to characterize $WH_X(\rn)$ via radial or non-tangential maximal functions.
We begin with the following notions of the radial functions and the non-tangential maximal functions.
\begin{definition}\label{Dem}
Let $\psi\in\cs(\rn)$, $a,\ b\in(0,\infty)$, $N\in\nn$ and $f\in\cs'(\rn)$.
\begin{itemize}
\item[(i)] The \emph{radial maximal function $M(f,\psi)$} of $f$ associated to $\psi$
is defined by setting, for any $x\in\rn$,
$$
M(f,\psi)(x):=\sup_{t\in(0,\infty)}|f\ast\psi_t(x)|.
$$
\item[(ii)] The \emph{non-tangential maximal function $M_{a}^\ast(f,\psi)$} of $f$ associated to
$\psi$ is defined by setting, for any $x\in\rn$,
$$
M_{a}^\ast(f,\psi)(x):=\sup_{t\in(0,\infty),|y-x|<at}|f\ast\psi_t(y)|.
$$
\item[(iii)] The \emph{maximal function of Peetre type, $M_{b}^{\ast\ast}(f,\psi)$,}
is defined by setting, for any $x\in\rn$,
$$
M_{b}^{\ast\ast}(f,\psi)(x):=\sup_{(y,t)\in\rr_+^{n+1}}\frac{|(\psi_t\ast f)(x-y)|}{(1+t^{-1}|y|)^b}.
$$
\item[(iv)] The \emph{non-tangential grand maximal function $M_{N,a}(f)$} of
$f$ is defined by setting, for any $x\in\rn,$
$$
M_{N,a}(f)(x):=\sup_{\varphi\in\cf_N(\rn)}\sup_{t\in(0,\infty),|y-x|<at}|f\ast\varphi_t(y)|.
$$
\item[(v)] The \emph{grand maximal function of Peetre type, $M_{N,b}^{\ast\ast}(f)$,}
is defined by setting, for any $x\in\rn$,
$$
M_{N,b}^{\ast\ast}(f)(x):=\sup_{\varphi\in\cf_N(\rn)}
\lf\{\sup_{(y,t)\in\rr_+^{n+1}}\frac{|\varphi_t\ast f(x-y)|}{(1+t^{-1}|y|)^b}\r\},
$$
where $\cf_N(\rn)$ is as in \eqref{EqcfN}. When $a=1$, we simply denote $M_{N,a}(f)$ by $M_N(f)$.
\end{itemize}
\end{definition}

The following theorem is the main result of this section,
which presents the maximal function characterizations of the space $WH_X(\rn)$.

\begin{theorem}\label{Thmc}
Let $a,\ b\in(0,\infty)$ and $X$ be a ball quasi-Banach function space.
Let $\psi\in\cs(\rn)$ satisfy $\int_\rn\psi(x)\ dx\neq0$.
\begin{itemize}
\item[{\rm(i)}] Let $N\geq\lfloor b+1\rfloor$ be an integer. Then, for any $f\in\cs'(\rn)$,
\begin{equation}\label{Eqmc1}
\|M(f,\psi)\|_{WX}\lesssim\|M_a^\ast(f,\psi)\|_{WX}\lesssim\|M_b^{\ast\ast}(f,\psi)\|_{WX},
\end{equation}
\begin{equation}\label{Eqmc2}
\|M(f,\psi)\|_{WX}\lesssim\|M_N(f)\|_{WX}\lesssim\|M_{\lfloor b+1\rfloor}(f)\|_{WX}\lesssim\|M_b^{\ast\ast}(f,\psi)\|_{WX},
\end{equation}
\begin{equation}\label{Eqmc3}
\|M_b^{\ast\ast}(f,\psi)\|_{WX}\sim\|M_{b,N}^{\ast\ast}(f)\|_{WX}
\end{equation}
and
\begin{equation}\label{Eqmc30}
\|M_N^0(f)\|_{WX}\sim\|M_{N}(f)\|_{WX},
\end{equation}
where the implicit positive constants are independent of $f$.
\item[{\rm(ii)}] Let $r\in(0,\infty)$. Assume that $b\in(n/r,\infty)$ and
the Hardy-Littlewood maximal operator $\cm$ in \eqref{mm} is bounded on $(WX)^{1/r}$.
Then, for any $f\in\cs'(\rn)$,
\begin{equation}\label{Eqmc4}
\|M_b^{\ast\ast}(f,\psi)\|_{WX}\lesssim\|M(f,\psi)\|_{WX},
\end{equation}
where the implicit positive constant is independent of $f$. In particular, when $N\geq \lfloor b+1\rfloor$,
if one of the quantities
$$
\|M_N^0(f)\|_{WX},\quad\|M(f,\psi)\|_{WX},\quad \|M_a^\ast(f,\psi)\|_{WX},\quad \|M_N(f)\|_{WX},
$$
$$
\|M_b^{\ast\ast}(f,\psi)\|_{WX}\quad\text{and}\quad\|M_{b,N}^{\ast\ast}(f)\|_{WX}
$$
is finite, then the others are also finite and mutually equivalent with the positive
equivalence constants
independent of $f$.
\end{itemize}
\end{theorem}

\begin{proof}
The proof of this theorem is similar to that of \cite[Theorem 3.1]{SHYY}.
For the convenience of the reader, we present some details.

Let $f\in\cs'(\rn)$. We first prove (i).
From (i), (ii) and (iii) of Definition \ref{Dem}, it follows that, for any $x\in\rn$,
$$
M(f,\psi)(x)\le M_a^\ast(f,\psi)(x)\lesssim M_b^{\ast\ast}(f,\psi)(x),
$$
which, together with Remark \ref{Rews}(ii), implies \eqref{Eqmc1}.

Moreover, by (i) and (iv) of Definition \ref{Dem} again, we have, for any $x\in\rn$,
\begin{equation}\label{Eqmc5}
M(f,\psi)(x)\lesssim M_N(f)(x)\lesssim M_{\lfloor b+1\rfloor}(f)(x).
\end{equation}
In addition, from the proof of \cite[Theorem 2.1.4(d)]{G20142},
we deduce that, for any $x\in\rn$,
$$M_{\lfloor b+1\rfloor}(f)(x)\lesssim M_b^{\ast\ast}(f,\psi)(x),$$
which, together with \eqref{Eqmc5} and Remark \ref{Rews}(ii), implies \eqref{Eqmc2}.

It is easy to see that, for any $x\in\rn$, $M_b^{\ast\ast}(f,\psi)(x)\lesssim M_{b,N}^{\ast\ast}(f)(x)$,
which, combined with \cite[Lemma 2.13]{SHYY}, implies \eqref{Eqmc3}. By \cite[Remark 3.6(i)]{YYYZ},
we know that there exists a positive constant $C$ such that, for any $x\in\rn$, $C^{-1}M_N(f)(x)\leq M_N^0(f)(x)\leq CM_N(f)(x)$,
which, together with Remark \ref{Rews}(ii),
implies that \eqref{Eqmc30} holds true. This finishes the proof of (i).

Now we prove (ii). It was proved in \cite[P.\,35]{SHYY} that, if $r\in(0,\infty)$ and $br>n$, then, for any $x\in\rn$,
$$
M_b^{\ast\ast}(f,\psi)(x)\lesssim\cm^{(r)}\lf[\sup_{t\in(0,\infty)}|\psi_t\ast f|\r](x)\sim\cm^{(r)}(M(f,\psi))(x),
$$
which, combined with Remark \ref{Rews}(ii) and
the assumption that $\cm$ is bounded on $WX^{1/r}$, further implies that
$$
\lf\|M_b^{\ast\ast}(f,\psi)\r\|_{WX}\lesssim
\lf\|\cm^{(r)}(M(f,\psi))\r\|_{WX}\lesssim\lf\|M(f,\psi)\r\|_{WX}.
$$
Thus, \eqref{Eqmc4} holds true. This finishes the proof Theorem \ref{Thmc}.
\end{proof}

For any $t\in(0,\infty)$, the \emph{Poisson kernel} $P_t$ is defined by setting, for any $x\in\rn$,
$$
P_t(x):=\frac{\Gamma([n+1]/2)}{\pi^{(n+1)/2}}\frac{t}{(t^2+|x|^2)^{(n+1)/2}},
$$
where $\Gamma$ denotes the \emph{Gamma function}.

Recall that $f\in\cs'(\rn)$ is called a \emph{bounded distribution} if,
for any $\varphi\in\cs(\rn),$ $f\ast\varphi\in L^\infty(\rn)$.
For any given bounded distribution $f$, its \emph{non-tangential maximal function $\cn(f)$},
with respect to Poisson kernels $\{P_t\}_{t\in(0,\infty)},$ is defined by setting,
for any $x\in\rn,$
$$
\cn(f)(x):=\sup_{t\in(0,\infty),|y-x|<t}|f\ast P_t(y)|.
$$

\begin{theorem}\label{Thmp}
Let $X$ be a ball quasi-Banach function space satisfying Assumption \ref{xinm}.
Assume that there exists a positive constant $C_0$ such that
\begin{equation}\label{Eqinf}
\inf_{x\in\rn}\|\mathbf{1}_{B(x,1)}\|_{WX}\geq C_0.
\end{equation}
Then $f\in WH_X(\rn)$ if and only if $f$ is a bounded distribution and $\cn(f)\in WX$.
Moreover, for any $f\in WH_X(\rn)$, $\|f\|_{WH_X(\rn)}\sim\|\cn(f)\|_{WX}$ with the
positive equivalence constants independent of $f$.
\end{theorem}

\begin{proof} Assume that $f\in WH_X(\rn)$. By Assumption \ref{xinm} and Theorem \ref{Thmc}(ii),
we know that, for any given $N\in[\lfloor\frac nr\rfloor+1,\infty)\cap\nn$,
$$\|M_N(f)\|_{WX(\rn)}\sim\|f\|_{WH_X(\rn)}.$$
It is easy to see that,
for any fixed $\varphi\in\cs(\rn)$, there exists a positive constant
$C_{(\varphi)}$ such that $C_{(\varphi)}\varphi\in\cf_N(\rn)$
with $\cf_N(\rn)$ as in \eqref{EqcfN}. Therefore, for any $x\in\rn$,
$M_1^*(f,C_{(\varphi)}\varphi)(x)\lesssim M_N(f)(x)$,
which, together with Definition \ref{Debqfs}(ii), Remark \ref{Rews}(ii), \eqref{Eqinf}
and Theorem \ref{Thmc}(ii), further implies that, for any $x\in\rn$,
\begin{align}\label{Eqdphi}
C_{(\varphi)}|(\varphi\ast f)(x)|&\le\inf_{|y-x|<1}\ M_1^*(f,C_{(\varphi)}\varphi)(y)=
\frac{\|\mathbf{1}_{B(x,1)}\inf_{|y-x|<1}M_1^*(f,C_{(\varphi)}\varphi)(y)\|_{WX}}
{\|\mathbf{1}_{B(x,1)}\|_{WX}}\\\noz
&\le\frac{\|\mathbf{1}_{B(x,1)}M_1^*(f,C_{(\varphi)}\varphi)\|_{WX}}{\|\mathbf{1}_{B(x,1)}\|_{WX}}
\lesssim\frac{\|M_N(f)\|_{WX}}{C_{0}}<\infty.
\end{align}
This means that $f$ is a bounded distribution. Next, we show that $\cn(f)\in WX$.
From the proof of \cite[p.\,72]{G20142},
we deduce that, for any $N\in\nn$ and $x\in\rn$, $\cn(f)(x)\le C_{(n,N)}M_N(f)(x)$,
which implies that $\cn(f)\in WX$ and
$\|\cn(f)\|_{WX}\lesssim\|M_N(f)\|_{WX}\sim\|f\|_{WH_X(\rn)}$.

Now, assume that
$f$ is a bounded distribution and $\cn(f)\in WX$. Then, by \cite[p.\,99]{S} or \cite[p.\,35]{SHYY},
we know that there exists $\psi_0\in\cs(\rn)$ with $\int_{\rn}\psi_0(x)\,dx=1$ such that,
for any $x\in\rn,$ $M(f,\psi_0)(x)\lesssim\cn(f)(x)$, which, combined with $\cn(f)\in WX$, Remark \ref{Rews}(ii),
Assumption \ref{xinm} and
Theorem \ref{Thmc}(ii), implies $f\in WH_X(\rn)$ and
$\|f\|_{WH_X(\rn)}\sim\|M(f,\psi_0)\|_{WX}\lesssim\|\cn(f)\|_{WX}$.
This finishes the proof of Theorem \ref{Thmp}.
\end{proof}

Now, we discuss the relation between the spaces $WX$ and $WH_X(\rn)$.

\begin{theorem}\label{dayu2}
Let $X$ be a ball quasi-Banach function space and $\cm$ in \eqref{mm}
bounded on $(WX)^{1/r}$ for some $r\in(1,\infty)$. Then
\begin{enumerate}
\item[{\rm(i)}] $WX\ \hookrightarrow\ \mathcal{S}'(\rn)$.

\item[{\rm(ii)}] If $f\in WX$, then $f\in WH_X(\rn)$ and there exists a positive constant
$C$, independent of $f$, such that $\|f\|_{WH_X(\rn)}\le C\|f\|_{WX}$.

\item[{\rm(iii)}] If $f\in WH_X(\rn)$, then there exists a locally
integrable function $g\in WX$ such that $g$ represents $f$, which means that $f=g$ in $\mathcal{S}'(\rn)$,
$\|f\|_{WH_X(\rn)}=\|g\|_{WH_X(\rn)}$ and there exists a positive constant
$C$, independent of $f$, such that $\|g\|_{WX}\le C\|f\|_{WH_X(\rn)}$.
\end{enumerate}
\end{theorem}

\begin{proof}
Observe that
$$
\ell_{WX}:=\sup\{r\in(0,\infty):\ \mathcal{M}\
\textup{is}\ \textup{bounded}\ \textup{on}\ (WX)^{1/r}\}>1.
$$
Moreover, by Lemma \ref{bal}, we know that  the
space $WX$ is a ball quasi-Banach function space.
Thus, all assumptions of \cite[Theorem 3.4]{SHYY} with $X$ and $H_X(\rn)$
replaced, respectively,
by $WX$ and $WH_X(\rn)$ are satisfied, from which we deduce all the desired conclusions
of Theorem \ref{dayu2}.
This finishes the proof of Theorem \ref{dayu2}.
\end{proof}

\section{Atomic characterizations\label{s4}}

In this section, we establish the atomic characterization of $WH_X(\rn)$.
Now we introduce the notion of atoms associated with $X$, which origins from \cite[Definition 3.5]{SHYY}.
\begin{definition}\label{Deatom}
Let $X$ be a ball quasi-Banach function space, $q\in(1,\infty]$ and $d\in\zz_+$.
Then a measurable function $a$ on $\rn$ is called an $(X,\,q,\,d)$-\emph{atom}
if there exists a ball $B\in\BB$ such that
\begin{itemize}
\item[(i)] $\supp a:=\{x\in\rn:\ a(x)\neq0\}\subset B$;
\item[(ii)] $\|a\|_{L^q(\rn)}\le\frac{|B|^{1/q}}{\|\mathbf{1}_B\|_X}$;
\item[(iii)] $\int_{\rn}a(x)x^\alpha\,dx=0$ for any
$\alpha:=(\alpha_1,\ldots,\alpha_n)\in\zz_+^n$ with $|\alpha|\le d$,
here and hereafter, for any $x:=(x_1,\ldots,x_n)\in\rn$,
$x^\alpha:=x_1^{\alpha_1}\cdots x_n^{\alpha_n}$.
\end{itemize}
\end{definition}

Now we first formulate a decomposition theorem as follows.
\begin{theorem}\label{Thad}
Let $X$ be a ball quasi-Banach function space satisfying that, for some
given $r\in(0,1)$ and for any $\{f_j\}_{j\in\nn}\subset\mathscr M(\rn)$,
\begin{equation}\label{eq}
\lf\|\lf\{\sum_{j\in\nn}\lf[\cm(f_j)\r]^{1/r}\r\}^{r}\r\|_{X^{1/r}}
\le C\lf\|\lf\{\sum_{j\in\nn}|f_j|^{1/r}\r\}^{r}\r\|_{X^{1/r}},
\end{equation}
where the positive constant $C$ is independent of $\{f_j\}_{j\in\nn}$.
Assume that $X$ satisfy Assumption \ref{xinm} and there exist
$\vartheta_0\in(1,\infty)$ and $p\in(0,\infty)$ such that $X$ is $\vartheta_0$-concave and $\cm$ is bounded on
$X^{1/(\vartheta_0 p)}$.
Let $d\geq \lfloor n(1/p-1)\rfloor$ be a fixed nonnegative integer and $f\in WH_X(\rn)$.
Then there exists a sequence $\{a_{i,j}\}_{i\in\zz,j\in\nn}$ of
$(X,\,\infty,\,d)$-atoms supported, respectively, in balls $\{B_{i,j}\}_{i\in\zz,j\in\nn}$
satisfying that,
for any $i\in\zz$,
$\sum_{j\in\nn}\mathbf{1}_{cB_{i,j}}\le A$ with $c\in(0,1]$ and $A$ being
a positive constant independent of $f$ and $i$,
such that $f=\sum_{i\in\zz}\sum_{j\in\nn}\lambda_{i,j}a_{i,j}$ in $\cs'(\rn)$,
where $\lambda_{i,j}:=\widetilde A2^i\|\mathbf{1}_{B_{i,j}}\|_{X}$ for any $i\in\zz$ and $j\in\nn$,
with $\widetilde A$ being a positive constant independent of $i$, $j$ and $f$, and
$$
\sup_{i\in\zz}\lf\|\sum_{j\in\nn}
\frac{\lambda_{i,j}\mathbf{1}_{B_{i,j}}}{\|\mathbf{1}_{B_{i,j}}\|_{X}}\r\|_{X}\lesssim\|f\|_{WH_X(\rn)},
$$
where the implicit positive constant is independent of $f$.
\end{theorem}

Before showing Theorem \ref{Thad}, we recall some notions and establish some
necessary lemmas.
Recall that $f\in\cs'(\rn)$ is said to \emph{vanish weakly at infinity} if, for any $\phi\in\cs(\rn)$,
$f\ast\phi_t\rightarrow0$ in $\cs'(\rn)$ as $t\rightarrow\infty$ (see, for instance, \cite[p.\,50]{FoS}).
\begin{lemma}\label{Le64}
Let $X$ be a ball quasi-Banach function space. If $f\in WH_X(\rn)$,
then $f$ vanishes weakly at infinity.
\end{lemma}
\begin{proof}
Let $f\in WH_X(\rn)$. By \cite[Proposition 3.10]{Bo}, we know that, for any $\phi\in\cs(\rn)$,
$t\in(0,\infty)$, $x\in\rn$ and $y\in B(x,t)$, $|f\ast\phi_t(x)|\lesssim M_N(f)(y)\lesssim M_N^0(f)(y)$,
where $N\in\nn$. Thus, there exists a positive constant $C_{(N)}$, independent of $x,\ t$ and $f$, such that
\begin{equation}\label{eq451}
B(x,t)\subset\lf\{y\in\rn:\ M_N^0(f)(y)\geq C_{(N)}|f\ast\phi_t(x)|\r\}.
\end{equation}
On the other hand, by \cite[Lemma 2.14]{SHYY}, we find that $1\notin X$, which, together with the fact that $\mathbf{1}_{B(x,t)}\uparrow1$, Definition \ref{Debqfs}(iii)
and \eqref{Eqde1}, implies that, for any $x\in\rn$, $\|\mathbf{1}_{B(x,t)}\|_{WX}=\|\mathbf{1}_{B(x,t)}\|_{X}\rightarrow\infty$ as $t\rightarrow\infty$.
From this and \eqref{eq451}, it follows that, for any $x\in\rn$,
\begin{align*}
\lf|f\ast\phi_t(x)\r|&\le\inf_{y\in B(x,t)}M_N^0(f)(y)\le\frac{\|\mathbf{1}_{B(x,t)}M_N^0(f)\|_{WX}}{\|\mathbf{1}_{B(x,t)}\|_{WX}}
\le C_{(N)}\frac{\|M_N^0(f)\|_{WX}}{\|\mathbf{1}_{B(x,t)}\|_{WX}}
\rightarrow0\quad\text{as}\quad t\rightarrow\infty,
\end{align*}
which implies that $f$ vanishes weakly at infinity. This finishes the proof of Lemma \ref{Le64}.
\end{proof}

In what follows, the symbol $\vec 0_n$ denotes the \emph{origin} of $\rn$ and,
for any $\varphi\in\cs(\rn)$, $\widehat\varphi$ denotes its \emph{Fourier transform}
which is defined by setting, for any $\xi\in\rn$,
$$
\widehat\varphi(\xi):=\int_\rn e^{-2\pi ix\xi}\varphi(x)\,dx.
$$
We also use the \emph{symbol $C_c^\fz(\rn)$} to denote the set of all
infinitely differentiable functions with compact supports,
and the \emph{symbol $\epsilon\rightarrow0^+$} to denote
$\epsilon\in(0,\infty)$ and $\epsilon\rightarrow0$.

Combining Calder\'on \cite[Lemma 4.1]{C1975} and Folland and Stein \cite[Theorem 1.64]{FoS}
(see also \cite[p.\,219]{C1977} and \cite[Lemma 4.6]{YYYZ}), we immediately obtain
the following Calder\'on reproducing formula and we omit the details.

\begin{lemma}\label{Le47}
Let $\phi$ be a Schwartz function and, for any $x\in\rn\setminus\{\vec 0_n\}$,
there exists $t\in (0,\fz)$ such that $\widehat\phi(tx)\not=0$. Then there
exists a $\psi\in\cs(\rn)$ such that $\wh\psi\in C^\fz_c(\rn)$ with its support
away from $\vec 0_n$, $\wh\phi\wh\psi\ge 0$ and, for any $x\in\rn\setminus\{\vec 0_n\}$,
$$\int^\fz_0\wh\phi(tx)\wh\psi(tx)\,\frac {dt}t=1.$$
Moreover, for any $f\in\cs'(\rn)$, if $f$ vanishes  weakly at infinity,
then
$$
f=\int_0^\infty f\ast\phi_t\ast\psi_t\,\frac{dt}{t}\quad\text{in}\quad\cs'(\rn),
$$
namely,
$$
f=\lim_{\substack{\epsilon\rightarrow0^+\\ A\rightarrow\infty}}
\int_\epsilon^A f\ast\phi_t\ast\psi_t\,\frac{dt}{t}\quad\text{in}\quad\cs'(\rn).
$$
\end{lemma}

Let $X$ be a ball quasi-Banach function space. For any $q\in[1,\infty)$ and $d\in\zz_+$,
a locally integrable function $f$ on $\rn$ is
said to be in the \emph{Campanato-type space} $\cl_{q,X,d}(\rn)$ if
$$
\|f\|_{\cl_{q,X,d}(\rn)}:=\sup_{Q}\lf\{\frac{|Q|}
{\|\mathbf{1}_Q\|_X}\lf[\frac1{|Q|}\int_Q\lf|f(x)-P_Q^df(x)\r|^q\,dx\r]^\frac1q\r\}<\infty,
$$
where the supremum is taken over all cubes $Q$ on $\rn$ and $P_Q^d$ denotes the \emph{unique polynomial} $P\in\cp_d(\rn)$
such that, for any polynomial $R\in\cp_d(\rn)$, $\int_Q[f(x)-P(x)]R(x)\,dx=0$
(see \cite[Definition 6.1]{NS}), here and hereafter, the
\emph{symbol $\cp_d(\rn)$} denotes the set of all polynomials with order at most $d$.

The following lemma comes from \cite[p.\,83]{TW}.

\begin{lemma}\label{Le48}
Let $d\in\zz_+$. Then there exists
a positive constant $C$ such that, for any $g\in L_{\loc}^{1}(\rn)$ and cube $Q\subset\rn$,
$$
\sup_{x\in Q}\lf|P_Q^dg(x)\r|\le\frac{C}{|Q|}\int_Q|g(x)|\,dx.
$$
\end{lemma}

\begin{lemma}\label{Le49}
Let $q\in[1,\infty)$, $d\in\zz_+$ and $X$ be a ball quasi-Banach function space. Assume that there exists
$p\in(0,\infty)$ such that $\cm$ in \eqref{mm} is bounded on $X^{1/p}$.
If $p\in(\frac{n}{n+d+1},\infty)$ and $f\in\cs(\rn)$, then $f\in\cl_{q,X,d}(\rn)$.
\end{lemma}
\begin{proof}
We first claim that there exists a positive constant $C$ such that, for any two cubes $Q_1$ and $Q_2$ with $Q_1\subset Q_2$,
\begin{equation}\label{Eq44}
\frac{\|\mathbf{1}_{Q_2}\|_X}{\|\mathbf{1}_{Q_1}\|_X}\le C\lf[\frac{|Q_2|}{|Q_1|}\r]^{1/p}.
\end{equation}
Indeed, we have $\mathbf{1}_{Q_2}\lesssim[|Q_2|/|Q_1|]^{1/p}[\cm(\mathbf{1}_{Q_1})]^{1/p}$.
By this, Definition \ref{Debqfs}(ii) and the assumption that $\cm$ is bounded on $X^{1/p}$,
we know that there exists a positive constant $C$, independent of $Q_1$ and $Q_2$, such that
\begin{align*}
\lf\|\mathbf{1}_{Q_2}\r\|_{X}\le\lf[\frac{|Q_2|}{|Q_1|}\r]^{1/p}\lf\|[\cm(\mathbf{1}_{Q_1})]^{1/p}\r\|_X
\lesssim\lf[\frac{|Q_2|}{|Q_1|}\r]^{1/p}\lf\|\mathbf{1}_{Q_1}\r\|_{X^{1/p}}^{1/p}
\sim\lf[\frac{|Q_2|}{|Q_1|}\r]^{1/p}\lf\|\mathbf{1}_{Q_1}\r\|_{X}.
\end{align*}
That is, the above claim holds true.

For any $f\in\cs(\rn)$, $x\in\rn$ and cube $Q:=Q(x_0, r)\subset\rn$ with $(x_0, r)\in\rr^{n+1}_+$,
to prove this lemma, let
$$
p_Q(x):= \sum_{|\beta|\le d}\frac{\partial^\beta f(x_0)}{\beta!} (x-x_0)^\beta\in \cp_d(\rn).
$$
Then, from Lemma \ref{Le48} and the H\"older inequality, it follows that
\begin{align}\label{Eq45}
&\lf[\int_Q\lf|f (x)-P^d_Qf(x)\r|^q\,dx\r]^\frac1q\\
&\quad\le\lf[\int_Q\lf|f(x)-p_Q(x)\r|^q\,dx\r]^\frac1q\noz
+\lf[\int_Q\lf|P^d_Q(p_Q-f)(x)\r|^q\,dx\r]^\frac1q\\\noz
&\quad\lesssim\lf[\int_Q\lf|f(x)-p_Q(x)\r|^q\,dx\r]^\frac1q+
\lf\{|Q|\lf[\frac1{|Q|}\int_Q|p_Q(x)-f(x)|\,dx\r]^q\r\}^\frac1q\\
\noz&\quad\lesssim\lf[\int_Q\lf|f(x)-p_Q(x)\r|^q\,dx\r]^\frac1q.
\end{align}
Now, if $|x_0|+r\le1$, namely, $Q\subset Q(\vec0_n,\sqrt n)$, then, by \eqref{Eq45},
the Taylor remainder theorem
 and \eqref{Eq44}, we conclude that
\begin{align}\label{Eq46}
\frac{|Q|}{\|\mathbf{1}_Q\|_{X}}\lf[\frac1{|Q|}\int_Q\lf|f(x)-P_Q^df(x)\r|^q\,dx\r]^\frac1q
&\lesssim\frac{|Q|}{\|\mathbf{1}_Q\|_{X}}\lf[\frac1{|Q|}\int_Q\sum_{|\beta|=d+1}
\lf|\frac{\partial^\beta f(\xi(x))}{\beta!}(x-x_0)^\beta\r|^q\,dx\r]^\frac1q\\\noz
&\lesssim\frac{|Q|}{\|\mathbf{1}_Q\|_{X}}\lf[\frac1{|Q|}\int_Q|x-x_0|^{q(d+1)}\,dx\r]^\frac1q\\\noz
&\lesssim|Q|^{1+(d+1)/n-1/p}\frac{|Q(\vec 0_n,\sqrt n)|^{1/p}}
{\|\mathbf{1}_{Q(\vec 0_n,\sqrt n)}\|_{X}}\lesssim1.
\end{align}
If $|x_0|+r > 1$ and $|x_0|\le 2r$, then $r>1/3$ and $|Q|\sim |Q(\vec0_n,\sqrt n(|x_0|+r))|$. From
Lemma \ref{Le48}, the H\"older inequality, the fact that $|f(x)|\lesssim(1+|x|)^{-n-\epsilon}$
for any $x\in\rn$
and some given $\epsilon\in(1+d,\infty)$, and \eqref{Eq44},
we deduce that
\begin{align}\label{Eq47}
\frac{|Q|}{\|\mathbf{1}_Q\|_{X}}\lf[\frac1{|Q|}\int_Q\lf|f(x)-P_Q^df(x)\r|^q\,dx\r]^\frac1q
&\lesssim\frac{|Q|}{\|\mathbf{1}_Q\|_{X}}\lf[\frac1{|Q|}\int_Q|f(x)|^q\,dx\r]^\frac1q\\\noz
&\lesssim\frac{|Q|}{\|\mathbf{1}_Q\|_{X}}\lf[\frac1{|Q|}\int_{B(\vec0_n,\sqrt n(|x_0|+r))}\frac1{(1+|x|)^{q(n+\epsilon)}}\,dx\r]^\frac1q\\\noz
&\lesssim\lf[\frac{|Q(\vec0_n,\sqrt n(|x_0|+r))|}{|Q|}\r]^{1/p}
\frac1{\|\mathbf{1}_{Q(\vec 0_n,\sqrt n(|x_0|+r)}\|_{X}}\lesssim1.
\end{align}
If $|x_0|+r>1$ and $|x_0|> 2r$, then, for any $x\in Q$, we have
$|x|\sim|x_0|\gtrsim2/3$ and $1+|x_0|\sim|x_0|+r$. By
this, \eqref{Eq45} and the fact that $|\partial^\gamma f(x)|\lesssim(1+|x|)^{-n-\epsilon}$
for any $x\in\rn$, $|\gamma|=d+1$ and some given $\epsilon\in(1+d,\infty)$, and \eqref{Eq44}, we find that
\begin{align}\label{Eq48}
\frac{|Q|}{\|\mathbf{1}_Q\|_{X}}\lf[\frac1{|Q|}\int_Q\lf|f(x)-P_Q^df(x)\r|^q\,dx\r]^\frac1q
&\lesssim\frac{|Q|}{\|\mathbf{1}_Q\|_{X}}\lf[\frac1{|Q|}\int_Q\sum_{|\beta|=d+1}\lf|
\frac{\partial^\beta f(\xi(x))}{\beta!}(x-x_0)^\beta\r|^q\,dx\r]^\frac1q\\\noz
&\lesssim(1+|x_0|)^{-n-\epsilon}\frac{|Q|}
{\|\mathbf{1}_Q\|_{X}}\lf[\frac1{|Q|}\int_Q|x-x_0|^{q(d+1)}\,dx\r]^\frac1q\\\noz
&\lesssim\frac{|Q|^{1+(d+1)/n}}{\|\mathbf{1}_Q\|_{X}}(1+|x_0|)^{-n-\epsilon}\\\noz
&\lesssim\frac{|Q|^{1+(d+1)/n-1/p}}{(|x_0|+r)^{n+\epsilon-n/p}}
\frac1{\|\mathbf{1}_{Q(\vec 0_n,\sqrt n(|x_0|+r))}\|_{X}}\lesssim1.
\end{align}
Combining \eqref{Eq46}, \eqref{Eq47}, and \eqref{Eq48}, we know that $f\in\cl_{q,X,d}(\rn)$, which completes
the proof of Lemma \ref{Le48}.
\end{proof}

Now let us show Theorem \ref{Thad}.
\begin{proof}[Proof of Theorem \ref{Thad}]
Assume that $\psi\in\cs(\rn)$ satisfies $\supp\psi\subset B(\vec 0_n,1)$ and $\int_\rn\psi(x)x^\gamma\,dx=0$
for any $\gamma\in\zz_+^n$ with $|\gamma|\le d$. Then, by Lemma \ref{Le47},
we know that there exists
$\phi\in\cs(\rn)$ such that the support of $\widehat\phi$ is compact and away from the origin and,
for any $x\in\rn\setminus\{\vec 0_n\}$,
$$
\int_0^\infty\widehat\psi(tx)\widehat\phi(tx)\,\frac{dt}{t}=1.
$$
Let $\eta$ be a function on $\rn$ such that
$\widehat\eta(\vec 0_n):=1$ and, for any $x\in\rn\setminus\{\vec 0_n\}$,
$$
\widehat\eta(x):=\int_1^\infty\widehat\psi(tx)\widehat\phi(tx)\,\frac{dt}{t}.
$$
Then, by \cite[p.\,219]{C1977}, we know that
such an $\eta$ exists and $\widehat\eta$ is infinitely differentiable,
has compact support and equals $1$ near the origin.

Let $x_0:=(2,\ldots,2)\in\rn$ and $f\in WH_X(\rn)$. For any $x\in\rn$ and $t\in(0,\infty)$, let
$\widetilde\phi(x):=\phi(x-x_0)$, $\widetilde\psi(x):=\psi(x+x_0)$, $F(x,t):=f\ast\widetilde\phi_t(x)$
and $G(x,t):=f\ast\eta_t(x)$. Then, due to Assumption \ref{xinm} and Theorem \ref{Thmc}(ii), for
any $f\in WH_X(\rn)$ and $x\in\rn$,
we have
\begin{equation}\label{m47}
M_\triangledown(f)(x):=\sup_{t\in(0,\infty),|y-x|\le3(|x_0|+1)t}[|F(y,t)|+|G(y,t)|]\in WX,
\end{equation}
and $\|M_\triangledown(f)\|_{WX}\sim\|f\|_{WH_X(\rn)}$.

Then, by Lemmas \ref{Le64} and \ref{Le47}, we know that
$$
f(x)=\int_0^\infty\int_\rn F(y,t)\widetilde\psi_t(x-y)\,\frac{dy\,dt}{t}\quad\text{in}\quad\cs'(\rn).
$$
For any $i\in\zz$, let $\Omega_i:=\{x\in\rn:\ M_\triangledown(f)(x)>2^i\}$. Then $\Omega_i$ is open and,
by \eqref{Eqde1}, we further find that
\begin{equation}\label{Eq418}
\sup_{i\in\zz}\lf\{2^i\lf\|\mathbf{1}_{\Omega_i}\r\|_{X}\r\}\le\lf\|M_\triangledown(f)\r\|_{WX}
\lesssim\|f\|_{WH_X(\rn)}.
\end{equation}
Since $\Omega_i$ is a proper open subset of $\rn$, by the Whitney decomposition (see, for instance, \cite[p.\,463]{G1}),
we know that there exists a sequence of cubes, $\{Q_{i,j}\}_{j\in\nn}$, such that, for any $i\in\zz$,
\begin{enumerate}
\item[(i)] $\bigcup_{j\in\nn} Q_{i,j}=\Omega_i$ and $\{Q_{i,j}\}_{j\in\nn}$ have disjoint interiors;
\item[(ii)] for any $j\in\nn$, $\sqrt nl_{Q_{i,j}}\le\dist(Q_{i,j},\Omega_i^\complement)
\le4\sqrt nl_{Q_{i,j}}$,
here and hereafter, $l_{Q_{i,j}}$ denotes the side length of the cube $Q_{i,j}$ and $\dist(Q_{i,j},\Omega_i^\complement)
:=\inf\{|x-y|:\ x\in Q_{i,j},\ y\in\Omega_i^\complement\}$;
\item[(iii)] for any $j,\ k\in\nn$, if the boundaries of two cubes $Q_{i,j}$ and $Q_{i,k}$ touch,
then $\frac14\le\frac{l_{Q_{i,j}}}{l_{Q_{i,k}}}\le4$;
\item[(iv)] for any given $j\in\nn$, there exist at most $12^n$ different cubes $\{Q_{i,k}\}_k$ that touch $Q_{i,j}$.
\end{enumerate}
For any $\epsilon\in(0,\infty)$, $i\in\zz$, $j\in\nn$ and $x\in\rn$, let
$$
\dist\lf(x,\Omega_i^\complement\r):=\inf\lf\{|x-y|:\ y\in\Omega_i\r\},
$$
$$
\widetilde\Omega_i:=\lf\{(x,t)\in\rr_+^{n+1}:=\rn\times(0,\infty):\
0<2t(|x_0|+1)<\dist\lf(x,\Omega_i^\complement\r)\r\},
$$
$$
\widetilde Q_{i,j}:=\lf\{(x,t)\in\rr_+^{n+1}:\ x\in Q_{i,j},\ (x,t)\in\widetilde\Omega_i\setminus\widetilde\Omega_{i+1}\r\}
$$
and
$$
b_{i,j}^\epsilon(x):=\int_\epsilon^{1/\epsilon}\int_\rn\mathbf{1}_{\widetilde Q_{i,j}}(y,t)F(y,t)\widetilde\psi_t(x-y)\,\frac{dy\,dt}{t}.
$$
Then, by the proof of \cite[pp.\,221-222]{C1977},
we know that there exist positive constants $C_1$
and $C_2$ such that,
for any $\epsilon\in(0,\infty)$, $i\in\zz$ and $j\in\nn$, $\supp b_{i,j}^\epsilon\subset C_1Q_{i,j}$,
$\|b_{i,j}^\epsilon\|_{L^\infty(\rn)}\le C_22^i$, and
$\int_\rn b_{i,j}^\epsilon(x)x^\gamma\,dx=0$ for any $\gamma\in\zz_+^n$ satisfying $|\gamma|\le d$.
Moreover, for any $\zeta\in\cs(\rn)$, by the Lebesgue dominated convergence theorem
and $\sum_{i\in\zz}\sum_{j\in\nn}\mathbf{1}_{\widetilde Q_{i,j}}=1$, we have
\begin{align*}
\lf\langle\sum_{i\in\zz}\sum_{j\in\nn}b_{i,j}^\epsilon,\zeta\r\rangle
&=\int_\rn\zeta(x)\sum_{i\in\zz}\sum_{j\in\nn}\int_\epsilon^{1/\epsilon}\int_\rn\mathbf{1}_{\widetilde Q_{i,j}}(y,t)F(y,t)
\widetilde\psi_t(x-y)\,\frac{dy\,dt}{t}\,dx\\
&=\int_\rn\zeta(x)\int_\epsilon^{1/\epsilon}\int_\rn F(y,t)
\widetilde\psi_t(x-y)\,\frac{dy\,dt}{t}\,dx
\end{align*}
and hence
\begin{equation}\label{Lefb}
f=\lim_{\epsilon\rightarrow0^+}\sum_{i\in\zz}
\sum_{j\in\nn}b_{i,j}^\epsilon\quad\text{in}\quad\cs'(\rn).
\end{equation}
Moreover, since, for any $\epsilon\in(0,1)$, $i\in\zz$ and $j\in\nn$,
$\|b_{i,j}^\epsilon\|_{L^\infty(\rn)}\le C_22^i$, $\{b_{i,j}^\epsilon\}_{\epsilon\in(0,1)}$
is bounded in $L^\infty(\rn)$. Then, by the Alaoglu theorem
(see, for instance, \cite[Theorem 3.17]{R}),
we find that there exist
$\{b_{i,j}\}_{i\in\zz,j\in\nn}\subset L^\infty(\rn)$ and a sequence $\{\epsilon_k\}_{k\in\nn}\subset(0,\infty)$
such that $\epsilon_k\rightarrow0$ as $k\rightarrow\infty$ and, for any $i\in\zz,\,j\in\nn$ and $g\in L^1(\rn)$,
\begin{equation}
\lim_{k\rightarrow\infty}\langle b_{i,j}^{\epsilon_k},g\rangle=\langle b_{i,j},g\rangle,
\end{equation}
$\supp b_{i,j}\subset C_1Q_{i,j}$, $\|b_{i,j}\|_{L^\infty(\rn)}\le C_22^i$ and,
for any $\gamma\in\zz_+^n$ with $|\gamma|\le d$,
$$
\int_\rn b_{i,j}(x)x^\gamma\,dx=\lf\langle b_{i,j},x^\gamma\mathbf{1}_{C_1Q_{i,j}}\r\rangle
=\lim_{k\rightarrow\infty}\int_\rn b_{i,j}^{\epsilon_k}(x)x^\gamma\,dx=0.
$$
Next we show that
\begin{equation}\label{Eq420}
\lim_{k\rightarrow\infty}\sum_{i\in\zz}\sum_{j\in\nn}b_{i,j}^{\epsilon_k}
=\sum_{i\in\zz}\sum_{j\in\nn}b_{i,j}\quad\text{in}\quad\cs'(\rn).
\end{equation}
Indeed, by the facts that, for any $i\in\zz$ and $j,\,k\in\nn$, $\|b_{i,j}\|_{L^\infty(\rn)}\lesssim2^i$, $\|b_{i,j}^{\epsilon_k}\|_{L^\infty(\rn)}\lesssim2^i$
and, for any $k\in\nn$ and $\gamma\in\zz_+^n$
with $|\gamma|\le d$, $\int_\rn b_{i,j}(x)x^\gamma\,dx
=0=\int_\rn b_{i,j}^{\epsilon_k}(x)x^\gamma\,dx$, we conclude that, for any $N\in\nn$
and $\zeta\in\mathcal{S}(\rn)$,
\begin{align*}
&\sum_{|i|\geq N}\sum_{j\in\nn}\lf[\lf|\lf\langle b_{i,j}^{\epsilon_k},\zeta\r\rangle\r|
+\lf|\lf\langle b_{i,j},\zeta\r\rangle\r|\r]\\
&\hspace*{12pt}=\sum_{i=-\infty}^{-N-1}\sum_{j\in\nn}
\lf[\lf|\lf\langle b_{i,j}^{\epsilon_k},\zeta\r\rangle\r|+\lf|\lf\langle b_{i,j},\zeta\r\rangle\r|\r]
+\sum_{i=N+1}^\infty\sum_{j\in\nn}\lf\{\lf|\int_{C_1Q_{i,j}}b_{i,j}^{\epsilon_k}(x)
\lf[\zeta(x)-P_{C_1Q_{i,j}}^d\zeta(x)\r]\,dx\r|\r.\\
&\hspace*{24pt}+\lf.\lf|\int_{C_1Q_{i,j}}b_{i,j}(x)
\lf[\zeta(x)-P_{C_1Q_{i,j}}^d\zeta(x)\r]\,dx\r|\r\}\\
&\hspace*{12pt}\lesssim\sum_{i=-\infty}^{-N-1}2^i\int_\rn|\zeta(x)|\,dx
+\sum_{i=N+1}^\infty\sum_{j\in\nn}2^i\int_{C_1Q_{i,j}}\lf|\zeta(x)-P_{C_1Q_{i,j}}^d\zeta(x)\r|\,dx.
\end{align*}
Since $d\geq\lfloor n(1/p-1)\rfloor$, it follows that
$p\in(\frac{n}{n+d+1},1)$, which, together with $X^{1/(\vartheta_0 p)}=[X^{1/\vartheta_0}]^{1/p}$,
the assumption that $\cm$ in \eqref{mm}
is bounded on $X^{1/(\vartheta_0 p)}$ and Lemma \ref{Le49},
further implies that, for any $\zeta\in\cs(\rn)$,
$\|\zeta\|_{\cl_{1,X^{1/\vartheta_0},d}(\rn)}<\infty$. By this, the assumption that $X^{1/\vartheta_0}$ is concave and $\vartheta_0>1$,
we further conclude that, for any $k,N\in\nn$
and $\zeta\in\mathcal{S}(\rn)$,
\begin{align}\label{411}
&\sum_{|i|\geq N}\sum_{j\in\nn}\lf[\lf|\lf\langle b_{i,j}^{\epsilon_k},\zeta\r\rangle\r|
+\lf|\lf\langle b_{i,j},\zeta\r\rangle\r|\r]\\ \noz
&\quad\lesssim2^{-N}\lf\|\zeta\r\|_{L^1(\rn)}+\sum_{i=N+1}^\infty\sum_{j\in\nn}2^i
\lf\|\mathbf{1}_{Q_{i,j}}\r\|_{X^{1/\vartheta_0}}
\lf\|\zeta\r\|_{\cl_{1,X^{1/\vartheta_0},d}(\rn)}\\\noz
&\quad\lesssim2^{-N}\lf\|\zeta\r\|_{L^1(\rn)}+\lf\|\zeta\r\|_{\cl_{1,X^{1/\vartheta_0},d}(\rn)}
\sum_{i=N+1}^\infty2^i\lf\|\mathbf{1}_{\Omega_{i}}\r\|_{X^{1/\vartheta_0}}\\ \noz
&\quad\lesssim2^{-N}\|\zeta\|_{L^1(\rn)}+\|\zeta\|_{\cl_{1,X^{1/\vartheta_0},d}(\rn)}
\lf[\sup_{i\in\zz}2^i\lf\|\mathbf{1}_{\Omega_{i}}
\r\|_{X}\r]^{\vartheta_0}\sum_{i=N+1}^\infty
2^{-i(\vartheta_0-1)}\\ \noz
&\quad\lesssim2^{-N}\lf\|\zeta\r\|_{L^1(\rn)}+2^{-N(\vartheta_0-1)}
\lf\|\zeta\r\|_{\cl_{1,X^{1/\vartheta_0},d}(\rn)}
\|f\|_{WH_X(\rn)}^{\vartheta_0},
\end{align}
where the implicit positive constants are independent of $k$, $N$ and $f$.
Similarly, for any given $N\in\nn$ and $\zeta\in\mathcal{S}(\rn)$, there exists a
positive constant $C_{(N,\zeta)}$ such that, for any $k\in\nn$,
\begin{align}\label{412}
\sum_{|i|\leq N}\sum_{j\in\nn}\lf[\lf|\langle b_{i,j}^{\epsilon_k},\zeta\rangle\r|
+\lf|\langle b_{i,j},\zeta\rangle\r|\r]\leq C_{(N,\zeta)}<\infty.
\end{align}
Therefore, using \eqref{411} and \eqref{412}, repeating the
argument similar to that used in \cite[p.\,651]{LYJ}, we find that \eqref{Eq420} holds true.

For any $i\in\zz$ and $j\in\nn$, let $B_{i,j}$ be the ball with the same center as $Q_{i,j}$
and the radius $5\sqrt nC_1l_{Q_{i,j}}$,
$$
a_{i,j}:=\frac{b_{i,j}}{C_22^i\|\mathbf{1}_{B_{i,j}}\|_{X}}
\quad\text{and}\quad\lambda_{i,j}:=C_22^i\lf\|\mathbf{1}_{B_{i,j}}\r\|_{X}.
$$
Then, using the properties of $b_{i,j}$, we know that $a_{i,j}$ is
an $(X,\infty,d)$-atom supported in the ball $B_{i,j}$ satisfying that
$\{cB_{i,j}\}_{j\in\nn}$ is finite overlapping for some $c\in(0,1]$
and, due to \eqref{Eq420} and \eqref{Lefb}, $f=\sum_{i\in\zz}\sum_{j\in\nn}\lambda_{i,j}a_{i,j}$ in $\cs'(\rn)$.
Similarly to \eqref{EqHLMS},
by \eqref{eq}, we conclude that
$$\lf\|\sum_{j\in\nn}\mathbf{1}_{B_{i,j}}\r\|_{X}\lesssim\lf\|\sum_{j\in\nn}\mathbf{1}_{Q_{i,j}}\r\|_{X}.$$
From this and \eqref{Eq418}, we deduce that
$$
\sup_{i\in\zz}\lf\|\sum_{j\in\nn}
\frac{\lambda_{i,j}\mathbf{1}_{B_{i,j}}}{\|\mathbf{1}_{B_{i,j}}\|_X}\r\|_X\sim \sup_{i\in\zz}2^i\lf\|\sum_{j\in\nn}\mathbf{1}_{B_{i,j}}\r\|_{X}
\lesssim \sup_{i\in\zz}2^i\lf\|\sum_{j\in\nn}\mathbf{1}_{Q_{i,j}}\r\|_{X}
\lesssim \sup_{i\in\zz}2^i\lf\|\mathbf{1}_{\Omega_{i}}\r\|_{X}\lesssim\|f\|_{WH_X(\rn)},
$$
which completes the proof of Theorem \ref{Thad}.
\end{proof}

Next we present a reconstruction theorem.

\begin{theorem}\label{Thar}
Let $X$ be a ball quasi-Banach function space satisfying
Assumption \ref{a2.15} for some $p_-\in(0,\infty)$.
Assume that, for any given $r\in(0,\underline{p})$
with $\underline{p}$ as in \eqref{Eqpll},
$X^{1/r}$ is a ball Banach function space.
Assume that there exist $r_0\in(0,\underline{p})$ and $p_0\in(r_0,\infty)$ such that,
for any $f\in(X^{1/r_0})'$,
\begin{equation}\label{Eqdm}
\lf\|\cm^{((p_0/r_0)')}(f)\r\|_{(X^{1/r_0})'}\le C\lf\|f\r\|_{(X^{1/r_0})'},
\end{equation}
where the positive constant $C$ is independent of $f$.
Let $d\in\zz_+$ with $d\geq \lfloor n(1/\underline{p}-1)\rfloor$,
$c\in(0,1]$, $q\in(\max\{1,p_0\},\infty]$
and $A,\ \widetilde A\in(0,\infty)$ and let
$\{a_{i,j}\}_{i\in\zz,j\in\nn}$ be a sequence of $(X,\,q,\,d)$-atoms supported, respectively, in balls
$\{B_{i,j}\}_{i\in\zz,j\in\nn}$
satisfying that, $\sum_{j\in\nn}\mathbf{1}_{cB_{i,j}}\le A$ for any
$i\in\zz$,
$\lambda_{i,j}:=\widetilde A2^i\|\mathbf{1}_{B_{i,j}}\|_{X}$ for any $i\in\zz$ and $j\in\nn$,
the series
$
f:=\sum_{i\in\zz}\sum_{j\in\nn}\lambda_{i,j}a_{i,j}
$
converges in $\cs'(\rn)$
and
$$
\sup_{i\in\zz}\lf\|\sum_{j\in\nn}
\frac{\lambda_{i,j}\mathbf{1}_{B_{i,j}}}{\|\mathbf{1}_{B_{i,j}}\|_{X}}\r\|_{X}<\infty.
$$
Then $f\in WH_X(\rn)$ and
$$
\|f\|_{WH_X(\rn)}\lesssim\sup_{i\in\zz}\lf\|\sum_{j\in\nn}
\frac{\lambda_{i,j}\mathbf{1}_{B_{i,j}}}{\|\mathbf{1}_{B_{i,j}}\|_{X}}\r\|_{X},
$$
where the implicit positive constant is independent of $f$.
\end{theorem}

To prove Theorem \ref{Thar}, we need the following useful technical lemma.

\begin{lemma}\label{Le45}
Let $r\in(0,\infty)$, $q\in(r,\infty]$ and $X$ be a ball quasi-Banach function space.
Assume that $X^{1/r}$ is a ball Banach function space and there exists a positive
constant $C$ such that, for any $f\in (X^{1/r})'$,
$\|\cm^{((q/r)')}(f)\|_{(X^{1/r})'}\le C\lf\|f\r\|_{(X^{1/r})'}$.
Then there exists a positive constant $C$
such that, for any sequence $\{B_j\}_{j\in\nn}$ of balls,
numbers $\{\lambda_j\}_{j\in\nn}\subset\cc$ and measurable functions
$\{a_j\}_{j\in\nn}$ satisfying that, for any $j\in\nn$, $\supp(a_j)\subset B_j$ and
$\|a_j\|_{L^q(\rn)}\le|B_j|^{1/q}$,
$$
\lf\|\lf(\sum_{j\in\nn}|\lambda_ja_j|^r\r)^\frac1r\r\|_{X}
\le C\lf\|\lf(\sum_{j\in\nn}|\lambda_j\mathbf{1}_{B_j}|^r\r)^\frac1r\r\|_{X}.
$$
\end{lemma}

\begin{proof}
By the definition of the associate space, the assumption that $X^{1/r}$ is a ball Banach function space
and Lemma \ref{Lesdual}, we have
\begin{align*}
\lf\|\lf(\sum_{j\in\nn}|\lambda_ja_j|^r\r)^\frac1r\r\|_{X}^r
&=\lf\|\sum_{j\in\nn}|\lambda_ja_j|^r\r\|_{X^{1/r}}
=\lf\|\sum_{j\in\nn}\lf|\lambda_ja_j\r|^r\r\|_{(X^{1/r})''}\\
&=\sup\lf\{\int_\rn\sum_{j\in\nn}\lf|\lambda_ja_j(x)\r|^rg(x)\,dx:
\ g\in (X^{1/r})'\text{ such that }\lf\|g\r\|_{(X^{1/r})'}=1\r\}.
\end{align*}
Then, from the H\"older inequality, we deduce that,
for any $g\in (X^{1/r})'$ with $\|g\|_{(X^{1/r})'}=1$,
\begin{align*}
\int_\rn\sum_{j\in\nn}|\lambda_ja_j(x)|^r|g(x)|\,dx&
=\sum_{j\in\nn}|\lambda_j|^r\int_{\rn}|a_j(x)|^r|g(x)|\,dx
\le\sum_{j\in\nn}|\lambda_j|^r\|a_j\|_{L^q(\rn)}^r
\lf\|g\mathbf{1}_{B_j}\r\|_{L^{(q/r)'}(\rn)}\\
&\le\sum_{j\in\nn}|\lambda_j|^r\lf|B_j\r|^{r/q}\lf\|g\mathbf{1}_{B_j}\r\|_{L^{(q/r)'}(\rn)}\\
&\le\sum_{j\in\nn}|\lambda_j|^r\int_{\rn}
\mathbf{1}_{B_j}(x)\lf[\cm\lf(g^{(q/r)'}\r)(x)\r]^{1/(q/r)'}\,dx=:K.
\end{align*}
Applying Lemma \ref{LeHolder} and the assumption that $\cm^{((q/r)')}$
is bounded on $(X^{1/r})'$, we conclude that
\begin{align*}
K&\lesssim\lf\|\sum_{j\in\nn}|\lambda_j|^r\mathbf{1}_{B_j}\r\|_{X^{1/r}}
\lf\|\lf[\cm\lf(g^{(q/r)'}\r)\r]^{1/(q/r)'}\r\|_{(X^{1/r})'}\\
&\lesssim\lf\|\sum_{j\in\nn}|\lambda_j|^r\mathbf{1}_{B_j}\r\|_{X^{1/r}}
\lf\|g\r\|_{(X^{1/r})'}
\sim\lf\|\sum_{j\in\nn}|\lambda_j\mathbf{1}_{B_j}|^r\r\|_{X^{1/r}},
\end{align*}
which, together with Definition \ref{Debf}(i), further implies the desired conclusion.
This finishes the proof of Lemma \ref{Le45}.
\end{proof}

Now we show Theorem \ref{Thar}.
\begin{proof}[Proof of Theorem \ref{Thar}]
Let $c\in(0,1]$, $q\in(p_0,p_0/r_0]$ and $\{a_{i,j}\}_{i\in\zz,j\in\nn}$ be a sequence of
$(X,\,q,\,d)$-atoms supported, respectively, in balls
$\{B_{i,j}\}_{i\in\zz,j\in\nn}$
satisfying that, for any
$i\in\zz$, $\sum_{j\in\nn}\mathbf{1}_{cB_{i,j}}\le A$
with $A$ being a positive constant independent of $i$,
$\lambda_{i,j}:=\widetilde A2^i\|\mathbf{1}_{B_{i,j}}\|_{X}$ for any $i\in\zz$ and $j\in\nn$
with $\widetilde A$ being a positive constant independent of
$i$ and $j$,
$$
f:=\sum_{i\in\zz}\sum_{j\in\nn}\lambda_{i,j}a_{i,j}\quad\text{in}\quad\cs'(\rn),
$$
and
$$
\sup_{i\in\zz}2^i\lf\|\sum_{j\in\nn}\mathbf{1}_{B_{i,j}}\r\|_{X}<\infty.
$$
To prove $f\in WH_X(\rn)$, by the definition of $WH_X(\rn)$, it suffices to show that
$$
\sup_{\alpha\in(0,\infty)}\lf\{\alpha\lf\|
\mathbf{1}_{\{x\in\rn:\ M_N^0(f)(x)>\alpha\}}\r\|_{X}\r\}
\lesssim\sup_{i\in\zz}2^i\lf\|\sum_{j\in\nn}\mathbf{1}_{B_{i,j}}\r\|_{X}.
$$
For any fixed $\alpha\in(0,\infty)$, let
$i_0\in\zz$ be such that $2^{i_0}\le\alpha<2^{i_0+1}$. Then we write
$$
f=\sum_{i=-\infty}^{i_0-1}\sum_{j\in\nn}\lambda_{i,j}a_{i,j}
+\sum_{i=i_0}^{\infty}\sum_{j\in\nn}\lambda_{i,j}a_{i,j}=:f_1+f_2.
$$
Then it follows from Definition \ref{Debqfs}(ii) that
\begin{align}\label{EqI123}
&\lf\|\mathbf{1}_{\{x\in\rn:\ M_N^0(f)(x)>\alpha\}}\r\|_{X}\\\noz
&\hs\lesssim \lf\|
\mathbf{1}_{\{x\in\rn:\ M_N^0(f_1)(x)>\frac{\alpha}2\}}\r\|_{X}+\lf\|
\mathbf{1}_{\{x\in A_{i_0}:\ M_N^0(f_2)(x)>\frac{\alpha}2\}}\r\|_{X}+\lf\|
\mathbf{1}_{\{x\in(A_{i_0})^\complement:\ M_N^0(f_2)(x)>\frac{\alpha}2\}}\r\|_{X}\\\noz
&\hs=:\mathrm{I_1}+\mathrm{I_2}+\mathrm{I_3},
\end{align}
where $A_{i_0}:=\bigcup_{i=i_0}^\infty\bigcup_{j\in\nn}(2B_{i,j})$.

For $\mathrm{I_1}$, by Definition \ref{Debqfs}(ii), we further decompose it into
\begin{align}\label{Eq410}
\mathrm{I_1}&\lesssim\lf\|
\mathbf{1}_{\{x\in\rn:\ \sum_{i=-\infty}^{i_0-1}\sum_{j\in\nn}
\lambda_{i,j}M_N^0(a_{i,j})(x)\mathbf{1}_{2B_{i,j}}(x)>\frac{\alpha}4\}}\r\|_{X}\\
&\qquad+\lf\|
\mathbf{1}_{\{x\in\rn:\ \sum_{i=-\infty}^{i_0-1}\sum_{j\in\nn}
\lambda_{i,j}M_N^0(a_{i,j})(x)\mathbf{1}_{(2B_{i,j})^\complement}(x)>\frac{\alpha}4\}}\r\|_{X}\noz\\\noz
&=:\mathrm{I_{1,1}}+\mathrm{I_{1,2}}.
\end{align}
We first estimate $\mathrm{I_{1,1}}$. Let $\widetilde q:=q/p_0\in(1,1/{r_0}]$
and $a\in(0,1-1/{\widetilde q})$. Then, from the H\"older inequality, we deduce that
\begin{align*}
&\sum_{i=-\infty}^{i_0-1}\sum_{j\in\nn}\lambda_{i,j}M_N^0(a_{i,j})\mathbf{1}_{2B_{i,j}}
\le\frac{2^{i_0a}}{(2^{a\widetilde q'}-1)^{1/\widetilde q'}}
\lf\{\sum_{i=-\infty}^{i_0-1}2^{-ia\widetilde q}\lf[\sum_{j\in\nn}
\lambda_{i,j}M_N^0(a_{i,j})\mathbf{1}_{2B_{i,j}}\r]^{\widetilde q}\r\}^{1/\widetilde q},
\end{align*}
where $\widetilde q':={\widetilde q}/{(\widetilde q-1)}$. By this, Definitions \ref{Debqfs}(ii)
and \ref{Debf}(i), $\widetilde qr_0\in(0,1]$ and
the fact that $M_N^0(f)\lesssim\cm(f)$
and the assumption that $X^{1/r_0}$ is a ball Banach function space,
we conclude that
\begin{align*}
\mathrm{I_{1,1}}
&\lesssim\lf\|\mathbf{1}_{\{x\in\rn:\ \frac{2^{i_0a}}
{(2^{a\widetilde q'}-1)^{1/\widetilde q'}}
\{\sum_{i=-\infty}^{i_0-1}2^{-ia\widetilde q}[\sum_{j\in\nn}
\lambda_{i,j}M_N^0(a_{i,j})(x)\mathbf{1}_{2B_{i,j}}(x)]^
{\widetilde q}\}^{1/\widetilde q}>2^{i_0-2}\}}\r\|_{X}\\
&\lesssim2^{-i_0\widetilde q(1-a)}\lf\|\sum_{i=-\infty}^
{i_0-1}2^{-ia\widetilde q}\lf[\sum_{j\in\nn}
\lambda_{i,j}M_N^0(a_{i,j})\mathbf{1}_{2B_{i,j}}\r]^{\widetilde q}\r\|_{X}\\
&\lesssim2^{-i_0\widetilde q(1-a)}
\lf\|\sum_{i=-\infty}^{i_0-1}2^{(1-a)i\widetilde qr_0}\sum_{j\in\nn}\lf[
\lf\|\mathbf{1}_{B_{i,j}}\r\|_{X}M_N^0(a_{i,j})\mathbf{1}_{2B_{i,j}}\r]^{\widetilde qr_0}\r\|_{X^{1/r_0}}^\frac{1}{r_0}\\
&\lesssim2^{-i_0\widetilde q(1-a)}\lf[\sum_{i=-\infty}^{i_0-1}2^{(1-a)i\widetilde qr_0}\lf\|\lf\{\sum_{j\in\nn}\lf[
\lf\|\mathbf{1}_{B_{i,j}}\r\|_{X}\cm(a_{i,j})\mathbf{1}_{2B_{i,j}}\r]^{\widetilde qr_0}\r\}^\frac1{r_0}\r\|_{X}^{r_0}\r]^\frac{1}{r_0}.
\end{align*}
From $q=p_0\widetilde q$ and the boundedness of $\cm$ on $L^q(\rn)$ and Definition \ref{Deatom}(ii),
it follows that, for any $i\in\zz$ and $j\in\nn$,
\begin{align*}
\lf\|\lf[
\|\mathbf{1}_{B_{i,j}}\|_{X}\cm(a_{i,j})\r]^{\widetilde q}\mathbf{1}_{2B_{i,j}}\r\|_{L^{p_0}(\rn)}
\lesssim\lf\|\mathbf{1}_{B_{i,j}}\r\|_{X}^{\widetilde q}\lf\|\cm(a_{i,j})\mathbf{1}_{2B_{i,j}}\r\|_{L^q(\rn)}^{\widetilde q}
\lesssim\lf\|\mathbf{1}_{B_{i,j}}\r\|_{X}^{\widetilde q}\lf\|a_{i,j}\r\|_{L^q(\rn)}^{\widetilde q}
\lesssim\lf|B_{i,j}\r|^\frac1{p_0},
\end{align*}
which, combined with Lemma \ref{Le45}, \eqref{EqHLMS} and $(1-a)\widetilde q>1$, further implies that
\begin{align*}
\mathrm{I_{1,1}}&\lesssim2^{-i_0\widetilde q(1-a)}
\lf[\sum_{i=-\infty}^{i_0-1}2^{(1-a)i\widetilde qr_0}\lf\|\lf(\sum_{j\in\nn}\mathbf{1}_{2B_{i,j}}\r)^\frac1{r_0}\r\|_{X}^{r_0}\r]^{\frac1{r_0}}\\
&\lesssim2^{-i_0\widetilde q(1-a)}
\lf[\sum_{i=-\infty}^{i_0-1}2^{(1-a)i\widetilde qr_0}\lf\|\lf(\sum_{j\in\nn}\mathbf{1}_{cB_{i,j}}\r)^\frac1{r_0}\r\|_{X}^{r_0}\r]^{\frac1{r_0}}\\
&\lesssim2^{-i_0\widetilde q(1-a)}
\lf[\sum_{i=-\infty}^{i_0-1}2^{[(1-a)\widetilde
q-1]ir_0}\r]^{\frac1{r_0}}\sup_{i\in\zz}2^i\lf\|\sum_{j\in\nn}
\mathbf{1}_{B_{i,j}}\r\|_{X}\\
&\lesssim\alpha^{-1}
\sup_{i\in\zz}2^i\lf\|\sum_{j\in\nn}\mathbf{1}_{B_{i,j}}\r\|_{X}.
\end{align*}
This shows that
\begin{equation}\label{Eq411}
\alpha\mathrm{I_{1,1}}\lesssim\sup_{i\in\zz}2^i\lf\|\sum_{j\in\nn}\mathbf{1}_{B_{i,j}}\r\|_{X}.
\end{equation}

To deal with $\mathrm{I_{1,2}}$, we first estimate $M_N^0(f)$ on $(2B_{i,j})^\complement$. Let $\phi\in\cf_N(\rn)$ and,
for any $i\in\zz$ and $j\in\nn$,
let $x_{i,j}$ denote the center of $B_{i,j}$ and $r_{i,j}$ its radius.
Then, using the vanishing moments of $a_{i,j}$ and the Taylor remainder theorem, we have,
for any $i\in\zz\cap[i_0,\infty)$, $j\in\nn$,
$t\in(0,\infty)$ and $x\in\rn$,
\begin{align}\label{EqMNa0}
|a_{i,j}\ast\phi_t(x)|&=\lf|\int_{B_{i,j}}a_{i,j}(y)\lf[\phi\lf(\frac{x-y}{t}\r)-\sum_{|\beta|\le d}\frac{\partial^\beta\phi(\frac{x-x_{i,j}}{t})}
{\beta!}\lf(\frac{x_{i,j}-y}{t}\r)^\beta\r]\,\frac{dy}{t^n}\r|\\\noz
&\lesssim\int_{B_{i,j}}\lf|a_{i,j}(y)\r|\sum_{|\beta|= d+1}\lf|\partial^\beta\phi\lf(\frac{\xi}{t}\r)\r|
\lf|\frac{x_{i,j}-y}{t}\r|^{d+1}\,\frac{dy}{t^n},
\end{align}
where $\xi:=(x-x_{i,j})+\theta(x_{i,j}-y)$ for some $\theta\in[0,1]$.

For any $i\in\zz$, $j\in\nn$, $x\in(2B_{i,j})^\complement$ and $y\in B_{i,j}$,
it is easy to see that $|x-y|\sim|x-x_{i,j}|$
and $|\xi|\geq|x-x_{i,j}|-|x_{i,j}-y|\gtrsim|x-x_{i,j}|$. By this, \eqref{EqMNa0}, the fact that $\phi\in\cs(\rn)$
and the H\"older inequality, we conclude that, for any $i\in\zz\cap[i_0,\infty)$, $j\in\nn$,
$t\in(0,\infty)$ and $x\in(2B_{i,j})^\complement$,
\begin{align}\label{EqMNa}
|a_{i,j}\ast\phi_t(x)|
&\lesssim\int_{B_{i,j}}|a_{i,j}(y)|\frac{|y-x_{i,j}|^{d+1}}{|x-x_{i,j}|^{n+d+1}}\,dy\\\noz
&\lesssim\frac{(r_{i,j})^{d+1}}{|x-x_{i,j}|^{n+d+1}}
\lf[\int_{B_{i,j}}|a_{i,j}(y)|^q\,dy\r]^{1/q}|B_{i,j}|^{1/q'}
\lesssim\lf\|\mathbf{1}_{B_{i,j}}\r\|_{X}^{-1}\lf(\frac{r_{i,j}}{|x-x_{i,j}|}\r)^{n+d+1},
\end{align}
which implies that, for any $x\in(2B_{i,j})^\complement$,
\begin{equation}\label{Eq412}
M_N^0(a_{i,j})(x)\lesssim
\lf\|\mathbf{1}_{B_{i,j}}\r\|_{X}^{-1}\lf[\cm(\mathbf{1}_{B_{i,j}})(x)\r]^\frac{n+d+1}{n}.
\end{equation}
Observe that $d\geq\lfloor n(\frac1{\underline{p}}-1)\rfloor$ implies that $\underline{p}\in(\frac{n}{n+d+1},1]$.
Let $r_1\in(0,\frac{n}{n+d+1})\subset(0,\underline{p})$, $q_1\in(\frac{n}{(n+d+1)r_1},\frac1{r_1})\subset(1,\infty)$ and $a\in(0,1-\frac1{q_1})$.
From the H\"older inequality, it follows that
\begin{align*}
&\sum_{i=-\infty}^{i_0-1}\sum_{j\in\nn}\lambda_{i,j}M_N^0(a_{i,j})\mathbf{1}_{(2B_{i,j})^\complement}
\le\frac{2^{i_0a}}{(2^{aq_1'}-1)^{1/q_1'}}
\lf\{\sum_{i=-\infty}^{i_0-1}2^{-iaq_1}\lf[\sum_{j\in\nn}
\lambda_{i,j}M_N^0(a_{i,j})\mathbf{1}_{(2B_{i,j})^\complement}\r]^{q_1}\r\}^{1/q_1},
\end{align*}
where $q_1':={q_1}/{(q_1-1)}$. By this, Definition \ref{Debqfs}(ii), \eqref{Eq412}, the definition of $\lambda_{i,j}$
and the assumption that $X^{1/r_1}$ is a ball Banach function space, we conclude that
\begin{align*}
\mathrm{I_{1,2}}&\lesssim
\lf\|\mathbf{1}_{\{x\in\rn:\frac{2^{i_0a}}{(2^{aq_1'}-1)^{1/q_1'}}
\{\sum_{i=-\infty}^{i_0-1}2^{-iaq_1}[\sum_{j\in\nn}
\lambda_{i,j}M_N^0(a_{i,j})(x)\mathbf{1}_{(2B_{i,j})^\complement}(x)]^{q_1}\}^{1/q_1}>2^{i_0-2}\}}\r\|_{X}\\
&\lesssim2^{-i_0q_1(1-a)}\lf\|\sum_{i=-\infty}^{i_0-1}2^{-iaq_1}\lf[\sum_{j\in\nn}
\lambda_{i,j}M_N^0(a_{i,j})\mathbf{1}_{(2B_{i,j})^\complement}\r]^{q_1}\r\|_{X}\\
&\lesssim2^{-i_0q_1(1-a)}\lf\{\sum_{i=-\infty}^{i_0-1}2^{(1-a)iq_1r_1}\lf\|\sum_{j\in\nn}\lf[
\cm(\mathbf{1}_{B_{i,j}})\r]^{\frac{(n+d+1)q_1r_1}{n}}\r\|_{X^{\frac1{r_1}}}\r\}^\frac1{r_1}.
\end{align*}
It is easy to see that $\frac{(n+d+1)q_1r_1}{n}\in(1,\infty)$, $\frac{n}{(n+d+1)q_1}\in(0,r_1)\subset(0,\underline{p})$
and $(1-a)q_1\in(1,\infty)$.
Then, from Definition \ref{Debf}(i), \eqref{EqHLMS} and $\sum_{j\in\nn}\mathbf{1}_{cB_{i,j}}\le A$, we further deduce that
\begin{align*}
\mathrm{I_{1,2}}&\lesssim2^{-i_0q_1(1-a)}
\lf[\sum_{i=-\infty}^{i_0-1}2^{(1-a)iq_1r_1}\lf\|\lf\{\sum_{j\in\nn}\lf[
\cm(\mathbf{1}_{B_{i,j}})\r]^{\frac{(n+d+1)q_1r_1}{n}}\r\}
^\frac{n}{(n+d+1)q_1r_1}\r\|_{X^{\frac{(n+d+1)q_1}{n}}}^\frac{(n+d+1)q_1r_1}{n}\r]^\frac1{r_1}\\
&\lesssim2^{-i_0q_1(1-a)}\lf\{\sum_{i=-\infty}^{i_0-1}2^{(1-a)iq_1{r_1}}\lf\|\sum_{j\in\nn}
\mathbf{1}_{B_{i,j}}\r\|_{X^{\frac1{r_1}}}\r\}^\frac1{r_1}\\
&\lesssim2^{-i_0q_1(1-a)}\lf\{\sum_{i=0}^{i_0-1}2^{[(1-a)q_1-1]ir_1}2^{ir_1}
\lf\|\lf(\sum_{j\in\nn}\mathbf{1}_{cB_{i,j}}\r)^\frac1{r_1}\r\|_{X}^{r_1}\r\}^\frac1{r_1}
\lesssim\alpha^{-1}\sup_{i\in\zz}2^i\lf\|\sum_{j\in\nn}\mathbf{1}_{B_{i,j}}\r\|_{X},
\end{align*}
which implies that
\begin{equation}\label{EqI12}
\alpha\mathrm{I_{1,2}}\lesssim\sup_{i\in\zz}2^i\lf\|\sum_{j\in\nn}\mathbf{1}_{B_{i,j}}\r\|_{X}.
\end{equation}
By this, \eqref{Eq410} and \eqref{Eq411}, we find that
\begin{equation}\label{Eq414}
\alpha\mathrm{I_1}\lesssim\sup_{i\in\zz}2^i\lf\|\sum_{j\in\nn}\mathbf{1}_{B_{i,j}}\r\|_{X}.
\end{equation}

Next we deal with $\mathrm{I_2}$. Let $r_2\in(0,\underline p)$.
Then, by \eqref{EqHLMS}, Definition \ref{Debf}(i),
the assumption that $X^{1/r_2}$ is a ball Banach
function space and $\sum_{j\in\nn}\mathbf{1}_{cB_{i,j}}\le A$,
we conclude that
\begin{align*}
\mathrm{I_2}&\lesssim\lf\|\mathbf{1}_{A_{i_0}}\r\|_{X}\lesssim
\lf\|\sum_{i=i_0}^\infty\sum_{j\in\nn}\mathbf{1}_{2B_{i,j}}\r\|_{X}
\lesssim\lf\|\sum_{i=i_0}^\infty\sum_{j\in\nn}\mathbf{1}_{cB_{i,j}}\r\|_{X}
\sim\lf\|\lf\{\sum_{i=i_0}^\infty\sum_{j\in\nn}
\mathbf{1}_{cB_{i,j}}\r\}^{r_2}\r\|_{X^\frac1{r_2}}^\frac1{r_2}\\
&\lesssim\lf[\sum_{i=i_0}^\infty\lf\|\sum_{j\in\nn}
\mathbf{1}_{cB_{i,j}}\r\|_{X^\frac1{r_2}}\r]^\frac1{r_2}
\lesssim\lf[\sum_{i=i_0}^\infty\lf\|\sum_{j\in\nn}
\mathbf{1}_{cB_{i,j}}\r\|_{X}^{r_2}\r]^\frac1{r_2}
\lesssim\lf\{\sum_{i=i_0}^\infty2^{-ir_2}\lf[2^i
\lf\|\sum_{j\in\nn}\mathbf{1}_{B_{i,j}}\r\|_{X}\r]^{r_2}\r\}^{\frac1{r_2}}\\
&\lesssim\sup_{i\in\nn}2^i\lf\|\sum_{j\in\nn}
\mathbf{1}_{B_{i,j}}\r\|_{X}\lf(\sum_{i=i_0}^\infty2^{-ir_2}\r)^\frac1{r_2}
\lesssim\alpha^{-1}\sup_{i\in\zz}2^i\lf\|\sum_{j\in\nn}\mathbf{1}_{B_{i,j}}\r\|_{X},
\end{align*}
which implies that
\begin{equation}\label{Eq415}
\alpha\mathrm{I_2}\lesssim\sup_{i\in\zz}2^i\lf\|\sum_{j\in\nn}\mathbf{1}_{B_{i,j}}\r\|_{X}.
\end{equation}

It remains to estimate $\mathrm{I_3}$. Recall that $\underline{p}\in(\frac{n}{n+d+1},1]$
and hence there exists $r_3\in(\frac{n}{\underline{p}(n+d+1)},1)$.
By Definitions \ref{Debqfs}(ii) and \ref{Debf}(i),
the assumption that $X^\frac{(n+d+1)r_3}{n}$ is a ball Banach function space
and \eqref{Eq412}, we conclude that
\begin{align*}
\mathrm{I_3}&\lesssim\lf\|\mathbf{1}_{\{x\in(A_{i_0})^\complement:\
\sum_{i=i_0}^\infty\sum_{j\in\nn}\lambda_{i,j}M_N^0(a_{i,j})(x)>\frac{\alpha}2\}}\r\|_{X}
\lesssim\alpha^{-r_3}\lf\|\sum_{i=i_0}^\infty\sum_{j\in\nn}
\lf[\lambda_{i,j}M_N^0(a_{i,j})\r]^{r_3}\mathbf{1}_{(A_{i_0})^\complement}\r\|_{X}\\
&\sim\alpha^{-r_3}\lf\|\lf\{\sum_{i=i_0}^\infty\sum_{j\in\nn}
\lf[\lambda_{i,j}M_N^0(a_{i,j})\r]^{r_3}
\mathbf{1}_{(A_{i_0})^\complement}\r\}^\frac{n}{(n+d+1)r_3}
\r\|_{X^\frac{(n+d+1)r_3}{n}}^\frac{(n+d+1)r_3}{n}\\
&\lesssim\alpha^{-r_3}\lf[\sum_{i=i_0}^\infty
\lf\|\lf\{\sum_{j\in\nn}\lf[\lambda_{i,j}M_N^0(a_{i,j})\r]^{r_3}
\mathbf{1}_{(A_{i_0})^\complement}\r\}^\frac{n}{(n+d+1)r_3}
\r\|_{X^\frac{(n+d+1)r_3}{n}}\r]^\frac{(n+d+1)r_3}{n}\\
&\lesssim\alpha^{-r_3}\lf\{\sum_{i=i_0}^\infty2^\frac{in}{(n+d+1)}\lf\|\sum_{j\in\nn}
\lf[\cm(\mathbf{1}_{B_{i,j}})\r]^\frac{(n+d+1)r_3}{n}
\r\|_{X}^\frac{n}{(n+d+1)r_3}\r\}^\frac{(n+d+1)r_3}{n}.
\end{align*}
Since $\frac{n}{(n+d+1)r_3}\in(0,\underline{p})\subset(0,1)$,
from Definition \ref{Debf}(i) and Assumption \ref{a2.15}, it follows that
\begin{align*}
\mathrm{I_3}&\lesssim\alpha^{-r_3}\lf[\sum_{i=i_0}^\infty2^\frac{in}{(n+d+1)}\lf\|\sum_{j\in\nn}
\mathbf{1}_{B_{i,j}}\r\|_{X}^\frac{n}{(n+d+1)r_3}\r]^\frac{(n+d+1)r_3}{n}
\lesssim\alpha^{-r_3}\sup_{i\in\zz}2^i\lf\|\sum_{j\in\nn}\mathbf{1}_{B_{i,j}}\r\|_{X}
\lf[\sum_{i=i_0}^\infty2^\frac{in(r_3-1)}{(n+d+1)r_3}\r]^\frac{(n+d+1)r_3}{n}\\
&\lesssim\alpha^{-1}\sup_{i\in\zz}2^i\lf\|\sum_{j\in\nn}\mathbf{1}_{B_{i,j}}\r\|_{X},
\end{align*}
namely,
\begin{equation}\label{Eq416}
\alpha\mathrm{I_3}\lesssim\sup_{i\in\zz}2^i\lf\|\sum_{j\in\nn}\mathbf{1}_{B_{i,j}}\r\|_{X}.
\end{equation}

By \eqref{EqI123}, \eqref{Eq414}, \eqref{Eq415} and \eqref{Eq416}, we conclude that
\begin{align*}
\|f\|_{WH_X(\rn)}&=\sup_{\alpha\in(0,\infty)}\lf\{
\alpha\lf\|\mathbf{1}_{\{x\in\rn:\ M_N^0(f)(x)>\alpha\}}\r\|_{X}\r\}
\lesssim\sup_{\alpha\in(0,\infty)}\alpha(\mathrm{I_1}+\mathrm{I_2}+\mathrm{I_3})
\lesssim\sup_{i\in\zz}2^i\lf\|\sum_{j\in\nn}\mathbf{1}_{B_{i,j}}\r\|_{X},
\end{align*}
which completes the proof of Theorem \ref{Thar}.
\end{proof}

\section{Molecular characterizations\label{s5}}

In this section, we establish the molecular characterization of $WH_X(\rn)$.
We begin with recalling the notion of molecules (see \cite[Definition 3.8]{SHYY}).

\begin{definition}\label{Demol}
Let $X$ be a ball quasi-Banach function space, $\epsilon\in(0,\infty)$, $q\in[1,\infty]$ and $d\in\zz_+$.
A measurable function $m$ is called an $(X,\,q,\,d,\,\epsilon)$-\emph{molecule} associated with some ball $B\subset\rn$ if
\begin{enumerate}
\item[(i)] for any $j\in\nn$, $\|m\|_{L^q(S_j(B))}\le2^{-j\epsilon}|S_j(B)|^\frac1q\|\mathbf{1}_B\|_{X}^{-1}$,
where $S_0:=B$ and, for any $j\in\nn$, $S_j(B):=(2^jB)\setminus(2^{j-1}B)$;
\item[(ii)] $\int_\rn m(x)x^\beta\,dx=0$ for any $\beta\in\zz_+^n$ with $|\beta|\le d$.
\end{enumerate}
\end{definition}

\begin{theorem}\label{Thmolcha0}
Let $X$ and $p$ be the same as in Theorem \ref{Thad}.
Let $d\geq \lfloor n(1/p-1)\rfloor$ be a fixed nonnegative integer, $\epsilon\in (n+d+1,\infty)$ and $f\in WH_X(\rn)$.
Then $f$ can be decomposed into
$$f=\sum_{i\in\zz}\sum_{j\in\nn}\lambda_{i,j}m_{i,j}\quad\text{in}\quad \cs'(\rn),$$
where $\{m_{i,j}\}_{i\in\zz,j\in\nn}$
is a sequence of $(X,\,\infty,\,d,\,\epsilon)$-molecules
associated, respectively, with balls $\{B_{i,j}\}_{i\in\zz,j\in\nn}$ and
$\{\lambda_{i,j}\}_{i\in\zz,j\in\nn}:=\{\widetilde A2^i\|\mathbf{1}_{B_{i,j}}\|_{X}\}_{i,\in\zz,j\in\nn}$ with $\widetilde A$
being a positive constant independent of $f$, $i$ and $j$, and there exist positive constants $A$ and $c$ such that,
for any $i\in\zz$, $\sum_{j\in\nn}\mathbf{1}_{cB_{i,j}}\le A$. Moreover,
$$
\sup_{i\in\zz}\lf\|\sum_{j\in\nn}
\frac{\lambda_{i,j}\mathbf{1}_{B_{i,j}}}{\|\mathbf{1}_{B_{i,j}}\|_{X}}\r\|_{X}
\lesssim\|f\|_{WX},
$$
where the implicit positive constant is independent of $f$.
\end{theorem}
\begin{proof}
Observe that every $(X,\,\infty,\,d)$-atom is also an $(X,\,\infty,\,d,\,\epsilon)$-molecule.
Thus, Theorem \ref{Thmolcha0} is a direct corollary of Theorem \ref{Thad}, which completes
the proof of Theorem \ref{Thmolcha0}.
\end{proof}

\begin{theorem}\label{Thmolcha}
Let $X$ be a ball quasi-Banach function space satisfying
Assumption \ref{a2.15} for some $p_-\in(0,\infty)$.
Assume that, for any given $r\in(0,\underline{p})$
with $\underline{p}$ as in \eqref{Eqpll},
$X^{1/r}$ is a ball Banach function space and
assume that there exist $p_+\in[p_-,\infty)$ such that, for any given $r\in(0,\underline{p})$
and $p\in(p_+,\infty)$, and any $f\in(X^{1/r})'$,
\begin{equation*}
\lf\|\cm^{((p/r)')}(f)\r\|_{(X^{1/r})'}\le C\lf\|f\r\|_{(X^{1/r})'},
\end{equation*}
where the positive constant $C$ is independent of $f$.
Let $d\in\zz_+$ with $d\geq \lfloor n(1/\underline{p}-1)\rfloor$.
Let $q\in(\max\{p_+,1\},\infty]$, $\epsilon\in(n+d+1,\infty)$,
$A,\ \widetilde A\in(0,\infty)$ and $c\in(0,1]$, and let
$\{m_{i,j}\}_{i\in\zz,j\in\nn}$ be a
sequence of $(X,\,q,\,d,\,\epsilon)$-molecules
associated, respectively, with balls $\{B_{i,j}\}_{i\in\zz,j\in\nn}$
satisfying that $\sum_{j\in\nn}\mathbf{1}_{cB_{i,j}}\le A$ for any $i\in\zz$, $\{\lambda_{i,j}\}_{i\in\zz,j\in\nn}:=\{\widetilde
A2^i\|\mathbf{1}_{B_{i,j}}\|_{X}\}_{i\in\zz,j\in\nn}$,
$$
\sup_{i\in\zz}\lf\|\sum_{j\in\nn}
\frac{\lambda_{i,j}\mathbf{1}_{B_{i,j}}}{\|\mathbf{1}_{B_{i,j}}\|_{X}}\r\|_{X}<\infty
$$
and the series
$f:=\sum_{i\in \zz}\sum_{j\in\nn}\lambda_{i,j}m_{i,j}$
converges in $\cs'(\rn)$.
Then $f\in WH_X(\rn)$ and
$$
\|f\|_{WH_X(\rn)}\lesssim\sup_{i\in\zz}\lf\|\sum_{j\in\nn}
\frac{\lambda_{i,j}\mathbf{1}_{B_{i,j}}}{\|\mathbf{1}_{B_{i,j}}\|_{X}}\r\|_{X},
$$
where the implicit positive constant is independent of $f$.
\end{theorem}

\begin{proof}
Let $m$ be any given $(X,\,q,\,d,\,\epsilon)$-molecule associated
with some ball $B:=B(x_B,r_B)$, where $x_B\in\rn$
and $r_B\in(0,\infty)$. Without loss of generality, we may assume that the center of the ball is the origin.
Then we claim that $m$ is an infinite linear combination of $(X,\,q,\,d)$-atoms both pointwisely
and in $\mathcal{S}'(\rn)$.

To show this, for any $k\in\zz_+$, let $m_k:=m\mathbf{1}_{S_k(B)}$
with $S_k(B)$ as in Definition \ref{Demol}(i), and $\cp_k$
be the linear vector space generated by the set $\{x^\gamma\mathbf{1}_{S_k(B)}\}_{|\gamma|\le d}$ of ``polynomial".
For any given $k\in\zz_+$, we know that there exists a unique polynomial $P_k\in\cp_k$
such that, for any multi-index $\beta$ with $|\beta|\le d$,
\begin{equation}\label{Eq51}
\int_\rn x^\beta\lf[m_k(x)-P_k(x)\r]\,dx=0,
\end{equation}
where $P_k$ is defined by setting
\begin{equation}\label{Eq52}
P_k:=\sum_{\beta\in\zz_+^n,|\beta|\le d}\lf[\frac1{|S_k(B)|}\int_\rn y^\beta m_k(y)\,dy\r]Q_{\beta,k}
\end{equation}
and, for any $\beta\in\zz_+^n$ and $|\beta|\leq d$, $Q_{\beta,k}$ is the unique polynomial
in $\cp_k(\rn)$ satisfying that, for any multi-indices $\beta$ and $\gamma$ with
$|\beta|\le d$ and $|\gamma|\leq d$,

\begin{equation}\label{Eq53}
\int_\rn x^\gamma Q_{\beta,k}(x)\,dx=|S_k(B)|\delta_{\gamma,\beta},
\end{equation}
where $\delta_{\gamma,\beta}$ denotes the \emph{Kronecker delta}, namely, when $\gamma=\beta$, $\delta_{\gamma,\beta}:=1$ and,
when $\gamma\neq\beta$, $\delta_{\gamma,\beta}:=0$ (see, for instance, \cite[p.\,77]{TW}).

Using the polynomials $\{P_k\}_{k=0}^\infty$, we decompose
$$
m=\sum_{k=0}^\infty m_k=\sum_{k=0}^\infty (m_k-P_k)+\sum_{k=0}^\infty P_k
$$
pointwisely. First we show that $\sum_{k=0}^\infty (m_k-P_k)$
can be divided into an infinite linear combination of $(X,\,q,\,d)$-atoms.
For any $k\in\zz_+$, obviously, $\supp(m_k-P_k)\subset S_k(B)$ and it was proved in \cite[p.\,83]{TW} that
$$
\sup_{x\in S_k(B)}|P_k(x)|\lesssim\frac1{|S_k(B)|}\|m_k\|_{L^1(\rn)},
$$
which, together the H\"older inequality and Definition \ref{Demol}(i), implies that
\begin{equation}\label{eq5.4}
\lf\|m_k-P_k\r\|_{L^q(\rn)}\le\lf\|m_k\r\|_{L^q(S_k(B))}+\lf\|P_k\r\|_{L^q(S_k(B))}
\le\widetilde C\lf\|m\r\|_{L^q(S_k(B))}
\le\widetilde C2^{-k\epsilon}\lf|2^kB\r|^\frac1{q}\lf\|\mathbf{1}_B\r\|_{X}^{-1},
\end{equation}
where $\widetilde C$ is a positive constant independent of $m,\ B$ and $k$.

For any $k\in\zz_+$, let
$$
a_k:=\frac{2^{k\epsilon}\|\mathbf{1}_B\|_{X}(m_k-P_k)}{\widetilde C\|\mathbf{1}_{2^kB}\|_{X}}\quad\text{and}\quad
\mu_k:=\widetilde C2^{-k\epsilon}\frac{\|\mathbf{1}_{2^kB}\|_{X}}{\|\mathbf{1}_B\|_{X}}.
$$
By \eqref{eq5.4} and \eqref{Eq51}, it is easy to show that,
for any $k\in\zz_+$, $a_k$ is an $(X,\,q,\,d)$-atom.
Therefore,
\begin{equation}\label{Eq55}
\sum_{k=0}^\infty(m_k-P_k)=\sum_{k=0}^\infty\mu_ka_k
\end{equation}
pointwisely is an infinite linear combination of $(X,\,q,\,d)$-atoms.

Now we prove that $\sum_{k=0}^\infty P_k$ can also be
pointwisely divided into an infinite linear combination of $(X,\,q,\,d)$-atoms.
For any $j\in\zz_+$ and $\ell\in\zz_+^n$, let
$$
N_\ell^j:=\sum_{k=j}^\infty\int_{S_k(B)}m_k(x)x^\ell\,dx.
$$
Then, for any $\ell\in\zz_+^n$ with $|\ell|\le d$, by Definition \ref{Demol}(ii), we have
\begin{equation}\label{Eq56}
N_\ell^0=\sum_{k=0}^\infty\int_{S_k(B)}m_k(x)x^\ell\,dx=\int_\rn m(x)x^\ell\,dx=0.
\end{equation}
Therefore, from the H\"older inequality and the assumption that $\epsilon\in(n+d+1,\infty)$, combined with
Definition \ref{Demol}(i), we deduce that, for any $j\in\zz_+$ and $\ell\in\zz_+^n$ with $|\ell|\le d$,
\begin{align}\label{Eq57}
|N^j_\ell|&\le\sum_{k=j}^\infty\int_{S_k(B)}\lf|m_k(x)x^\ell\r|\,dx\le\sum_{k=j}^\infty
\lf(2^kr_B\r)^{|\ell|}\lf|2^kB\r|^{1/q'}\lf\|m\r\|_{L^q(S_k(B))}\\\noz
&\lesssim\sum_{k=j}^\infty2^{-k(\epsilon-n-|\ell|)}|B|^{1+|\ell|/n}\|\mathbf{1}_B\|_{X}^{-1}
\lesssim2^{-j(\epsilon-n-|\ell|)}|B|^{1+|\ell|/n}\|\mathbf{1}_B\|_{X}^{-1}.
\end{align}
Furthermore, by the proof in \cite[p.\,77]{TW}, we know that, for any $j\in\zz_+$,
$\beta\in\zz_+^n$ with $|\beta|\le d$, $|Q_{\beta,j}|\le(2^jr_B)^{-|\beta|}$, which,
together with
\eqref{Eq57}, implies that, for any $j\in\zz_+$, $\ell\in\zz_+^n$ with $|\ell|\le d$ and $x\in\rn$,
\begin{equation}\label{Eq58}
\lf|S_j(B)\r|^{-1}\lf|N^j_\ell Q_{\ell,j}(x)\mathbf{1}_{S_j(B)}(x)\r|\lesssim 2^{-j\epsilon}\|\mathbf{1}_B\|_{X}^{-1}.
\end{equation}

Moreover, by \eqref{Eq52}, the definition of $N_\ell^j$ and \eqref{Eq56}, we conclude that
\begin{align}\label{Eq59}
\sum_{k=0}^\infty P_k&=\sum_{\ell\in\zz_+^n,|\ell|\le d}
\sum_{k=0}^\infty\lf|S_k(B)\r|^{-1}Q_{\ell,k}\int_\rn m_k(x)x^\ell\,dx\\\noz
&=\sum_{\ell\in\zz_+^n,|\ell|\le d}\sum_{k=0}^\infty
N_\ell^{k+1}\lf[\lf|S_{k+1}(B)\r|^{-1}Q_{\ell,k+1}\mathbf{1}_{S_{k+1}(B)}-
\lf|S_k(B)\r|^{-1}Q_{\ell,k}\mathbf{1}_{S_k(B)}\r]\\\noz
&=:\sum_{\ell\in\zz_+^n,|\ell|\le d}\sum_{k=0}^\infty b_\ell^k
\end{align}
pointwisely. From this, \eqref{Eq58} and \eqref{Eq53},
it follows that there exists a positive constant $C_0$ such that, for any
$k\in\zz_+$ and $\ell\in\zz_+^n$ with $|\ell|\le d$,
\begin{equation}\label{eq510}
\lf\|b_\ell^k\r\|_{L^\infty(\rn)}\le C_02^{-k\epsilon}
\lf\|\mathbf{1}_B\r\|_{X}^{-1}\quad\text{and}\quad\supp b_\ell^k\subset2^{k+1}B;
\end{equation}
moreover, for any $\gamma\in\zz_+^n$ with $|\gamma|\le d$, $\int_\rn b_\ell^k(x)x^\gamma\,dx=0$.
For any $k\in\zz_+$ and $\ell\in\zz_+^n$ with $|\ell|\le d$, let
$$
\mu_\ell^k:=2^{-k\epsilon}\frac{\|\mathbf{1}_{2^{k+1}B}\|_{X}}{\|\mathbf{1}_B\|_{X}}\quad\text{and}\quad
a_\ell^k:=2^{k\epsilon}b_\ell^k\frac{\|\mathbf{1}_{B}\|_{X}}{\|\mathbf{1}_{2^{k+1}B}\|_{X}}.
$$
By \eqref{Eq53} and the definitions of $b^k_\ell$ and $a^k_\ell$, we find that, for any $\gamma\in\zz_+^n$
with $|\gamma|\le d$, $\int_\rn a_\ell^k(x)x^\gamma\,dx = 0$. Obviously, $\supp (a^k_\ell)\subset 2^{k+1}B$.
Thus, $a^k_\ell$ is an $(X,\,\infty,\,d)$-atom and hence an $(X,\,q,\,d)$-atom up to a positive constant multiple.
Moreover, we find that
\begin{equation}\label{Eq511}
\sum_{k=0}^\infty P_k=\sum_{\ell\in\zz_+^n,|\ell|\le d}\sum_{k=0}^\infty\mu_\ell^k a_\ell^k
\end{equation}
pointwisely forms an infinite linear combination of $(X,\,q,\,d)$-atoms.

Combining \eqref{Eq55} and \eqref{Eq511}, we obtain
\begin{equation}\label{Eq512}
m=\sum_{k=0}^\infty m_k=\sum_{k=0}^\infty(m_k-P_k)+\sum_{k=0}^\infty P_k
=\sum_{k=0}^\infty\mu_k a_k
+\sum_{\ell\in\zz_+^n,|\ell|\le d}\sum_{k=0}^\infty\mu_\ell^k a_\ell^k\quad \mathrm{pointwisely,}
\end{equation}
which shows that any $(X,\,q,\,d,\epsilon)$-molecule
is an infinite linear combination of $(X,\,q,\,d)$-atoms both pointwisely
and in $\mathcal{S}'(\rn)$. Therefore, we have proved the above claim.

To show $f\in WH_X(\rn)$, it suffices to prove that, for any $\alpha\in(0,\infty)$,
\begin{equation}\label{Eq514}
\alpha\lf\|\mathbf{1}_{\{x\in\rn:\ M_N^0(f)(x)>\alpha\}}\r\|_{X}\lesssim
\sup_{i\in\zz}\lf\|\sum_{j\in\nn}
\frac{\lambda_{i,j}\mathbf{1}_{B_{i,j}}}{\|\mathbf{1}_{B_{i,j}}\|_{X}}\r\|_{X},
\end{equation}
where the implicit positive constant is independent of $f$ and $\alpha$.

For any given $\alpha\in(0,\infty)$, we know that there exists an $i_0\in\zz$ such that $2^{i_0}\le\alpha<2^{i_0+1}$.
Then we decompose $f$ into
$$
f=\sum_{i=-\infty}^{i_0-1}\sum_{j\in\nn}
\lambda_{i,j}m_{i,j}+\sum_{i=i_0}^{\infty}\sum_{j\in\nn}\lambda_{i,j}m_{i,j}=:f_1+f_2.
$$
By the fact that $\mathbf{1}_{\{x\in\rn:\ M_N^0(f)(x)>\alpha\}}\leq\mathbf{1}_{\{x\in\rn:\ M_N^0(f_1)(x)>\alpha/2\}}+\mathbf{1}_{\{x\in\rn:\ M_N^0(f_2)(x)>\alpha/2\}}$
and Definition \ref{Debqfs}(ii), we have
\begin{equation}\label{Eq515}
\lf\|\mathbf{1}_{\{x\in\rn:\ M_N^0(f)(x)>\alpha\}}\r\|_{X}\lesssim
\lf\|\mathbf{1}_{\{x\in\rn:\ M_N^0(f_1)(x)>\alpha/2\}}\r\|_{X}+
\lf\|\mathbf{1}_{\{x\in\rn:\ M_N^0(f_2)(x)>\alpha/2\}}\r\|_{X}=:\mathrm{I_1}+\mathrm{I_2}.
\end{equation}

To deal with $\mathrm{I_1}$, we first need an estimate of $M_N^0(m_{i,j})$.
By \eqref{eq5.4}, \eqref{eq510} and \eqref{Eq512},
for any $i\in\zz$ and $j\in\nn$, we have a sequence of multiples of $(X,\,q,\,d)$-atoms,
$\{a_{i,j}^l\}_{l\in\zz_+}$, supported, respectively, in balls $\{2^{l+1}B_{i,j}\}_{l\in\zz_+}$ such that
$$
\lf\|a_{i,j}^l\r\|_{L^q(\rn)}\lesssim\frac{2^{-l\epsilon}|2^{l+1}B_{i,j}|^{1/q}}{\|\mathbf{1}_{B_{i,j}}\|_{X}}
$$
and $m_{i,j}=\sum_{l\in\zz_+}a_{i,j}^l$ pointwisely in $\rn$. Then, for any $i\in\zz\cap(-\infty,i_0-1]$ and
$j\in\nn$,
\begin{equation}\label{Eq516}
M_N^0(m_{i,j})\le\sum_{l\in\zz_+}M_N^0(a_{i,j}^l)=
\sum_{l\in\zz_+}\sum_{k\in\zz_+}M_N^0(a_{i,j}^l)\mathbf{1}_{S_k(2^lB_{i,j})}
=:\sum_{l\in\zz_+}\sum_{k\in\zz_+}\mathrm{J}_{l,k},
\end{equation}
where $S_k(2^lB_{i,j}):=(2^{k+l}B_{i,j})\setminus(2^{k+l-1}B_{i,j})$. From this, we deduce that
\begin{align}\label{Eq517}
\mathrm{I_1}&\lesssim\lf\|\mathbf{1}_{\{x\in\rn:\
\sum_{i=-\infty}^{i_0-1}\sum_{j\in\nn}\lambda_{i,j}M_N^0(m_{i,j})(x)>\frac\alpha2\}}\r\|_{X}\\\noz
&\lesssim\lf\|\mathbf{1}_{\{x\in\rn:\
\sum_{i=-\infty}^{i_0-1}\sum_{j\in\nn}\sum_{l\in\zz_+}
\sum_{k=0}^2\lambda_{i,j}\mathrm{J}_{l,k}(x)>\frac\alpha2\}}\r\|_{X}
+\lf\|\mathbf{1}_{\{x\in\rn:\
\sum_{i=-\infty}^{i_0-1}\sum_{j\in\nn}\sum_{l\in\zz_+}
\sum_{k=3}^\infty\lambda_{i,j}\mathrm{J}_{l,k}(x)>\frac\alpha2\}}\r\|_{X}\\\noz
&=:\mathrm{I_{1,1}}+\mathrm{I_{1,2}}.
\end{align}
For $\mathrm{I_{1,1}}$, by a similar argument to that used in the estimation of \eqref{Eq411},
we obtain
\begin{equation}\label{Eq518}
\alpha\mathrm{I_{1,1}}\lesssim\sup_{i\in\zz}\lf\|\sum_{j\in\nn}
\frac{\lambda_{i,j}\mathbf{1}_{B_{i,j}}}{\|\mathbf{1}_{B_{i,j}}\|_{X}}\r\|_{X}.
\end{equation}
For $\mathrm{I_{1,2}}$, we first estimate every term $\mathrm{J}_{l,k}$.
By an argument similar to that used in the estimation of \eqref{EqMNa}, we conclude that,
for any $i\in\zz$, $j\in\nn$, $l\in\zz_+$, $k\in[3,\infty)\cap\zz_+$ and $x\in S_k(2^lB_{i,j})$,
\begin{align}\label{Eq519}
\mathrm{J}_{l,k}(x)&\lesssim\int_{2^{l+1}B_{i,j}}\frac{|y-x_{i,j}|^{d+1}}{|x-x_{i,j}|^{n+d+1}}
|a_{i,j}^l(y)|\,dy\mathbf{1}_{S_k(2^lB_{i,j})}(x)\\\noz
&\lesssim\frac{(2^{l+1}r_{i,j})^{d+1}}{(2^{l}r_{i,j})^{n+d+1}}\|a_{i,j}^l\|_{L^q(\rn)}
|2^{l+1}B_{i,j}|^{1/q'}\mathbf{1}_{S_k(2^lB_{i,j})}(x)\\\noz
&\lesssim\frac{2^{-l(n+\epsilon)-k(n+d+1)}}{r_{i,j}^n
\|\mathbf{1}_{B_{i,j}}\|_{X}}|2^{l+1}B_{i,j}|\mathbf{1}_{S_k(2^lB_{i,j})}(x)
\sim\frac{2^{-l\epsilon-k(n+d+1)}}
{\|\mathbf{1}_{B_{i,j}}\|_{X}}\mathbf{1}_{S_k(2^lB_{i,j})}(x).
\end{align}
Let $r\in(\frac{n}{n+d+1},\underline{p})$. By Definition \ref{Debf}(i), we have
$\|\mathbf{1}_{\{x\in\rn:\ M_N^0(f_2)(x)>\alpha\}}\|_{X}=\|\mathbf{1}_{\{x\in\rn:\ M_N^0(f_2)(x)>\alpha\}}\|_{X^{1/r}}^{1/r}$.
This, together with Definition \ref{Debqfs}(ii), \eqref{Eq516}, \eqref{Eq519}, the assumption that $X^{1/r}$ is a ball Banach function space, \eqref{EqHLMS} and the fact that $\epsilon\in(n+d+1,\infty)$, implies that
\begin{align*}
\alpha\mathrm{I_{1,2}}&\lesssim\alpha^{1-1/r}
\lf\|\sum_{i=-\infty}^{i_0-1}\sum_{j\in\nn}\sum_{l\in\zz_+}\sum_{k=3}^\infty
2^i2^{-l\epsilon}2^{-k(n+d+1)}\mathbf{1}_{S_k(2^lB_{i,j})}\r\|_{X^{1/r}}^{1/r}\\
&\lesssim\alpha^{1-1/r}\lf[\sum_{l\in\zz_+}\sum_{k=3}^\infty2^{-l\epsilon}2^{-k(n+d+1)}
\sum_{i=-\infty}^{i_0-1}2^i\lf\|\sum_{j\in\nn}\mathbf{1}_{S_k(2^lB_{i,j})}\r\|_{X^{1/r}}\r]^{1/r}\\
&\lesssim\alpha^{1-1/r}\lf[\sum_{l\in\zz_+}\sum_{k=3}^\infty2^{-l\epsilon}2^{-k(n+d+1)}2^{\frac{n(k+l)}{r}}
\sum_{i=-\infty}^{i_0-1}2^i\lf\|\sum_{j\in\nn}\mathbf{1}_{B_{i,j}}\r\|_{X^{1/r}}\r]^{1/r}\\
&\lesssim\alpha^{1-1/r}\sup_{i\in\zz}2^i\lf\|\sum_{j\in\nn}\mathbf{1}_{B_{i,j}}\r\|_{X}
\lf[\sum_{i=-\infty}^{i_0-1}2^{i(1-r)}\r]^{1/r}
\lesssim\sup_{i\in\zz}2^i\lf\|\sum_{j\in\nn}\mathbf{1}_{B_{i,j}}\r\|_{X}.
\end{align*}
By this, \eqref{Eq517} and \eqref{Eq518}, we conclude that
\begin{equation}\label{Eq520}
\alpha\mathrm{I_1}\lesssim\sup_{i\in\zz}2^i\lf\|\sum_{j\in\nn}\mathbf{1}_{B_{i,j}}\r\|_{X}.
\end{equation}

Next we turn to estimate $\mathrm{I_2}$. To this end,
by \eqref{Eq516} and Definition \ref{Debqfs}(ii), we know that
\begin{align}\label{Eq521}
\mathrm{I_2}&\lesssim\lf\|\mathbf{1}_{\{x\in\rn:\ \sum_{i=i_0}^\infty\sum_{j\in\nn}\sum_{l\in\zz_+}
\sum_{k=0}^2\lambda_{i,j}\mathrm{J}_{l,k}(x)>\frac\alpha4\}}\r\|_{X}
+\lf\|\mathbf{1}_{\{x\in\rn:\ \sum_{i=i_0}^\infty\sum_{j\in\nn}\sum_{l\in\zz_+}
\sum_{k=3}^\infty\lambda_{i,j}\mathrm{J}_{l,k}(x)>\frac\alpha4\}}\r\|_{X}\\\noz
&=:\mathrm{I_{2,1}}+\mathrm{I_{2,2}}.
\end{align}

We first deal with $\mathrm{I_{2,1}}$. For any $\widetilde q\in(0,1)$, we have
$$
\sum_{i=i_0}^\infty\sum_{j\in\nn}\sum_{l\in\zz_+}
\sum_{k=0}^2\lambda_{i,j}M_N^0(a_{i,j})\mathbf{1}_{S_k(2^lB_{i,j})}
\le \lf\{\sum_{i=i_0}^\infty\sum_{j\in\nn}\sum_{l\in\zz_+}\sum_{k=0}^2
\lf[\lambda_{i,j}M_N^0(a_{i,j})\mathbf{1}_{S_k(2^lB_{i,j})}\r]
^{\widetilde q}\r\}^{1/\widetilde q}.
$$
Let $r\in(\frac{n}{n+d+1},\underline{p})$ and choose $\widetilde q\in(0,1)$ such that
$r\widetilde q>\frac{n}{n+d+1}$.
By Definition \ref{Debqfs}(ii),
$\lambda_{i,j}:=\widetilde A2^i\|\mathbf{1}_{B_{i,j}}\|_{X}$,
and the assumption that $X^{1/r}$ is a ball Banach function space,
we conclude that
\begin{align*}
\mathrm{I_{2,1}}&\lesssim2^{-i_0\widetilde q}
\lf\|\sum_{i=i_0}^\infty\sum_{j\in\nn}\sum_{l\in\zz_+}\sum_{k=0}^2
\lf[\lambda_{i,j}M_N^0(a_{i,j})\mathbf{1}_{S_k(2^lB_{i,j})}\r]^{\widetilde q}\r\|_X\\
&\sim2^{-i_0\widetilde q}\lf\|\sum_{i=i_0}^\infty
2^{i\widetilde q}\sum_{l\in\zz_+}\sum_{k=0}^2\sum_{j\in\nn}
\lf[\|\mathbf{1}_{B_{i,j}}\|_XM_N^0(a_{i,j})\mathbf{1}_{S_k(2^lB_{i,j})}\r]^{\widetilde q}\r\|_X\\
&\lesssim2^{-i_0\widetilde q}\lf\{\sum_{i=i_0}^\infty
2^{i\widetilde q}\sum_{l\in\zz_+}2^{-lr\epsilon\widetilde q}
\sum_{k=0}^2\lf\|\sum_{j\in\nn}
\lf[2^{l\epsilon}\|\mathbf{1}_{B_{i,j}}\|_XM_N^0(a_{i,j})\mathbf{1}_{S_k(2^lB_{i,j})}
\r]^{r\widetilde q}\r\|_{X^{1/r}}\r\}^{1/r}\\
&\sim2^{-i_0\widetilde q}\lf\{\sum_{i=i_0}^\infty2^{i\widetilde q}\sum_{l\in\zz_+}2^{-lr\epsilon\widetilde q}\sum_{k=0}^2
\lf\|\lf\{\sum_{j\in\nn}
\lf[2^{l\epsilon}\|\mathbf{1}_{B_{i,j}}\|_XM_N^0(a_{i,j})\mathbf{1}_{S_k(2^lB_{i,j})}
\r]^{r_0\widetilde q}\r\}^{1/r}\r\|_{X}^{r}\r\}^{1/r}.
\end{align*}
Let $p_0:=q/\widetilde q$. Then $p_0\in(p_+,\infty)$ and, from this and
the boundedness of $\cm$ on $L^q(\rn)$,
we deduce that, for any $i\in\zz$, $j\in\nn$ and $k\in\{0,1,2\}$,
$$
\lf\|\lf[2^{l\epsilon}\lf\|\mathbf{1}_{B_{i,j}}\r\|_X
M_N^0(a_{i,j})\mathbf{1}_{S_k(2^lB_{i,j})}\r]^{\widetilde q}\r\|_{L^{p_0}(\rn)}
\lesssim2^{l\epsilon\widetilde q}\lf\|\mathbf{1}_{B_{i,j}}\r\|_X^{\widetilde q}
\lf\|M_N^0(a_{i,j})\r\|_{L^q(\rn)}^{\widetilde q}\lesssim\lf|\mathbf{1}_{B_{i,j}}\r|^{1/p_0}.
$$
By Lemma \ref{Le45}, \eqref{EqHLMS},
the fact that $\sum_{j\in\nn}\mathbf 1_{cB_{i,j}}\le A$ and $\widetilde q\in(0,1)$, we further conclude that
\begin{align*}
\mathrm{I_{2,1}}&\lesssim2^{-i_0\widetilde q}\lf\{\sum_{i=i_0}^\infty2^{i\widetilde q}\sum_{l\in\zz_+}2^{-lr\epsilon\widetilde q}
\lf\|\lf\{\sum_{j\in\nn}
\mathbf{1}_{2^{l+1}B_{i,j}}\r\}^{1/r}\r\|_{X}^{r}\r\}^{1/r}\\
&\lesssim2^{-i_0\widetilde q}\lf\{\sum_{i=i_0}
^\infty2^{i\widetilde q}\sum_{l\in\zz_+}2^{-lr\epsilon\widetilde q}
2^{(l+1)n}\lf\|\lf\{\sum_{j\in\nn}
\mathbf{1}_{cB_{i,j}}\r\}^{1/r}\r\|_{X}^{r}\r\}^{1/r}\\
&\lesssim2^{-i_0\widetilde q}\lf[\sum_{i=i_0}^\infty2^{ir(\widetilde q-1)}\r]^{1/r}\sup_{i}2^i\lf\|\sum_{j\in\nn}\mathbf{1}_{B_{i,j}}\r\|_X
\lesssim\alpha^{-1}\sup_{i}2^i\lf\|\sum_{j\in\nn}\mathbf{1}_{B_{i,j}}\r\|_X,
\end{align*}
which implies that
\begin{equation}\label{Eq522}
\alpha\mathrm{I_{2,1}}\lesssim\sup_{i}2^i\lf\|\sum_{j\in\nn}\mathbf{1}_{B_{i,j}}\r\|_X.
\end{equation}

To estimate $\mathrm{I_{2,2}}$, for any $a\in(0,1)$, we also have
\begin{align}\label{Eq523}
\sum_{i=i_0}^\infty\sum_{j\in\nn}\sum_{l\in\zz_+}\sum_{k=3}^\infty\lambda_{i,j}\mathrm{J}_{l,k}
\le\lf[\sum_{i=i_0}^\infty\sum_{j\in\nn}\sum_{l\in\zz_+}\sum_{k=3}^\infty\lf(\lambda_{i,j}
\mathrm{J}_{l,k}\r)^a\r]^{1/a}.
\end{align}
Then, by Definition \ref{Debqfs}(ii), \eqref{Eq523}, \eqref{Eq519} and $\lambda_{i,j}:=\widetilde A2^i\|\mathbf{1}_{B_{i,j}}\|_{X}$, we know that
\begin{align*}
\mathrm{I_{2,2}}&\le\lf\|\mathbf{1}_{\{x\in\rn:\
\sum_{i=i_0}^\infty\sum_{j\in\nn}\sum_{l\in\zz_+}\sum_{k=3}^\infty
[\lambda_{i,j}\mathrm{J}_{l,k}(x)]^a>2^{i_0a}\}}\r\|_{X}
\lesssim2^{-i_0a}\lf\|\sum_{i=i_0}^\infty
\sum_{j\in\nn}\sum_{l\in\zz_+}\sum_{k=3}^\infty\lf(\lambda_{i,j}\mathrm{J}_{l,k}\r)^a\r\|_{X}\\
&\lesssim2^{-i_0a}\lf\|\sum_{i=i_0}^\infty2^{ia}
\sum_{l\in\zz_+}\sum_{k=3}^\infty 2^{-la\epsilon}2^{-ka(n+d+1)}\sum_{j\in\nn}\mathbf{1}_{S_k(2^lB_{i,j})}\r\|_{X}.
\end{align*}
Let $r\in(\frac{n}{n+d+1},\underline{p})$. We choose $a\in(0,1)$ such that $ar>\frac{n}{n+d+1}$. From Definition \ref{Debf}(i),
the assumption that $X^{1/r}$ is a ball Banach function space, \eqref{EqHLMS}, $ar(n+d+1)-n>0$ and $\varepsilon>n+d+1$,
we further deduce that
\begin{align*}
\mathrm{I_{2,2}}&\lesssim2^{-i_0a}\lf[\sum_{i=i_0}^\infty2^{ira}
\sum_{l\in\zz_+}\sum_{k=3}^\infty 2^{-lar\epsilon}2^{-kar(n+d+1)}
\lf\|\lf\{\sum_{j\in\nn}\mathbf{1}_{S_k(2^lB_{i,j})}\r\}^r\r\|_{X^{1/r}}\r]^{\frac1r}\\
&\lesssim2^{-i_0a}\lf[\sum_{i=i_0}^\infty2^{ira}
\lf\{\sum_{l\in\zz_+}\sum_{k=3}^\infty 2^{l(n-ar\epsilon)}2^{k[n-ar(n+d+1)]}\r\}
\lf\|\sum_{j\in\nn}\mathbf{1}_{B_{i,j}}\r\|_{X}^{r}\r]^{\frac1r}
\lesssim2^{-i_0a}\lf[\sum_{i=i_0}^\infty2^{ira}
\lf\|\sum_{j\in\nn}\mathbf{1}_{B_{i,j}}\r\|_{X}^r\r]^{\frac1r}\\
&\lesssim2^{-i_0a}\lf[\sum_{i=i_0}^\infty2^{-ir(1-a)}\r]^{\frac1r}
\sup_{i\in\zz}2^i\lf\|\sum_{j\in\nn}\mathbf{1}_{B_{i,j}}\r\|_{X}
\lesssim\alpha^{-1}\sup_{i\in\zz}2^i\lf\|\sum_{j\in\nn}\mathbf{1}_{B_{i,j}}\r\|_{X},
\end{align*}
which implies that
$$
\alpha\mathrm{I_{2,2}}\lesssim\sup_{i\in\zz}2^i\lf\|\sum_{j\in\nn}\mathbf{1}_{B_{i,j}}\r\|_{X}.
$$
This, combined with \eqref{Eq522} and \eqref{Eq521}, implies that
$$
\alpha\mathrm{I_2}\lesssim\sup_{i\in\zz}2^i\lf\|\sum_{j\in\nn}\mathbf{1}_{B_{i,j}}\r\|_{X}.
$$
By this, \eqref{Eq515} and \eqref{Eq520},
we know that \eqref{Eq514} holds true and hence complete the proof of Theorem \ref{Thmolcha}.
\end{proof}

\section{Boundedness of Calder\'on--Zygmund operators\label{s6}}

In this section, as an application of the weak Hardy type space $WH_X(\rn)$,
we establish the boundedness of Calder\'on--Zygmund operators from the
Hardy type space $H_X(\rn)$ to $WH_X(\rn)$. We begin with recalling the notion
of the Hardy type space $H_X(\rn)$ (see \cite[Definition 2.22]{SHYY}).

\begin{definition}\label{DeHX}
Let $X$ be a ball quasi-Banach function space.
The \emph{Hardy space} $H_X(\rn)$ associated with $X$ is defined to be the set of all $f\in\cs'(\rn)$ such that
$$
\lf\|f\r\|_{H_X(\rn)}:=\lf\|M_b^{**}(f,\psi)\r\|_{X}<\infty,
$$
where $M_b^{**}(f,\psi)$ is as in Definition \ref{Dem}(iii) with $b$ sufficiently large and $\psi\in\cs(\rn)$
satisfying $\int_\rn\psi(x)\,dx\neq0$.
\end{definition}

In what follows, we assume that the ball quasi-Banach function space $X$ satisfies the following assumption:
For some $\theta,\ s\in(0,1]$, there exists a positive constant $C$ such that, for any $\{f_j\}_{j=1}^\infty\subset\mathscr M(\rn)$,
\begin{equation}\label{aa}
\lf\|\lf\{\sum_{j=1}^\infty\lf[\cm^{(\theta)}(f_j)\r]^s\r\}^{1/s}\r\|_X\le C\lf\|\lf\{\sum_{j=1}^\infty|f_j|^s\r\}^{1/s}\r\|_X.
\end{equation}

Let $X$ be a ball quasi-Banach function space satisfying \eqref{aa} for some $\theta,\ s\in(0,1]$.
Let $d\geq\lfloor n(1/\theta-1)\rfloor$ be a fixed integer and  $q\in(1,\infty]$.
Assume that, for any $f\in\mathscr M(\rn)$,
\begin{equation}\label{aax}
\lf\|\cm^{((q/s)')}(f)\r\|_{(X^{1/s})'}\lesssim\lf\|f\r\|_{(X^{1/s})'},
\end{equation}
where the implicit positive constant is independent of $f$. The \emph{atomic Hardy space} $H^{X,q,d}_{\atom}(\rn)$ is defined to be the set of
all $f\in\cs'(\rn)$ such that $f=\sum_{j\in\nn}\lambda_ja_j$ in $\cs'(\rn)$, where $\{\lambda_j\}_{j\in\nn}$
is a sequence of non-negative numbers and $\{a_j\}_{j\in\nn}$ is a sequence of $(X,q,d)$ atoms as in Definition \ref{Deatom}, and
$$
\|f\|_{H^{X,q,d}_{\atom}(\rn)}:=\inf\lf\{\lf\|\lf\{\sum_{j\in\nn}
\lf[\frac{\lambda_j\mathbf{1}_{B_j}}{\|\mathbf{1}_{B_j}\|_{X}}\r]^{s}
\r\}^\frac1{s}\r\|_{X}\r\}<\infty,
$$
where the infimum is taken over all the decompositions of $f$ as above.

The following atomic characterization of $H_X(\rn)$ comes
from \cite[Theorems 3.6 and 3.7]{SHYY}.
\begin{lemma}\label{Le72}
Let $\theta,\ s\in(0,1]$, $q\in(1,\infty]$ and $d\geq\lfloor n(1/\theta-1)\rfloor$ be a fixed integer.
Assume that $X$ is a ball quasi-Banach function space satisfying \eqref{aa},
\eqref{aax} and that $X^{1/s}$ is a ball Banach function space.
Then $H_X(\rn)=H^{X,q,d}_{\atom}(\rn)$ with equivalent quasi-norms.
\end{lemma}

Next, let us recall the notion of absolutely continuous quasi-norms;
see, for instance, \cite[Definition 2.5]{SHYY}.

\begin{definition}
For a ball quasi-Banach function space $X$, its quasi-norm $\|\cdot\|_X$ is called an
\emph{absolutely continuous quasi-norm} if
$\|\mathbf{1}_{E_j}\|_X\downarrow0$ whenever $\{E_j\}_{j=1}^\infty$
is a sequence of measurable sets satisfying
$E_j\supset E_{j+1}$ for any $j\in\nn$ and $\bigcap_{j=1}^\infty E_j=\emptyset$.
\end{definition}
\begin{remark}\label{Re73}
Let $q$ and $X$ be as in Lemma \ref{Le72}. Assume further that $X$ has an absolutely continuous quasi-norm.
Then, by \cite[Remark 3.12]{SHYY},
we know that the subspace $H_X(\rn)\cap L^q(\rn)$ is dense in $H_X(\rn)$.
\end{remark}

Recall that, for any given $\delta\in(0,1]$, a linear operator $T$ is called
a \emph{convolutional $\delta$-type Calder\'on--Zygmund operator} $T$
(see, for instance, \cite{AM1986})
if $T$ is a linear bounded operator on $L^2(\rn)$ with kernel $k\in\cs'(\rn)$ coinciding with a
locally integrable function on $\rn\setminus\{\vec 0_n\}$ and satisfying that, for any
$x,\ y\in\rn$ with
$|x|>2|y|$,
$$
\lf|k(x-y)-k(x)\r|\le C\frac{|y|^\delta}{|x|^{n+\delta}}
$$
and, for any $f\in L^2(\rn)$, $Tf=\text{p.\,v.}\,k\ast f$.

The boundedness from $H_X(\rn)$ to $WH_X(\rn)$ of convolutional
$\delta$-type Calder\'on--Zygmund operators is stated as follows.

\begin{theorem}\label{Thcz}
Let $\theta,\ s,\ \delta\in(0,1]$ and $q\in(1,\infty)$.
Assume that $X$ is a ball quasi-Banach function
space having an absolutely continuous quasi-norm and satisfying \eqref{aa},
\eqref{aax} and Assumption \ref{xinm}. Assume that $X^{1/s}$ is a ball Banach function space.
Let $T$ be a convolutional $\delta$-type Calder\'on--Zygmund
operator. If there exists a positive constant $C_0$ such that,
for any $\alpha\in(0,\infty)$ and any sequence $\{f_j\}_{j\in\nn}\subset\mathscr M(\rn)$,
\begin{equation}\label{Eqfs11}
\alpha\lf\|\mathbf{1}_{\{x\in\rn:\ \{\sum_{j\in\nn}[\cm(f_j)(x)]^\frac{n+\delta}{n}\}
^\frac{n}{n+\delta}>\alpha\}}\r\|_{X^{\frac{n+\delta}{n}}}
\le C_0\lf\|\lf(\sum_{j\in\nn}|f_j|^\frac{n+\delta}{n}\r)
^\frac{n}{n+\delta}\r\|_{X^{\frac{n+\delta}{n}}},
\end{equation}
then $T$ has a unique extension on $H_X(\rn)$ and, moreover,
there exists a positive constant $C$ such that,
for any $f\in H_X(\rn)$,
$$
\lf\|Tf\r\|_{WH_X(\rn)}\le C\lf\|f\r\|_{H_X(\rn)}.
$$
\end{theorem}

\begin{proof}
Let $\theta$, $s$ and $d$ be as in Lemma \ref{Le72} and $f\in H_X(\rn)\cap L^q(\rn)$.
Then, by the proof of \cite[Theorem 3.7]{SHYY}, we find that there
exist a sequence of $(X,q,d)$-atoms, $\{a_j\}_{j\in\nn}$,
supported, respectively, in balls
$\{B_j\}_{j\in\nn}:=\{B(x_j,r_j):\ x_j\in\rn\ \text{and}\ r_j\in(0,\infty)\}_{j\in\nn}$ and
a sequence $\{\lambda_j\}_{j\in\nn}$ of positive constants
such that
\begin{equation}\label{Eq73}
f=\sum_{j\in\nn}\lambda_ja_j\quad\text{in}\quad L^q(\rn)
\end{equation}
and
$$
\lf\|\lf\{\sum_{j\in\nn}\lf[\frac{\lambda_j\mathbf{1}_{B_j}}{\|\mathbf{1}_{B_j}
\|_{X}}\r]^{s}\r\}^\frac1{s}\r\|_{X}
\lesssim\|f\|_{H_X(\rn)}.
$$
From the fact that
$T$ is bounded on $L^q(\rn)$ (see, for instance, \cite[Theorem 5.1]{d}) and \eqref{Eq73},
we deduce that
$$
T(f)=\sum_{j\in\nn}\lambda_jT(a_j)
$$
holds true in $L^q(\rn)$, namely, $Tf$ for any $f\in H_X(\rn)\cap L^q(\rn)$ is well defined.
Let $\psi\in\cs(\rn)$ satisfy $\int_\rn\psi(x)\,dx\neq0$. Then, to prove Theorem \ref{Thcz},
by Assumption \ref{xinm} and Theorem \ref{Thmc}(ii),
we only need to show that, for any $f\in H_X(\rn)$,
\begin{equation}\label{Eq74}
\lf\|M(Tf,\psi)\r\|_{WX}\lesssim\|f\|_{H_X(\rn)},
\end{equation}
where $M(Tf,\psi)$ is as in Definition \ref{Dem}(i) with $f$ replaced by $Tf$.

For any $\alpha\in(0,\infty)$, by Lemma \ref{Leqn}(iii) and Remark \ref{Rews}(i), we have
\begin{align*}
&\alpha\lf\|\mathbf{1}_{\{x\in\rn:\ M(Tf,\psi)(x)>\alpha\}}\r\|_{X}\\
&\hs\le\alpha\lf\|\mathbf{1}_{\{x\in\rn:\ \sum_{j\in\nn}\lambda_j M(Ta_j,\psi)(x)>\alpha\}}\r\|_{X}\\
&\hs\lesssim\alpha\lf\|\mathbf{1}_{\{x\in4B_j:\ \sum_{j\in\nn}
\lambda_j M(Ta_j,\psi)(x)>\frac{\alpha}2\}}\r\|_{X}
+\alpha\lf\|\mathbf{1}_{\{x\in(4B_j)^\complement:\ \sum_{j\in\nn}
\lambda_j M(Ta_j,\psi)(x)>\frac{\alpha}2\}}\r\|_{X}\\
&\hs\lesssim\lf\|\sum_{j\in\nn}\lambda_jM(Ta_j,\psi)\mathbf{1}_{4B_j}\r\|_{X}
+\alpha\lf\|\mathbf{1}_{\{x\in(4B_j)^\complement:\ \sum_{j\in\nn}
\lambda_j M(Ta_j,\psi)(x)>\frac{\alpha}2\}}\r\|_{X}=:\mathrm{I}+\mathrm{II}.
\end{align*}

We first estimate $\mathrm{I}$. Observing that $M(Ta_j,\psi)\lesssim\cm(Ta_j)$ and $a_j\in L^q(\rn)$,
by the fact that
$T$ is bounded on $L^q(\rn)$ (see, for instance, \cite[Theorem 5.1]{d}) and the size condition of $a_j$, we conclude that
$$
\lf\|M(Ta_j,\psi)\r\|_{L^q(\rn)}\lesssim\lf\|\cm(Ta_j)\r\|_{L^q(\rn)}
\lesssim\lf\|Ta_j\r\|_{L^q(\rn)}
\lesssim\lf\|a_j\r\|_{L^q(\rn)}\lesssim\frac{|B_j|^{1/q}}{\|\mathbf{1}_{B_j}\|_{X}},
$$
which, combined with Lemma \ref{Le45}, \eqref{aa} and \cite[Theorem 3.6]{SHYY}, implies that
\begin{align}\label{Eq75}
\mathrm{I}&\lesssim\lf\|\lf\{\sum_{j\in\nn}\lf[\frac{\lambda_j\mathbf{1}_{4B_j}}
{\|\mathbf{1}_{B_j}\|_{X}}\r]^{s}\r\}^\frac1{s}\r\|_{X}
\lesssim\lf\|\lf\{\sum_{j\in\nn}\lf[\frac{\lambda_j\mathbf{1}_{B_j}}
{\|\mathbf{1}_{B_j}\|_{X}}\r]^{s}\r\}^\frac1{s}\r\|_{X}
\lesssim\|f\|_{H_X(\rn)}.
\end{align}

To deal with the term $\mathrm{II}$, for any $t\in(0,\infty)$, let $k^{(t)}:=k\ast\psi_t$ with
$\psi_t(\cdot):=t^{-n}\psi({\cdot}/{t})$. By \cite[p.\,2881]{YYYZ},
we know that $k^{(t)}$ satisfies the same conditions as $k$.
From this, together with the vanishing moments of $a_j$,
the H\"older inequality and the size condition of $a_j$,
we deduce that, for any $x\in(4B_j)^\complement$,
\begin{align*}
\lf|M(Ta_j,\psi)(x)\r|&=\sup_{t\in(0,\infty)}|\psi_t\ast(k\ast a_j)(x)|=\sup_{t\in(0,\infty)}\lf|k^{(t)}\ast a_j(x)\r|\\
&\le\sup_{t\in(0,\infty)}\int_\rn\lf|k^{(t)}(x-y)-k^{(t)}(x-x_j)\r|\lf|a_j(y)\r|\,dy\\
&\lesssim\int_{B_j}\frac{|y-x_j|^\delta}{|x-x_j|^{n+\delta}}|a_j(x)|\,dy
\lesssim\frac{r_j^\delta}{|x-x_j|^{n+\delta}}\|a_j\|_{L^q(\rn)}|B_j|^{1/q'}\\
&\lesssim\frac{r_j^{n+\delta}}{|x-x_j|^{n+\delta}}\frac{1}{\|\mathbf{1}_{B_j}\|_{X}}
\lesssim\lf[\cm(\mathbf{1}_{B_j})(x)\r]^\frac{n+\delta}{n}\frac{1}{\|\mathbf{1}_{B_j}\|_{X}}.
\end{align*}
This shows that, for any $x\in(4B_j)^\complement$,
$$
\lf|M(Ta_j,\psi)(x)\mathbf{1}_{(4B_j)^\complement}(x)\r|\lesssim
\lf[\cm(\mathbf{1}_{B_j})(x)\r]^\frac{n+\delta}{n}\frac{1}{\|\mathbf{1}_{B_j}\|_{X}}.
$$
Therefore, by this and \eqref{Eqfs11}, we find that
\begin{align}\label{Eq76}
\mathrm{II}&\lesssim\alpha\lf\|\mathbf{1}_{\{x\in\rn:\ \sum_{j\in\nn}\frac{\lambda_j}{\|\mathbf{1}_{B_j}\|_{X}}
[\cm(\mathbf{1}_{B_j})(x)]^\frac{n+\delta}{n}>\frac\alpha2\}}\r\|_{X}\\\noz
&\lesssim\frac{\alpha}{2}\lf\|\mathbf{1}_{\{x\in\rn:\ [\sum_{j\in\nn}\frac{\lambda_j}{\|\mathbf{1}_{B_j}\|_{X}}
\{\cm(\mathbf{1}_{B_j})(x)\}^{\frac{n+\delta}{n}}]^\frac{n}{n+\delta}>(\frac\alpha2)^\frac{n}{n+\delta}\}}
\r\|_{X^{\frac{n+\delta}{n}}}^{\frac{n+\delta}{n}}\\\noz
&\lesssim\lf\|\lf[\sum_{j\in\nn}\frac{\lambda_j\mathbf{1}_{B_j}}{\|\mathbf{1}_{B_j}\|_{X}}
\r]^\frac{n}{n+\delta}\r\|_{X^{\frac{n+\delta}{n}}}^{\frac{n+\delta}{n}}
\lesssim\lf\|\lf\{\sum_{j\in\nn}\lf[\frac{\lambda_j\mathbf{1}_{B_j}}{\|\mathbf{1}_{B_j}\|_{X}}
\r]^{s}\r\}^\frac{1}{s}\r\|_{X}\lesssim\|f\|_{H_X(\rn)}.
\end{align}

Finally, combining \eqref{Eq75} and \eqref{Eq76}, we conclude that, for any $\alpha\in(0,\infty)$,
$$
\alpha\lf\|\mathbf{1}_{\{x\in\rn:\ M(Tf,\psi)(x)>\alpha\}}\r\|_{X}\lesssim\|f\|_{H_X(\rn)},
$$
namely, \eqref{Eq74} holds true. This, together with Remark \ref{Re73} and a dense argument, then finishes the proof of Theorem \ref{Thcz}.
\end{proof}

Recall that, for any given $\gamma\in(0,\infty)$, a linear operator $T$ is called a \emph{non-convolutional} $\gamma$\emph{-order
Calder\'on--Zygmund operator }if $T$ is bounded on $L^2(\rn)$ and its kernel
$$
k:\ (\rn\times\rn)\setminus\{(x,x):\ x\in\rn\}\rightarrow\cc
$$
satisfies that there exists a positive constant $C$ such that, for any $\alpha\in\zz_+^n$ with
$|\alpha|\le\lceil\gamma\rceil-1$ and $x,\ y,\ z\in\rn$ with $|x-y|\geq2|y-z|$,
\begin{equation}\label{Eq71}
\lf|\partial_x^\alpha k(x,y)-\partial_x^\alpha k(x,z)\r|\le C\frac{|y-z|^{\gamma-\lceil\gamma\rceil+1}}{|x-y|^{n+\gamma}}
\end{equation}
and, for any $f\in L^2(\rn)$ having compact support and $x\notin\supp f$,
$$
T(f)(x)=\int_{\supp f}k(x,y)f(y)\,dy.
$$
Here and hereafter, for any $\beta\in(0,\infty)$,
the \emph{symbol
$\lceil\beta\rceil$} denotes the minimal integer not less than $\beta$.

For any given $m\in\nn$, an operator $T$ is said to have the
\emph{vanishing moments up to order} $m$ if, for any $a\in L^2(\rn)$
having compact support and satisfying that, for any $\beta\in\zz_+^n$ with $|\beta|\le m$, $\int_\rn a(x)x^\beta\,dx=0$,
it holds true that $\int_\rn x^\beta T(a)(x)\,dx=0$.

We now have the boundedness of non-convolutional
$\gamma$-order Calder\'on--Zygmund operators from $H_X(\rn)$ to $WH_X(\rn)$ as follows.

\begin{theorem}\label{Thcz1}
Let $\theta,\ s,\ \delta\in(0,1]$. Assume that $X$ is a ball quasi-Banach
function space
having an absolutely continuous quasi-norm and satisfying \eqref{aa}, \eqref{aax} with $q=2$
and Assumption \ref{xinm}.
Assume that $X^{1/s}$ is a ball Banach function space. Let $\gamma\in(0,\infty)$ and $T$ be a non-convolutional $\gamma$-order Calder\'on--Zygmund
operator having the vanishing moments up to order $\lceil\gamma\rceil-1$ satisfying
$\lceil\gamma\rceil-1\leq n(1/\theta-1)$.
If there exists a positive constant $C_0$ such that, for any $\alpha\in(0,\infty)$ and
any sequence
$\{f_j\}_{j\in\nn}\subset\mathscr M(\rn)$,
\begin{equation}\label{Eqfs12}
\alpha\lf\|\mathbf{1}_{\{x\in\rn:\ \{\sum_{j\in\nn}[\cm(f_j)(x)]^\frac{n+\gamma}{n}\}
^\frac{n}{n+\gamma}>\alpha\}}\r\|_{X^{\frac{n+\gamma}{n}}}
\le C_0\lf\|\lf(\sum_{j\in\nn}|f_j|^\frac{n+\gamma}{n}\r)
^\frac{n}{n+\gamma}\r\|_{X^{\frac{n+\gamma}{n}}},
\end{equation}
then $T$ has a unique extension on $H_X(\rn)$ and, moreover,
there exists a positive constant $C$ such that, for any $f\in H_X(\rn)$,
$$
\lf\|Tf\r\|_{WH_X(\rn)}\le C\lf\|f\r\|_{H_X(\rn)}.
$$
\end{theorem}

\begin{remark}
\begin{itemize}
\item[(i)]
Recall that, for any given $\delta\in(0,1]$,
a linear operator $T$ is called a \emph{non-convolutional $\delta$-type Calder\'on--Zygmund operator} $T$
if $T$ is a linear bounded operator on $L^2(\rn)$ and there exist a kernel $k$ on $(\rn\times\rn)\setminus\{(x,x):\ x\in\rn\}$
and a positive constant $C$ such that, for any $x,\ y,\ z\in\rn$ with $|x-y|>2|y-z|$,
$$
\lf|k(x,y)-k(x,z)\r|\leq C\frac{|y-z|^\delta}{|x-y|^{n+\delta}}
$$
and, for any $f\in L^2(\rn)$ having compact support and $x\notin\supp f$,
$$
T(f)(x)=\int_{\supp f}k(x,y)f(y)\,dy.
$$
Observe that, when $\gamma:=\delta\in(0,1]$, the operator $T$ in Theorem \ref{Thcz1}
coincides with a non-convolutional
$\delta$-type Calder\'on--Zygmund operator.
Therefore, the operators in Theorem \ref{Thcz1} include non-convolutional
$\delta$-type Calder\'on--Zygmund operators as special cases.
By this, we know that the \emph{critical index} of
non-convolutional $\delta$-type Calder\'on--Zygmund operators is $\frac{n}{n+\delta}$
(see Remark \ref{111} for more details).
\item[(ii)]
Theorems \ref{Thcz} and \ref{Thcz1} obtain the boundedness of convolutional
$\delta$-type and non-convolutional $\gamma$-order Calder\'on-Zygmund operators
from $H_X(\rn)$ to $WH_X(\rn)$. Since, for any
$q\in(2,\infty)$, the boundedness of
non-convolutional $\gamma$-order Calder\'on-Zygmund operators
on $L^q(\rn)$ can not be guaranteed by our assumptions on $T$,
Assumption \eqref{aax} for some $q\in(1,\infty)$
in Theorem \ref{Thcz} is weaker than \eqref{aax} with $q=2$
in Theorem \ref{Thcz1}.
\end{itemize}
\end{remark}

\begin{proof}[Proof of Theorem \ref{Thcz1}]
By an argument similar to that used in the proof of Theorem \ref{Thcz},
to show Theorem \ref{Thcz1},
it suffices to prove that, for any $\alpha\in(0,\infty)$
and $f\in H_X(\rn)\cap L^2(\rn)$,
\begin{equation}\label{Eq77}
\alpha\lf\|\mathbf{1}_{\{x\in\rn:\ \sum_{j\in\nn}\lambda_jM(Ta_j,\psi)(x)\mathbf{1}_{(4B_j)^\complement}(x)>\frac{\alpha}{2}\}}\r\|_{X}
\lesssim\|f\|_{H_X(\rn)},
\end{equation}
where, for any $j\in\nn$,
$\lambda_j,\ a_j$ and $B_j$ are the same as in the proof of Theorem \ref{Thcz}.

For any given $j\in\nn$, we
first estimate $M(Ta_j,\psi)$, which is as in Definition \ref{Dem}(i)
with $f$ replaced by $Ta_j$.
By the vanishing moments of $T$ and the fact that $\lceil\gamma\rceil-1\le n(1/\theta-1)$
implies that $\lceil\gamma\rceil-1\le d$, we know that,
for any $j\in\nn$, $t\in(0,\infty)$ and $x\in(4B_j)^\complement$,
\begin{align}\label{Eq78}
\lf|\psi_t\ast T(a_j)(x)\r|&\le\frac1{t^n}\int_\rn\lf|\psi\lf(\frac{x-y}{t}\r)-
\sum_{|\beta|\le \lceil\gamma\rceil-1}
\frac{\partial^\beta\psi(\frac{x-x_j}{t})}{\beta!}\lf(\frac{y-x_j}{t}\r)^\beta\r||T(a_j)(y)|\,dy\\\noz
&=\frac1{t^n}\lf(\int_{|y-x_j|<2r_j}+\int_{2r_j\le|y-x_j|
<\frac{|x-x_j|}{2}}+\int_{|y-x_j|\geq\frac{|x-x_j|}{2}}\r)\\\noz
&\hspace*{12pt}\times\lf|\psi\lf(\frac{x-y}{t}\r)-\sum_{|\beta|\le\lceil\gamma\rceil-1}
\frac{\partial^\beta\psi(\frac{x-x_j}{t})}{\beta!}
\lf(\frac{y-x_j}{t}\r)^\beta\r||T(a_j)(y)|\,dy=:\mathrm{I_1}+\mathrm{I_2}+\mathrm{I_3}.
\end{align}

For $\mathrm{I_1}$, the Taylor remainder theorem guarantees that, for any $j\in\nn$ and $y\in\rn$ with $|y-x_j|<2r_j$, there
exists $\xi_1(y)\in2B_j$ such that
\begin{align*}
\mathrm{I_1}\lesssim\frac1{t^n}\int_{|y-x_j|<2r_j}\lf|\sum_{|\beta|=\lceil\gamma\rceil} \partial^\beta\psi\lf(
\frac{x-\xi_1(y)}{t}\r)\r|\lf(\frac{|y-x_j|}{t}\r)^{\lceil\gamma\rceil}|Ta_j(y)|\,dy,
\end{align*}
which, together with the H\"older inequality and the fact that $T$ is bounded on $L^2(\rn)$, further implies that,
for any $t\in(0,\infty)$ and $x\in(4B_j)^\complement$,
\begin{align}\label{Eq79}
\mathrm{I_1}&\lesssim\frac1{t^n}\int_{|y-x_j|<2r_j}\frac{t^{n+\lceil\gamma\rceil}}{|x-x_j|^{n+\lceil\gamma\rceil}}
\frac{|y-x_j|^{\lceil\gamma\rceil}}{t^{\lceil\gamma\rceil}}|Ta_j(y)|\,dy\\\noz
&\lesssim\frac{r_j^{\lceil\gamma\rceil}}{|x-x_j|^{n+\lceil\gamma\rceil}}\|Ta_j\|_{L^2(\rn)}|B_j|^{1/2}
\lesssim\frac{r_j^{n+\lceil\gamma\rceil}}{|x-x_j|^{n+\lceil\gamma\rceil}}\frac1{\|\mathbf{1}_{B_j}\|_{X}}.
\end{align}

For $\mathrm{I_2}$, from the Taylor remainder theorem, the vanishing moments of $a_j$,
$\lceil\gamma\rceil-1\le \lfloor n(1/\theta-1)\rfloor\le d$, \eqref{Eq71} and the H\"older inequality, it follows that,
for any $z\in B_j$, there exists $\xi_2(z)\in B_j$ such that, for any $t\in(0,\infty)$ and $x\in(4B_j)^\complement$,
\begin{align}\label{Eq710}
\mathrm{I_2}
&\lesssim\int_{2r_j\le|y-x_j|<\frac{|x-x_j|}{2}}
\frac{|y-x_j|^{\lceil\gamma\rceil}}{|x-x_j|^{n+\lceil\gamma\rceil}}\\\noz
&\hs\times
\lf[\int_{B_j}\lf|a_j(z)\r|\lf|k(y,z)-\sum_{|\beta|\le\lceil\gamma\rceil-1}\frac{\partial_y^\beta k(y,x_j)}{\beta!}(z-x_j)^\beta\r|
\,dz\r]\,dy\\\noz
&\sim\frac1{|x-x_j|^{n+\lceil\gamma\rceil}}
\int_{2r_j\le|y-x_j|<\frac{|x-x_j|}{2}}|y-x_j|^{\lceil\gamma\rceil}\\\noz
&\hs\times\int_{B_j}\lf|a_j(z)\r|\lf|\sum_{|\beta|=\lceil\gamma\rceil-1}
\frac{\partial_y^\beta k(y,x_j)-\partial_y^\beta k(y,\xi_2(z))}{\beta!}(z-x_j)^\beta\r|\,dz\,dy\\\noz
&\lesssim\frac1{|x-x_j|^{n+\lceil\gamma\rceil}}\int_{2r_j\le|y-x_j|
<\frac{|x-x_j|}{2}}|y-x_j|^{\lceil\gamma\rceil}
\int_{B_j}\lf|a_j(z)\r|\frac{|z-x_j|^\gamma}{|y-x_j|^{n+\gamma}}\,dz\,dy\\\noz
&\lesssim\frac{r_j^\gamma}{|x-x_j|^{n+\lceil\gamma\rceil}}\int_{2r_j\le|y-x_j|<\frac{|x-x_j|}{2}}
\frac1{|y-x_j|^{n+\gamma-\lceil\gamma\rceil}}\,dy\|a_j\|_{L^2(\rn)}|B_j|^{1/2}
\lesssim\frac{r_j^{n+\gamma}}{|x-x_j|^{n+\gamma}}\frac1{\|\mathbf{1}_{B_j}\|_{X}}.
\end{align}

For $\mathrm{I_3}$, by the vanishing moments of $a_j$,
$\lceil\gamma\rceil-1\le\lfloor n(1/\theta-1)\rfloor\le d$,
\eqref{Eq71} and the H\"older inequality, we know that, for any $z\in B_j$, there exists $\xi_3(z)\in B_j$
such that, for any $t\in(0,\infty)$ and $x\in(4B_j)^\complement$,
\begin{align}\label{Eq711}
\mathrm{I_3}&\le\int_{|y-x_j|\geq\frac{|x-x_j|}{2}}\lf|\frac1{t^n}\lf[\psi\lf(\frac{x-y}{t}\r)
-\sum_{|\beta|\le\lceil\gamma\rceil-1}
\frac{\partial^\beta\psi(\frac{x-x_j}{t})}{\beta!}\lf(\frac{y-x_j}{t}\r)^\beta\r]\r|\\\noz
&\hspace*{12pt}\times\lf\{\int_{B_j}\lf|a_j(z)\r|\lf|k\lf(y,z\r)
-\sum_{|\beta|\le\lceil\gamma\rceil-1}\frac{\partial_y^\beta k(y,x_j)}{\beta!}\lf(z-x_j\r)^\beta\r|\,dz\r\}\,dy\\\noz
&\lesssim\int_{|y-x_j|\geq\frac{|x-x_j|}{2}}\lf|\frac1{t^n}\lf[\psi\lf(\frac{x-y}{t}\r)
-\sum_{|\beta|\le\lceil\gamma\rceil-1}
\frac{\partial^\beta\psi(\frac{x-x_j}{t})}{\beta!}\lf(\frac{y-x_j}{t}\r)^\beta\r]\r|\\\noz
&\hspace*{12pt}\times\int_{B_j}\lf|a_j(z)\r|\lf|\sum_{|\beta|=\lceil\gamma\rceil-1}
\frac{\partial_y^\beta k(y,x_j)-\partial_y^\beta k(y,\xi_3(z))}{\beta!}\lf(z-x_j\r)^\beta\r|\,dz\,dy\\\noz
&\lesssim\int_{|y-x_j|\geq\frac{|x-x_j|}{2}}
\lf[|\psi_t\lf(x-y\r)|+\lf|\frac1{t^n}\sum_{|\beta|\le\lceil\gamma\rceil-1}
\frac{\partial^\beta\psi(\frac{x-x_j}{t})}{\beta!}\lf(\frac{y-x_j}{t}\r)^\beta\r|\r]\\\noz
&\hspace*{12pt}\times\int_{B_j}\lf|a_j(z)\r|\frac{|z-x_j|^\gamma}{|y-x_j|^{n+\gamma}}\,dz\,dy\\\noz
&\lesssim\lf\|a_j\r\|_{L^2(\rn)}\lf|B_j\r|^{1/2}\lf[\frac{r_j^\gamma}{|x-x_j|^{n+\gamma}}
\int_{|y-x_j|\geq\frac{|x-x_j|}{2}}|\psi_t(x-y)|\,dy\r.\\\noz
&\hspace*{12pt}\lf.+\sum_{|\beta|\le\lceil\gamma\rceil-1}r_j^\gamma
\int_{|y-x_j|\geq\frac{|x-x_j|}{2}}\frac 1{|x-x_j|^{n+|\beta|}}
\frac1{|y-x_j|^{n+\gamma-|\beta|}}\,dy\r]
\lesssim\frac{r_j^{n+\gamma}}{|x-x_j|^{n+\gamma}}\frac1{\|\mathbf{1}_{B_j}\|_{X}}.
\end{align}

Combining \eqref{Eq78}, \eqref{Eq79}, \eqref{Eq710} and \eqref{Eq711}, we obtain,
for any $x\in(4B_j)^\complement$,
\begin{align*}
\lf|M(Ta_j,\psi)(x)\r|=\sup_{t\in(0,\infty)}\lf|\psi_t\ast Ta_j(x)\r|\lesssim
\frac{r_j^{n+\gamma}}{|x-x_j|^{n+\gamma}}\frac1{\|\mathbf{1}_{B_j}\|_{X}}
\lesssim\lf[\cm(\mathbf{1}_{B_j})(x)\r]^\frac{n+\gamma}{n}\frac1{\|\mathbf{1}_{B_j}\|_{X}},
\end{align*}
which implies that
$$
\lf|M(Ta_j,\psi)(x)\mathbf{1}_{(4B_j)^\complement}(x)\r|
\lesssim\lf[\cm(\mathbf{1}_{B_j})(x)\r]^\frac{n+\gamma}{n}\frac1{\|\mathbf{1}_{B_j}\|_{X}}.
$$
Therefore, by \eqref{Eqfs12} and an argument similar to that used in the estimation of \eqref{Eq76},
we conclude that
\begin{align*}
\alpha\lf\|\mathbf{1}_{\{x\in(4B_j)^\complement:\ \sum_{j\in\nn}\lambda_jM(Ta_j,\psi)(x)
>\frac{\alpha}{2}\}}\r\|_{X}
&\lesssim\frac{\alpha}2\lf\|\mathbf{1}_{\{x\in\rn:\ [\sum_{j\in\nn}\frac{\lambda_j}
{\|\mathbf{1}_{B_j}\|_{X}}\{\cm(\mathbf{1}_{B_j})(x)\}^\frac{n+\gamma}{n}]^{\frac{n}{n+\gamma}}>
(\frac{\alpha}2)^{\frac{n}{n+\gamma}}\}}\r\|_{X^\frac{n+\gamma}{n}}^{\frac{n+\gamma}{n}}\\
&\lesssim\lf\|\lf[\sum_{j\in\nn}\frac{\lambda_j
\mathbf{1}_{B_j}}{\|\mathbf{1}_{B_j}\|_{X}}\r]^\frac{n+\gamma}{n}
\r\|_{X^\frac{n+\gamma}{n}}^{\frac{n+\gamma}{n}}\lesssim\|f\|_{H_X(\rn)}.
\end{align*}
This shows that \eqref{Eq77} holds true and hence finishes the proof of Theorem \ref{Thcz}.
\end{proof}

\section{Applications}\label{s7}

In this section, we apply all above results to three concrete examples
of ball quasi-Banach function spaces, namely, Morrey spaces, mixed-norm Lebesgue spaces
and Orlicz-slice spaces, respectively, in Subsections \ref{s7.1},
\ref{s7.2} and \ref{s7.3}.

\subsection{Morrey spaces\label{s7.1}}

We begin with recalling the notion of Morrey spaces.

\begin{definition}\label{mmus}
Let $0<q\leq p<\infty$.
The \emph{Morrey space $M_q^p(\rn)$} is defined
to be the set of all measurable functions $f$ such that
$$
\|f\|_{M_q^p(\rn)}:=\sup_{B\in\BB}|B|^{1/p-1/q}\|f\|_{L^q(B)}<\infty,
$$
where $\BB$ is as in \eqref{Eqball} (the set of all balls of $\rn$).
\end{definition}

The space $M_q^p(\rn)$ was introduced by Morrey \cite{m38}.
Furthermore, the following Fefferman--Stein
vector-valued maximal inequalities
for $M_q^p(\rn)$ hold true (see, for instance, \cite{cf,H15}),
which shows that the Morrey space $M_q^p(\rn)$ satisfies Assumption \ref{a2.15}.

\begin{lemma}\label{mo1}
Let $0<q\leq p<\infty$.
Assume that $r\in(1,\infty)$ and $s\in(0,q)$.
Then there exists a positive constant $C$ such that,
for any $\{f_j\}_{j=1}^\infty\subset\mathscr M(\rn)$,
$$
\lf\|\lf\{\sum_{j=1}^\infty\lf[\cm(f_j)\r]^r\r\}^{1/r}\r\|_{[M_q^p(\rn)]^{1/s}}
\le C\lf\|\lf\{\sum_{j=1}^\infty\lf|f_j\r|^r\r\}^{1/r}\r\|_{[M_q^p(\rn)]^{1/s}}.
$$
where $[M_q^p(\rn)]^{1/s}$ denotes the $\frac1s$-convexification of
$M_q^p(\rn)$ as in Definition \ref{Debf}(i) with $X$ and $p$ replaced, respectively,
by $M_q^p(\rn)$ and $1/s$.
\end{lemma}

Now we recall the notion of the weak Morrey space $WM_q^p(\rn)$.

\begin{definition}
Let $0<q\leq p<\infty$.
The \emph{weak Morrey space $WM_q^p(\rn)$}
is defined to be the set of all measurable functions $f$ such that
$$
\|f\|_{WM_q^p(\rn)}:=\sup_{\alpha\in(0,\infty)}\,\lf\{\alpha\lf\|\mathbf1_{\{x\in\rn:\ |f(x)|>\alpha\}}\r\|_{M_q^p(\rn)}\r\}<\infty.
$$
\end{definition}

\begin{remark}\label{mmo}
Let $0<q\leq p<\infty$. The weak Morrey space $WM_q^p(\rn)$
is just the weak Morrey space $\mathcal{M}_u^{q,\infty}(\rn)$ in \cite{H17} with
$u(B):=|B|^{1/q-1/p}$ for any $B\in\BB$, where $\BB$ is as in \eqref{Eqball}.
\end{remark}

The following Fefferman--Stein
vector-valued maximal inequalities
for $WM_q^p(\rn)$ hold true (see, for instance, \cite[Theorem 3.2]{H17}),
which shows that the Morrey space $M_q^p(\rn)$ satisfies Assumption \ref{xinm}.

\begin{lemma}\label{mwfs}
Let $0<q\leq p<\infty$.
Assume that $r\in(1,\infty)$ and $s\in(0,q)$.
Then there exists a positive constant $C$ such that,
for any $\{f_j\}_{j=1}^\infty\subset\mathscr M(\rn)$,
$$
\lf\|\lf\{\sum_{j=1}^\infty\lf[\cm(f_j)\r]^r\r\}^{1/r}\r\|_{[WM_q^p(\rn)]^{1/s}}
\le C\lf\|\lf\{\sum_{j=1}^\infty\lf|f_j\r|^r\r\}^{1/r}\r\|_{[WM_q^p(\rn)]^{1/s}},
$$
where $[WM_q^p(\rn)]^{1/s}$ denotes the $\frac{1}{s}$-convexification of
$WM_q^p(\rn)$ as in Definition \ref{Debf}(i) with
$X$ and $p$ replaced, respectively, by $WM_q^p(\rn)$ and $\frac{1}{s}$.
\end{lemma}

Similarly to \cite[Lemma 5.7]{H15} and \cite[Theorem 4.1]{ST}, we can easily show
the following conclusion and we omit
the details here.

\begin{lemma}\label{mM}
Let $0<q\leq p<\infty$,
$r\in (0,q)$ and $s\in(q,\infty]$. Then there exists
a positive constant $C$
such that, for any $f\in\mathscr{M}(\rn)$,
\begin{equation*}
\lf\|\mathcal{M}^{((s/r)')}(f) \r\|_{([M_q^p(\rn)]^{1/r})'}\leq C\lf\|f \r\|_{([M_q^p(\rn)]^{1/r})'},
\end{equation*}
where $([M_q^p(\rn)]^{1/r})'$ is as in \eqref{asso} with
$X:=[M_q^p(\rn)]^{1/r}$.
\end{lemma}

Now we introduce the notion of the weak Morrey Hardy space $WHM_q^p(\rn)$.

\begin{definition}\label{mmha}
Let $0<q\leq p<\infty$.
The \emph{weak Morrey Hardy
space $WHM_q^p(\rn)$} is defined to be the set of all $f\in\cs'(\rn)$
such that $\|f\|_{WHM_q^p(\rn)}:=\
\|M_N^0(f)\|_{WM_q^p(\rn)}<\infty$,
where $M_N^0(f)$ is as in \eqref{EqMN0} with $N$ sufficiently large.
\end{definition}

\begin{remark}
Let $1<q\leq p<\infty$. By Lemma \ref{mwfs},
we conclude that, for any $r\in(1,q)$,
$\cm$ in \eqref{mm} is bounded on $(WM_q^p(\rn))^{1/r}$, which, combined with Theorem \ref{dayu2}, implies that
$WHM_q^p(\rn)=WM_q^p(\rn)$  with equivalent norms.
\end{remark}

By Lemma \ref{mwfs}  and Theorem \ref{Thmc}(ii),
we obtain the following maximal function characterizations of the weak Morrey Hardy space $WHM_q^p(\rn)$.

\begin{theorem}\label{mmth1}
Let $0<q\leq p<\infty$,
and $\psi\in\mathcal{S}(\rn)$ satisfy $\int_{\rn}\psi(x)\,dx\neq0.$
Assume that $b\in(n/q,\infty)$ and
$N\ge\lfloor b+1\rfloor$. For any $f\in\mathcal{S}'(\rn)$,
if one of the following quantities
$$
\lf\|M_N^0(f)\r\|_{WM_q^p(\rn)},\  \lf\|M(f,\psi)\r\|_{WM_q^p(\rn)},\  \lf\|M_a^*(f,\psi)\r\|_{WM_q^p(\rn)},\ \lf\|M_N(f)\r\|_{WM_q^p(\rn)},
$$
$$
\lf\|M_b^{**}(f,\psi)\r\|_{WM_q^p(\rn)},\ \lf\|M_{b,\ N}^{**}(f)\r\|_{WM_q^p(\rn)}
\quad and \quad\lf\|\cn(f)\r\|_{WM_q^p(\rn)}
$$
is finite, then the others are also finite and mutually equivalent with the implicit
positive equivalence constants independent of $f$.
\end{theorem}

\begin{remark}\label{remax}
Let $0<q\leq p<\infty$ and $p=q$. Then we know that
$M_q^p(\rn)=L^q(\rn)$ and $WM_q^p(\rn)(\rn)=WH^q(\rn)$,
where $WH^q(\rn)$ denotes the classical weak Hardy space, and
the characterizations of $WH^q(\rn)$ in terms of all
the maximal functions except for $M_b^{**}(f,\psi)$ and $M_{N,b}^{**}(f)$ in Theorems \ref{mmth1}
were obtained in \cite[Theorems 2.10 and 2.11]{LYJ}
or \cite[Theorem 3.7 and Corollary 3.8]{YYYZ} as a special case.
Moreover, in this case, Theorem \ref{mdj} widens the range of
$N\in(\frac nq+n+1,\infty)\cap\nn$ in \cite[Theorem 3.7 and Corollary 3.8]{YYYZ}
into $N\in[\lfloor\frac nq+1\rfloor,\infty)\cap\nn$.
\end{remark}

Using Lemmas \ref{mo1}, \ref{mwfs} and \ref{mM}
and Theorems \ref{Thad} and \ref{Thar},
we immediately obtain the atomic characterization of $WHM_q^p(\rn)$
(see Theorem \ref{Mth2} below) and
the molecular characterization of $WHM_q^p(\rn)$
(see Theorem \ref{mfenzi} below) as follows.

\begin{theorem}\label{Mth2}
Let $0<q\leq p<\infty$. Assume that $r\in(\max\{1,q\},\infty)$
and $d\in\zz_+$ satisfying $d\geq\lfloor n(\frac{1}{\min\{1,q\}}-1)\rfloor$.
Then $f\in WHM_q^p(\rn)$ if and only if
$$
f=\sum_{i\in \zz}\sum_{j\in\nn}\lambda_{i,j}a_{i,j}\quad\text{in}\quad\cs'(\rn)
\quad and\quad
\sup_{i\in\zz}\lf\|\sum_{j\in\nn}
\frac{\lambda_{i,j}\mathbf{1}_{B_{i,j}}}
{\|\mathbf{1}_{B_{i,j}}\|_{M_q^p(\rn)}}\r\|_{M_q^p(\rn)}<\infty,
$$
where $\{a_{i,j}\}_{i\in\zz,j\in\nn}$ is a sequence of $(M_q^p(\rn),\,r,\,d)$-atoms
supported, respectively, in balls
$\{B_{i,j}\}_{\gfz{i\in\zz}{j\in\nn}}$ such that, for any
$i\in\zz$, $\sum_{j\in\nn}\mathbf{1}_{cB_{i,j}}\le A$ with $c\in(0,1]$ and $A$
being a positive constant independent of $f$ and $i$, and,
for any $i\in\zz$ and $j\in\nn$,
$\lambda_{i,j}:=\widetilde A2^i\|\mathbf{1}_{B_{i,j}}\|_{M_q^p(\rn)}$
with $\widetilde A$ being a positive constant independent of $f$ and $i$.

Moreover, for any $f\in WHM_q^p(\rn)$,
$$
\|f\|_{WHM_q^p(\rn)}\sim\inf\lf\{\sup_{i\in\zz}\lf\|\sum_{j\in\nn}
\frac{\lambda_{i,j}\mathbf{1}_{B_{i,j}}}
{\|\mathbf{1}_{B_{i,j}}\|_{M_q^p(\rn)}}\r\|_{M_q^p(\rn)}\r\},
$$
where the infimum is taken over all decompositions of $f$ as above
and the positive equivalence constants are independent of $f$.
\end{theorem}

\begin{remark} We should point out that, when $q\in(0,1]$ and $r=\fz$,
Theorem \ref{Mth2} was obtained by Ho \cite[Theorems 4.2 and 4.3]{H17}.
\end{remark}

\begin{theorem}\label{mfenzi}
Let $p$, $q$,
$r$ and $d$ be the same as in Theorem \ref{Mth2}, and $\epsilon\in(n+d+1,\infty)$.
Then $f\in WHM_q^p(\rn)$ if and only if
$$
f=\sum_{i\in \zz}\sum_{j\in\nn}\lambda_{i,j}m_{i,j}\quad\text{in}\quad\cs'(\rn)
\quad and\quad
\sup_{i\in\zz}\lf\|\sum_{j\in\nn}
\frac{\lambda_{i,j}\mathbf{1}_{B_{i,j}}}
{\|\mathbf{1}_{B_{i,j}}\|_{M_q^p(\rn)}}\r\|_{M_q^p(\rn)}<\infty,
$$
where $\{m_{i,j}\}_{i\in\zz,j\in\nn}$ is a sequence of
$(M_q^p(\rn),r,d,\epsilon)$-molecules associated, respectively, with balls
$\{B_{i,j}\}_{i\in\zz,j\in\nn}$ such that, for any
$i\in\zz$, $\sum_{j\in\nn}\mathbf{1}_{cB_{i,j}}\le A$ with $c\in(0,1]$ and
$A$ being a positive constant independent of $f$
and $i$, and,
for any $i\in\zz$ and $j\in\nn$,
$\lambda_{i,j}:=\widetilde A2^i\|\mathbf{1}_{B_{i,j}}\|_{M_q^p(\rn)}$
with $\widetilde A$ being a positive constant independent of $f$, $i$ and $j$.

Moreover, for any $f\in WHM_q^p(\rn)$,
$$
\|f\|_{WHM_q^p(\rn)}\sim
\inf\lf[
\sup_{i\in\zz}\lf\|
\sum_{j\in\nn}\frac{  \lambda_{i,j}\mathbf{1}_{B_{i,j}}  }
{\|\mathbf{1}_{B_{i,j}}\|_{M_q^p(\rn)}}
\r\|_{M_q^p(\rn)}
\r],
$$
where the infimum is taken over all decompositions of $f$ as above
and the positive equivalence constants are independent of $f$.
\end{theorem}

\begin{remark}
Let $0<q\leq p<\infty$ and $p=q$.
In this case,
for any $\tau\in(0,\infty),\ r\in[1,\infty]$ and $d\in\zz_+,$
any $(M_q^p(\rn),r,d)$-atom
and any $(M_q^p(\rn),r,d,\epsilon)$-molecule
just become, respectively, a well-known classical atom (see, for instance, \cite[Definition 1.1]{Lu}
or \cite[p.\,112]{S}) and a well-known classical molecule (see, for instance, \cite[Definition 1.2]{HYZ}
with $X:=L^q(\rn)$).
In this case, Theorem \ref{Mth2} was obtained by \cite[Theorem 3.5]{LYJ} and \cite[Theorem 4.4]{YYYZ} as a special case;
Theorem \ref{mfenzi} was obtained by \cite[Theorem 3.9]{LYJ} and \cite[Theorem 5.3]{YYYZ} as a special case.
\end{remark}

Now, we recall the notion of Morrey Hardy space $HM_p^q(\rn)$ as follows.
\begin{definition}\label{Mhardy}
Let $0<q\leq p<\infty$.
The \emph{Morrey Hardy
space $HM_q^p(\rn)$} is defined to be the set of all $f\in\cs'(\rn)$
such that $\|f\|_{HM_q^p(\rn)}:=\
\|M_N^0(f)\|_{M_q^p(\rn)}<\infty$,
where $M_N^0(f)$ is as in \eqref{EqMN0} with $N$ sufficiently large.
\end{definition}

To obtain the boundedness of Calder\'on--Zygmund operators from $HM_q^p(\rn)$
to $WHM_q^p(\rn)$, we need the following vector-valued inequality of
the Hardy--Littlewood maximal operator $\cm$ in \eqref{mm} from $M_1^p(\rn)$ to $WM_1^p(\rn)$.

\begin{proposition}\label{pro822}
Let  $p\in[1,\infty)$ and $r\in(1,\infty)$. Then there exists a positive constant $C$ such that,
for any $\{f_j\}_{j\in\nn}\subset M_1^p(\rn)$,
$$
\lf\|\lf\{\sum_{j=1}^\infty\lf[\cm(f_j)\r]^r\r\}^\frac1r\r\|_{WM_1^p(\rn)}\le
C\lf\|\lf\{\sum_{j=1}^\infty|f_j|^r\r\}^\frac1r\r\|_{M_1^p(\rn)}.
$$
\end{proposition}

\begin{proof}
Let $B:=B(x_0,R)\subset\BB$ with $x_0\in\rn$ and $R\in(0,\infty)$,
where $\BB$ is as in \eqref{Eqball} (the set of all balls of $\rn$).
For any given $j\in\nn$, we decompose $f_j$ into
$$
f_j=f_j^{(0)}+\sum_{k=1}^\infty f_j^{(k)},
$$
where $f_j^{(0)}:=f_j\mathbf1_{2B}$ and, for any $k\in\nn$, $f_j^{(k)}:=f_j\mathbf1_{2^{k+1}B\setminus2^kB}$.
From this and the Minkowski inequality, we deduce that
$$
\lf\{\sum_{j=1}^\infty\lf[\cm(f_j)\r]^r\r\}^\frac1r\leq
\lf\{\sum_{j=1}^\infty\lf[\cm(f_j^{(0)})\r]^r\r\}^\frac1r
+\sum_{k=1}^\infty\lf\{\sum_{j=1}^\infty\lf[\cm(f_j^{(k)})\r]^r\r\}^\frac1r.
$$
For any given $\lambda\in(0,\infty)$, we find that
\begin{align*}
&\lf\|\mathbf1_{\{x\in\rn:\ \{\sum_{j=1}^\infty[\cm(f_j)(x)]^r\}^\frac1r>\lambda\}}\r\|_{L^1(B)}\\
&\quad\leq
\lf\|\mathbf1_{\{x\in\rn:\ \{\sum_{j=1}^\infty[\cm(f_j^{(0)})(x)]^r\}^\frac1r>\lambda/2\}}\r\|_{L^1(B)}
+\lf\|\mathbf1_{\{x\in\rn:\ \sum_{k=1}^\infty\{\sum_{j=1}^\infty[\cm(f_j^{(k)})(x)]^r\}^\frac1r>\lambda/2\}}\r\|_{L^1(B)}\\
&\quad\lesssim
\lf\|\mathbf1_{\{x\in\rn:\ \{\sum_{j=1}^\infty[\cm(f_j^{(0)})(x)]^r\}^\frac1r>\lambda/2\}}\r\|_{L^1(B)}
+\lambda^{-1}\lf\|\sum_{k=1}^\infty\lf\{\sum_{j=1}^\infty\lf[\cm(f_j^{(k)})\r]^r\r\}^\frac1r\r\|_{L^1(B)}\\
&\quad\lesssim
\lf\|\mathbf1_{\{x\in\rn:\ \{\sum_{j=1}^\infty[\cm(f_j^{(0)})(x)]^r\}^\frac1r>\lambda/2\}}\r\|_{L^1(B)}
+\lambda^{-1}\sum_{k=1}^\infty\lf\|\lf\{\sum_{j=1}^\infty\lf[\cm(f_j^{(k)})\r]^r\r\}^\frac1r\r\|_{L^1(B)}\\
&\quad=:\mathrm{I}+\mathrm{II}.
\end{align*}
From the Fefferman--Stein vector-valued inequality (see \cite[Theorem 1(2)]{FS}), it follows that
\begin{align*}
\mathrm{I}\lesssim\lambda^{-1}\lf\|\lf[\sum_{j=1}^\infty\lf|f_j^{(0)}\r|^r\r]^\frac1r\r\|_{L^1(\rn)}
\sim\lambda^{-1}\lf\|\lf[\sum_{j=1}^\infty\lf|f_j\r|^r\r]^\frac1r\r\|_{L^1(2B)}.
\end{align*}
For any given $j,\ k\in\nn$ and $x\in B$, it is easy to find that
\begin{align*}
\cm(f_j^{(k)})(x)&=\sup_{t>0}\frac1{|B(x,t)|}\int_{B(x,t)}|f_j^{(k)}(y)|\,dy\\
&\sim\sup_{t>2^kR}\frac1{|B(x,t)|}\int_{B(x,t)}|f_j^{(k)}(y)|\,dy
\lesssim\lf(2^{k}R\r)^{-n}\int_\rn|f_j^{(k)}(y)|\,dy.
\end{align*}
From this and the Minkowski inequality, we deduce that, for any $k\in\nn$ and $x\in B$,
\begin{align*}
\lf\{\sum_{j=1}^\infty\lf[\cm(f_j^{(k)})(x)\r]^r\r\}^\frac1r&\lesssim \lf\{\sum_{j=1}^\infty\lf[\lf(2^kR\r)^{-n}\int_\rn\lf|f_j^{(k)}(x)\r|\,dx\r]^r\r\}^\frac1r\\
&\lesssim\lf(2^kR\r)^{-n}\int_\rn\lf[\sum_{j=1}^\infty\lf|f_j^{(k)}(x)\r|^r\r]^\frac1r\,dx\lesssim
\lf(2^kR\r)^{-n}\lf\|\lf[\sum_{j=1}^\infty\lf|f_j\r|^r\r]^\frac1r\r\|_{L^1(2^{k+1}B)},
\end{align*}
which implies that
\begin{align*}
\mathrm{II}\lesssim\lambda^{-1}\sum_{k=1}^\infty|B|\lf(2^kR\r)^{-n}\lf\|\lf[\sum_{j=1}^\infty
\lf|f_j\r|^r\r]^\frac1r\r\|_{L^1(2^{k+1}B)}
\sim\lambda^{-1}\sum_{k=1}^\infty2^{-kn}\lf\|\lf[\sum_{j=1}^\infty\lf|f_j\r|^r\r]^\frac1r\r\|_{L^1(2^{k+1}B)}.
\end{align*}
By the estimates of $\mathrm{I}$ and $\mathrm{II}$, we conclude that
\begin{align*}
|B|^{\frac1p-1}\lf\|\mathbf1_{\{x\in\rn:\ \{\sum_{j=1}^\infty[\cm(f_j)(x)]^r\}^\frac1r>\lambda\}}\r\|_{L^1(B)}
&\lesssim\lambda^{-1}\sum_{k=0}^\infty2^{-\frac{kn}{p}}|2^{k+1}B|^{\frac1p-1}\lf\|\lf[\sum_{j=1}^\infty
\lf|f_j\r|^r\r]^\frac1r\r\|_{L^1(2^{k+1}B)}\\
&\lesssim\lambda^{-1}\sum_{k=0}^\infty2^{-\frac{kn}{p}}\lf\|
\lf[\sum_{j=1}^\infty\lf|f_j\r|^r\r]^\frac1r\r\|_{M_1^p(\rn)}\\
&\sim\lambda^{-1}\lf\|\lf[\sum_{j=1}^\infty\lf|f_j\r|^r\r]^\frac1r\r\|_{M_1^p(\rn)}.
\end{align*}
This finishes the proof of Proposition \ref{pro822}.
\end{proof}

Applying Proposition \ref{pro822}, Lemmas \ref{mo1} and \ref{mM}, Theorems \ref{Thcz} and \ref{Thcz1},
we immediately obtain the following boundedness from
$HM_q^p(\rn)$ to $WHM_q^p(\rn)$
of both convolutional $\delta$-type and $\gamma$-type
Calder\'on--Zygmund operators, respectively, as follows.

\begin{theorem}\label{cz81}
Let $q\in(0,1]$ and $p\in(0,\infty)$ with $q\le p$, and $\delta\in(0,1]$.
Let $T$ be a convolutional $\delta$-type Calder\'on--Zygmund
operator. If $q\in[\frac{n}{n+\delta},1]$, then $T$ has a unique extension on $HM_q^p(\rn)$ and,
moreover, there exists a positive constant $C$ such that,
for any $f\in HM_q^p(\rn)$,
$$
\|Tf\|_{WHM_q^p(\rn)}\le C\|f\|_{HM_q^p(\rn)}.
$$
\end{theorem}

\begin{theorem}\label{cz82}
Let $q\in(0,1]$ and $p\in(0,\infty)$ with $q\le p$, and $\gamma\in(0,\infty)$.
Let $T$ be a $\gamma$-type Calder\'on--Zygmund
operator and have the vanishing moments up to order $\lceil\gamma\rceil-1$.
If $\lceil\gamma\rceil-1\le n(\frac1{q}-1)\le\gamma$, then $T$ has a unique extension on
$HM_q^p(\rn)$ and, moreover, there exists a
positive constant $C$ such that,
for any $f\in HM_q^p(\rn)$,
$$
\|Tf\|_{WHM_q^p(\rn)}\le C\|f\|_{HM_q^p(\rn)}.
$$
\end{theorem}

\begin{remark}\label{111}
Let $0<q\leq p<\infty$ and $p=q$.
In this case,
we know that $M_q^p(\rn)=L^q(\rn)$ and $WM_q^p(\rn)=WL^q(\rn)$. Thus, by Theorem \ref{cz81},
we recover that the convolutional $\delta$-type Calder\'on--Zygmund operator $T$
is bounded from $H^{\frac{n}{n+\delta}}(\rn)$ to $WH^{\frac{n}{n+\delta}}(\rn)$,
which is just \cite[Theorem 1]{L} (see also \cite[Theorem 5.2]{LYJ} and \cite[Theorem 7.4]{YYYZ}).
Here, $\frac{n}{n+\delta}$ is called the \emph{critical index}.
Also, by Theorem \ref{cz82}, we recover that any $\gamma$-order Calder\'on--Zygmund operator
is bounded from $H^{\frac{n}{n+\gamma}}(\rn)$ to $WH^{\frac{n}{n+\gamma}}(\rn)$, which is a special
case of \cite[Theorem 7.6]{YYYZ}. Yan et al. \cite{YYYZ} pointed out that the \emph{critical index} of
$\gamma$-order Calder\'on--Zygmund operators is $\frac{n}{n+\gamma}$.
\end{remark}

\subsection{Mixed-norm Lebesgue spaces\label{s7.2}}

We begin with recalling the notion of mixed-norm Lebesgue spaces.

\begin{definition}\label{mix}
Let $\vec{p}:=(p_1,\ldots,p_n)\in(0,\infty]^n$.
The \emph{mixed-norm Lebesgue space $L^{\vec{p}}(\rn)$} is defined
to be the set of all measurable functions $f$ such that
$$
\|f\|_{L^{\vec{p}}(\rn)}:=\lf\{\int_{\rr}\cdots\lf[\int_{\rr}|f(x_1,\ldots,x_n)|^{p_1}\,dx_1\r]
^{\frac{p_2}{p_1}}\cdots\,dx_n\r\}^{\frac{1}{p_n}}<\infty
$$
with the usual modifications made when $p_i=\infty$ for some $i\in\{1,\ldots,n\}$.
\end{definition}

The space $L^{\vec{p}}(\rn)$
was studied by Benedek and Panzone \cite{bp} in 1961, which can be traced
back to H\"ormander \cite{H60}.
From the definition
of $\|\cdot\|_{L^{\vec{p}}(\rn)}$, it is easy to deduce that
the mixed-norm Lebesgue space $L^{\vec{p}}(\rn)$
is a ball quasi-Banach function space.
Let $\vec{p}:=(p_1,\ldots,p_n)\in[1,\infty]^n$. Then, for any $f\in L^{\vec{p}}(\rn)$
and $g\in L^{\vec{p'}}(\rn)$, it is easy to know that
$$\int_{\rn}|f(x)g(x)|\, dx \leq \|f\|_{L^{\vec{p}}(\rn)}\|g\|_{L^{\vec{p'}}(\rn)},
$$
where $\vec{p'}$
denotes the \emph{conjugate vector} of $\vec{p}$, namely, for any $i\in\{1,\ldots,n\}$, $1/p_i + 1/p_i'
= 1$.
This implies that $L^{\vec{p}}(\rn)$ with $\vec{p}\in[1,\infty]^n$ is a ball
Banach function space, which is not a Banach function space (see the following remark).

\begin{remark}\label{mbfs-not}
It is worth pointing out that
$L^{\vec{p}}(\rn)$ with $\vec{p}\in[1,\infty]^n$ may not be a Banach function space.
For instance, let $\vec{p}:=(2,1)$ and $n:=2$.
In this case, $L^{\vec{p}}(\rr^n)=L^{(2,1)}(\rr^2)$.
Let
$$E:=\bigcup_{m\in\nn}[m,m+1/m]\times[m,m+1/\sqrt{m}].$$
Then it is easy to show that $|E|<\infty$, but
\begin{align*}
\|\mathbf{1}_{E}\|_{L^{(2,1)}(\rr^2)}
=\int_{\rr}\lf[\int_{\rr}\mathbf{1}_{E}(x_1,x_2)\,dx_1\r]^{\frac 12}\,dx_2
=\sum_{m\in\nn}\int_{m}^{m+1/\sqrt{m}}\lf[\int_{m}^{m+1/m}\,dx_1\r]^{\frac 12}\,dx_2=\infty.
\end{align*}
Thus, $L^{(2,1)}(\rr^2)$ does not satisfy Definition \ref{bn}(iv), which means
that $L^{(2,1)}(\rr^2)$ is not a Banach function space.
\end{remark}

Furthermore, the following Fefferman--Stein
vector-valued maximal inequalities
for $L^{\vec{p}}(\rn)$ hold true (see, for instance, \cite[Lemma 3.7]{hlyy}),
which shows that the mixed-norm Lebesgue space $L^{\vec{p}}(\rn)$ satisfies Assumption \ref{a2.15}.
For any $\vec{p}:=(p_1,\ldots,p_n)\in(0,\infty)^n$, we always let
$p_-:= \min\{p_1, \ldots , p_n\}$ and  $p_+ := \max\{p_1, \ldots , p_n\}$.

\begin{lemma}\label{mixo1}
Let $\vec{p}\in(0,\infty)^n$.
Assume that $r\in(1,\infty)$ and $s\in(0,p_-)$.
Then there exists a positive constant $C$ such that,
for any $\{f_j\}_{j=1}^\infty\subset\mathscr M(\rn)$,
$$
\lf\|\lf\{\sum_{j=1}^\infty\lf[\cm(f_j)\r]^r\r\}^{1/r}\r\|_{[L^{\vec{p}}(\rn)]^{1/s}}
\le C\lf\|\lf\{\sum_{j=1}^\infty\lf|f_j\r|^r\r\}^{1/r}\r\|_{[L^{\vec{p}}(\rn)]^{1/s}},
$$
where $[L^{\vec{p}}(\rn)]^{1/s}$ denotes the $\frac1s$-convexification of
$L^{\vec{p}}(\rn)$ as in Definition \ref{Debf}(i) with $X$ and $p$ replaced, respectively,
by $L^{\vec{p}}(\rn)$ and $1/s$.
\end{lemma}

Now we introduce the weak mixed-norm Lebesgue space $WL^{\vec{p}}(\rn)$.

\begin{definition}\label{wmix}
Let $\vec{p}\in(0,\infty)^n$.
The \emph{weak mixed-norm Lebesgue space $WL^{\vec{p}}(\rn)$}
is defined to be the set of all measurable functions $f$ such that
$$
\|f\|_{WL^{\vec{p}}(\rn)}:=\sup_{\alpha\in(0,\infty)}\,\lf\{\alpha\lf\|\mathbf1_{\{x\in\rn:\ |f(x)|>\alpha\}}\r\|_{L^{\vec{p}}(\rn)}\r\}<\infty.
$$
\end{definition}

Let $T$ be an operator defined on
$\mathscr M(\rn)$. Then $T$ is called a \emph{sublinear operator} if,
for any $f,\ g\in\mathscr M(\rn)$ and any $\lambda\in \cc$,
$$
|T(f + g)|\le |T(f)| + |T(g)|\quad \text{and}\quad|T(\lambda f)|=|\lambda||T(f)|.
$$
The interpolation theorem of operators on the mixed-norm Lebesgue space $L^{\vec{p}}(\rn)$ is stated as follows.

\begin{theorem}\label{mixTh1}
Let $\vec{p}\in(1,\infty)^n$. Let $r_1\in(\frac1{p_-},1)$ and $r_2\in(1,\infty)$.
Assume that $T$ is a sublinear operator defined on
$L^{r_1\vec{p}}(\rn)+L^{r_2\vec{p}}(\rn)$ satisfying that there exist
positive constants $C_1$ and $C_2$
such that, for any $i\in\{1,2\}$ and $f\in L^{r_i\vec{p}}(\rn)$,
\begin{equation}\label{mixEq1}
\lf\|T(f)\r\|_{WL^{r_i\vec{p}}(\rn)}
\le C_i\|f\|_{L^{r_i\vec{p}}(\rn)},
\end{equation}
where $r_i\vec{p}:=(r_ip_1,\ldots,r_ip_n)$ for any $i\in\{1,2\}$.
Then $T$ is bounded on $WL^{\vec{p}}(\rn)$
and there exists a positive constant $C$ such that, for any $f\in WL^{\vec{p}}(\rn)$,
$$
\|T(f)\|_{WL^{\vec{p}}(\rn)}\le C\|f\|_{WL^{\vec{p}}(\rn)}.
$$
\end{theorem}

\begin{proof}
Let $f\in WL^{\vec{p}}(\rn)$ and
$$
\lambda:=\|f\|_{WL^{\vec{p}}(\rn)}=\sup_{\alpha\in(0,\infty)}\lf\{\alpha\lf
\|\mathbf{1}_{\{x\in\rn:\ |f(x)|>\alpha\}}\r\|_{L^{\vec{p}}(\rn)}\r\}.
$$
We need to show that, for any $\alpha\in(0,\infty)$,
$$
\alpha\lf
\|\mathbf{1}_{\{x\in\rn:\ |Tf(x)|>\alpha\}}\r\|_{L^{\vec{p}}(\rn)}
\lesssim\lambda
$$
with the implicit positive constant independent of $\alpha$ and $f$.

To this end, for any $\alpha\in(0,\infty)$, let
$$
f^{(\alpha)}:=f\mathbf{1}_{\{x\in\rn:\ |f(x)|>\alpha\}}\quad\text{and}\quad
f_{(\alpha)}:=f\mathbf{1}_{\{x\in\rn:\ |f(x)|\le\alpha\}}.
$$
We claim that
\begin{equation}\label{mixEq2}
\lf\|f^{(\alpha)}\r\|_{L^{r_1\vec{p}}(\rn)}\lesssim
\alpha\lf(\lambda/\alpha\r)^{1/r_1}
\end{equation}
and
\begin{equation}\label{mixEq3}
\lf\|f_{(\alpha)}\r\|_{L^{r_2\vec{p}}(\rn)}
\lesssim \alpha\lf(\lambda/\alpha\r)^{1/r_2}.
\end{equation}
Assuming that this claim holds true for the moment, then, by the condition
that $T$ is sublinear and \eqref{mixEq1}, we conclude that, for any $\alpha\in(0,\infty)$,
\begin{align*}
&\alpha\lf\|\mathbf{1}_{\{x\in\rn:\ |T(f)(x)|>\alpha\}}\r\|_{L^{\vec{p}}(\rn)}\\
&\hspace*{12pt}\lesssim\alpha\lf\|\mathbf{1}_{\{x\in\rn:\ |T(f^{(\alpha)})(x)|>\alpha/2\}}
\r\|_{L^{\vec{p}}(\rn)}+\alpha\lf\|\mathbf{1}_{\{x\in\rn:\ |T(f_{(\alpha)})(x)|>\alpha/2\}}\r\|_{L^{\vec{p}}(\rn)}\\
&\hspace*{12pt}\sim\alpha\lf\|\mathbf{1}_{\{x\in\rn:\ |T(f^{(\alpha)})(x)|>\alpha/2\}}
\r\|_{L^{r_1\vec{p}}(\rn)}^{r_1}+\alpha\lf\|\mathbf{1}_{\{x\in\rn:\
|T(f_{(\alpha)})(x)|>\alpha/2\}}\r\|_{L^{r_2\vec{p}}(\rn)}^{r_2}\\
&\hspace*{12pt}\lesssim\alpha^{1-r_1}\lf\|f^{(\alpha)}\r\|_{L^{r_1\vec{p}}(\rn)}^{r_1}
+\alpha^{1-r_2}\lf\|f^{(\alpha)}\r\|_{L^{r_2\vec{p}}(\rn)}^{r_2}\lesssim\lambda.
\end{align*}
This implies that $\|T(f)\|_{WL^{\vec{p}}(\rn)}\lesssim \|f\|_{WL^{\vec{p}}(\rn)}$, which is the desired conclusion.

Therefore, it remains to prove the above claim.
To prove \eqref{mixEq2}, by the Minkowski inequality, we have
\begin{align*}
&\lf\|\frac{|f^{(\alpha)}|/\alpha}{(\lambda/\alpha)^{1/r_1}}
\r\|_{L^{r_1\vec{p}}(\rn)}
=\lf\|\int_0^{\infty}
\mathbf{1}_{\{y\in\rn:\
[|f^{(\alpha)}(y)|/\alpha]^{r_1}>\frac{\lambda\tau}{\alpha}\}}\,d\tau
\r\|_{L^{\vec{p}}(\rn)}^{\frac{1}{r_1}}\\
&\hspace*{12pt}\lesssim
\lf\{\int_0^{\infty}\lf\|
\mathbf{1}_{\{y\in\rn:\
[|f^{(\alpha)}(y)|/\alpha]^{r_1}>\frac{\lambda\tau}{\alpha}\}}
\r\|_{L^{\vec{p}}(\rn)}\,d\tau\r\}^\frac1{r_1}\\
&\hspace*{12pt}\lesssim\lf\{\int_0^{\alpha/\lambda}\lf\|
\mathbf{1}_{\{y\in\rn:\
[|f^{(\alpha)}(y)|/\alpha]^{r_1}>\frac{\lambda\tau}{\alpha}\}}
\r\|_{L^{\vec{p}}(\rn)}\, d\tau\r\}^\frac1{r_1}
+\lf\{\int_{\alpha/\lambda}^\infty\cdots d\tau\r\}^\frac1{r_1}
=:\mathrm{I_1}+\mathrm{I_2}.
\end{align*}
By the definition $f^{(\alpha)}$ and Definition \ref{wmix}, it is easy to see that
\begin{align*}
\mathrm{I_1}&\lesssim\lf\{\int_0^{\alpha/\lambda}\lf\|\mathbf{1}_{\{y\in\rn:\
|f(y)|>\alpha\}}\r\|_{L^{\vec{p}}(\rn)}\, d\tau\r\}^\frac1{r_1}
\lesssim\lf\{\frac{\alpha}{\lambda}\lf\|\mathbf{1}_{\{y\in\rn:\
|f(y)|>\alpha\}}\r\|_{L^{\vec{p}}(\rn)}\r\}^\frac1{r_1}
\lesssim 1.
\end{align*}
As for $\mathrm{I_2}$, from the definition $f^{(\alpha)}$,
Definition \ref{wmix} and $\frac1{r_1}>1$, it follows that
\begin{align*}
\mathrm{I_2}&\lesssim\lf\{\int_{\alpha/\lambda}^\infty\lf\|\mathbf{1}_{\{y\in\rn:\
|f(y)|>\alpha[\frac{\lambda\tau}{\alpha}]^{1/r_1}\}}
\r\|_{L^{\vec{p}}(\rn)}\, d\tau\r\}^\frac1{p_1}\\
&\lesssim\lf\{\int_{\alpha/\lambda}^\infty\lf[\alpha^{-1}\lf(\frac{\lambda\tau}{\alpha}\r)^{-\frac1{r_1}}\r]
\lambda\,d\tau\r\}^\frac1{r_1}\lesssim 1.
\end{align*}
Combining the estimates for $\mathrm{I_1}$ and $\mathrm{I_2}$, we then obtain \eqref{mixEq2}.

To prove \eqref{mixEq3}, by a proof similar to the estimation of \eqref{mixEq2}, we have

\begin{align*}
\lf\|
\frac{|f_{(\alpha)}|/\alpha}{(\lambda/\alpha)^{1/r_2}}
\r\|_{L^{r_2\vec{p}}(\rn)}
&\lesssim\lf\{\int_0^{\alpha/\lambda}\lf\|
\mathbf{1}_{\{y\in\rn:\
[|f_{(\alpha)}(y)|/\alpha]^{r_2}>\frac{\lambda\tau}{\alpha}\}}
\r\|_{L^{\vec{p}}(\rn)}\, d\tau\r\}^\frac1{r_2}
+\lf\{\int_{\alpha/\lambda}^\infty\cdots d\tau\r\}^\frac1{r_2}\\
&=:\mathrm{II_1}+\mathrm{II_2}.
\end{align*}
From the definition $f_{(\alpha)}$,
Definition \ref{wmix} and $0<\frac1{r_2}<1$, we deduce that
\begin{align*}
\mathrm{II_1}&\lesssim\lf\{\int_0^{\alpha/\lambda}\lf\|\mathbf{1}_{\{y\in\rn:\
|f(y)|>\alpha[\frac{\lambda\tau}{\alpha}]^{1/r_2}\}}
\r\|_{L^{\vec{p}}(\rn)}\, d\tau\r\}^\frac1{r_2}\\
&\lesssim\lf\{\int_0^{\alpha/\lambda}\lf[\alpha^{-1}\lf(\frac{\lambda\tau}{\alpha}\r)^{-\frac1{r_2}}\r]
\lambda\,d\tau\r\}^\frac1{r_2}\lesssim 1.
\end{align*}

Observe that, when $\tau\in(\frac{\alpha}{\lambda},\infty)$,
$(|f_{(\alpha)}|/\alpha)^{p_2}\le1<\frac{\tau\lambda}{\alpha}$ and hence $\mathrm{II_2}=0$,
which, together with the estimate for $\mathrm{II_1}$, implies \eqref{mixEq3}.
Thus, we complete the proof of our claim and hence of Theorem \ref{mixTh1}.
\end{proof}

We also need the following Fefferman--Stein
vector-valued maximal inequality
on $WL^{\vec{p}}(\rn)$.

\begin{theorem}\label{mixProfs}
Let $\vec{p}\in(1,\infty)^n$ and $s\in(1,\infty)$.
Then there exists a positive constant $C$
such that, for any sequence $\{f_j\}_{j\in\nn}\subset\mathscr{M}(\rn)$,
$$
\lf\|\lf\{\sum_{j\in\nn}[\cm(f_j)]^s\r\}^\frac1s\r\|_{WL^{\vec{p}}(\rn)}
\le C\lf\|\lf\{\sum_{j\in\nn}|f_j|^s\r\}^\frac1s\r\|_{WL^{\vec{p}}(\rn)}.
$$
\end{theorem}

\begin{proof}
Let $\{f_j\}_{j\in\nn}$ be a given arbitrary sequence of measurable functions and,
for any measurable function $g$ and $x\in\rn$, define
$$
A(g)(x):=\lf\{\sum_{j\in\nn}[\cm(g\eta_j)(x)]^s\r\}^\frac1s,
$$
where $s\in(1,\infty)$ and, for any $i\in\nn$ and $y\in\rn$,
$$
\eta_j(y):=\frac{f_j(y)}{[\sum_{j\in\nn}|f_j(y)|^s]^{1/s}}
\quad\text{when}\quad\lf[\sum_{j\in\nn}|f_j(y)|^s\r]^{1/s}\neq0,
$$
and $\eta_j(y):=0$ otherwise. By the Minkowski inequality, we conclude that,
for any $\lambda\in\mathbb{C}$
and $g_1,\ g_2\in\mathscr M(\rn)$,
$$
A(g_1+g_2)\le A(g_1)+A(g_2)\quad\text{and}\quad A(\lambda g)=|\lambda|A(g).
$$
Thus, $A$ is sublinear. For any $\vec{p}\in(1,\infty)^n$ and $s\in(1,\infty)$,
from Lemma \ref{mixo1}, we deduce that
$$
\lf\|\lf\{\sum_{j\in\nn}[\cm(f_j)]^s\r\}^\frac1s\r\|_{L^{\vec{p}}(\rn)}\lesssim \lf\|\lf\{\sum_{j\in\nn}|f_j|^s\r\}^\frac1s\r\|_{L^{\vec{p}}(\rn)}.
$$
Using this, we know that, for any given $r_1\in(\frac1{\min\{p_\Phi^-,q\}},1)$ and
$r_2\in(1,\infty)$ and any $h\in\mathscr M(\rn)$,
\begin{align*}
\|A(h)\|_{WL^{r_i\vec{p}}(\rn)}&=\lf\|\lf\{\sum_{j\in\nn}[\cm(h\eta_j)]^s
\r\}^\frac1s\r\|_{WL^{r_i\vec{p}}(\rn)}
\leq\lf\|\lf\{\sum_{j\in\nn}[\cm(h\eta_j)]^s
\r\}^\frac1s\r\|_{L^{r_i\vec{p}}(\rn)}\\
&\lesssim\lf\|\lf\{\sum_{j\in\nn}|h\eta_j|^s\r\}^\frac1s\r\|_{L^{r_i\vec{p}}(\rn)}
\sim\|h\|_{L^{r_i\vec{p}}(\rn)},
\end{align*}
which implies that the operator $A$ is bounded on $WL^{r_i\vec{p}}(\rn)$, where $i\in\{1,2\}$.
Now, taking $g:=[\sum_{j\in\nn}|f_j|^s]^{1/s}$, then, by Theorem \ref{mixTh1}, we conclude that
\begin{align*}
\lf\|\lf\{\sum_{j\in\nn}[\cm(f_j)]^s\r\}^\frac1s
\r\|_{WL^{\vec{p}}(\rn)}&=\|A(g)\|_{WL^{\vec{p}}(\rn)}
\lesssim\|g\|_{WL^{\vec{p}}(\rn)}\sim
\lf\|\lf\{\sum_{j\in\nn}|f_j|^s\r\}^\frac1s\r\|_{WL^{\vec{p}}(\rn)},
\end{align*}
which completes the proof of Theorem \ref{mixProfs}.
\end{proof}

By \cite[Lemma 3.5]{hlyy} and \cite[Theorem 1.a]{bp}, we can easily obtain
the following conclusion and we omit
the details here.

\begin{lemma}\label{mixmM}
Let $\vec{p}\in(0,\infty)^n$,
$r\in (0,p_-]$ and $s\in(p_+,\infty]$. Then there exists
a positive constant $C$
such that, for any $f\in\mathscr{M}(\rn)$,
\begin{equation*}
\lf\|\mathcal{M}^{((s/r)')}(f) \r\|_{([L^{\vec{p}}(\rn)]^{1/r})'}\leq C\lf
\|f\r\|_{([L^{\vec{p}}(\rn)]^{1/r})'},
\end{equation*}
where $([L^{\vec{p}}(\rn)]^{1/r})'$ is as in \eqref{asso} with
$X:=[L^{\vec{p}}(\rn)]^{1/r}$.
\end{lemma}

Now we give the notion of the weak mixed-norm Hardy space $WH^{\vec{p}}(\rn)$.

\begin{definition}\label{mixmha}
Let $\vec{p}\in(0,\infty)^n$.
The \emph{weak mixed-norm Hardy space $WH^{\vec{p}}(\rn)$}
is defined to be the set of all $f\in\cs'(\rn)$
such that $\|f\|_{WH^{\vec{p}}(\rn)}:=\
\|M_N^0(f)\|_{WL^{\vec{p}}(\rn)}<\infty$,
where $M_N^0(f)$ is as in \eqref{EqMN0} with $N$ sufficiently large.
\end{definition}

\begin{remark}
Let $\vec{p}\in(1,\infty)^n$. By Theorem \ref{mixProfs},
we conclude that, for any $r\in(1,p_-)$,
$\cm$ in \eqref{mm} is bounded on $(WL^{\vec{p}}(\rn))^{1/r}$,
which, combined with Theorem \ref{dayu2}, implies that
$WH^{\vec{p}}(\rn)=WL^{\vec{p}}(\rn)$ with equivalent norms.
\end{remark}

By Lemma \ref{mixProfs}  and Theorem \ref{Thmc}(ii),
we obtain the following maximal function characterizations of the
weak mixed-norm Hardy space $WH^{\vec{p}}(\rn)$.

\begin{theorem}
Let $\vec{p}\in(0,\infty)^n$
and $\psi\in\mathcal{S}(\rn)$ satisfy $\int_{\rn}\psi(x)\,dx\neq0.$
Assume that $b\in(n/p_-,\infty)$ and
$N\ge\lfloor b+1\rfloor$. For any $f\in\mathcal{S}'(\rn)$,
if one of the following quantities
$$
\lf\|M_N^0(f)\r\|_{WL^{\vec{p}}(\rn)},\  \lf\|M(f,\psi)\r\|_{WL^{\vec{p}}(\rn)},\  \lf\|M_a^*(f,\psi)\r\|_{WL^{\vec{p}}(\rn)},\ \lf\|M_N(f)\r\|_{WL^{\vec{p}}(\rn)},
$$
$$
\lf\|M_b^{**}(f,\psi)\r\|_{WL^{\vec{p}}(\rn)},\ \lf\|M_{b,\ N}^{**}(f)\r\|_{WL^{\vec{p}}(\rn)}
\quad and \quad\lf\|\cn(f)\r\|_{WL^{\vec{p}}(\rn)}
$$
is finite, then the others are also finite and mutually equivalent with the implicit
positive equivalence constants independent of $f$.
\end{theorem}

Using Lemmas \ref{mixo1}, \ref{mixProfs} and \ref{mixmM},
and Theorems \ref{Thad} and \ref{Thar},
we immediately obtain the atomic characterization of $WH^{\vec{p}}(\rn)$
and the molecular characterization of $WH^{\vec{p}}(\rn)$, respectively, as follows.

\begin{theorem}\label{mixMth2}
Let $\vec{p}\in(0,\infty)^n$, $r\in(\max\{1,p_+\},\infty)$
and $d\in\zz_+$ with $d\geq\lfloor n(\frac{1}{\min\{1,p_-/\max\{1,p_+\}\}}-1)\rfloor$.
Then $f\in WH^{\vec{p}}(\rn)$ if and only if
$$
f=\sum_{i\in \zz}\sum_{j\in\nn}\lambda_{i,j}a_{i,j}\quad\text{in}\quad\cs'(\rn)
\quad \text{and}\quad
\sup_{i\in\zz}\lf\|\sum_{j\in\nn}
\frac{\lambda_{i,j}\mathbf{1}_{B_{i,j}}}
{\|\mathbf{1}_{B_{i,j}}\|_{L^{\vec{p}}(\rn)}}\r\|_{L^{\vec{p}}(\rn)}<\infty,
$$
where $\{a_{i,j}\}_{i\in\zz,j\in\nn}$ is a sequence of $(L^{\vec{p}}(\rn),\,r,\,d)$-atoms
supported, respectively, in balls
$\{B_{i,j}\}_{\gfz{i\in\zz}{j\in\nn}}$ such that, for any
$i\in\zz$, $\sum_{j\in\nn}\mathbf{1}_{cB_{i,j}}\le A$ with $c\in(0,1]$ and $A$
being a positive constant independent of $f$ and $i$, and,
for any $i\in\zz$ and $j\in\nn$,
$\lambda_{i,j}:=\widetilde A2^i\|\mathbf{1}_{B_{i,j}}\|_{L^{\vec{p}}(\rn)}$
with $\widetilde A$ being a positive constant independent of $f$ and $i$.

Moreover, for any $f\in WH^{\vec{p}}(\rn)$,
$$
\|f\|_{WH^{\vec{p}}(\rn)}\sim\inf\lf\{\sup_{i\in\zz}\lf\|\sum_{j\in\nn}
\frac{\lambda_{i,j}\mathbf{1}_{B_{i,j}}}
{\|\mathbf{1}_{B_{i,j}}\|_{L^{\vec{p}}(\rn)}}\r\|_{L^{\vec{p}}(\rn)}\r\},
$$
where the infimum is taken over all decompositions of $f$ as above
and the positive equivalence constant is independent of $f$.
\end{theorem}

\begin{theorem}\label{mixfenzi}
Let $\vec{p}$,
$r$ and $d$ be the same as in Theorem \ref{mixMth2}, and $\epsilon\in(n+d+1,\infty)$.
Then $f\in WH^{\vec{p}}(\rn)$ if and only if
$$
f=\sum_{i\in \zz}\sum_{j\in\nn}\lambda_{i,j}m_{i,j}\quad\text{in}\quad\cs'(\rn)
\quad\text{and}\quad
\sup_{i\in\zz}\lf\|\sum_{j\in\nn}
\frac{\lambda_{i,j}\mathbf{1}_{B_{i,j}}}
{\|\mathbf{1}_{B_{i,j}}\|_{L^{\vec{p}}(\rn)}}\r\|_{L^{\vec{p}}(\rn)}<\infty,
$$
where $\{m_{i,j}\}_{i\in\zz,j\in\nn}$ is a sequence of $(L^{\vec{p}}(\rn),r,d,\epsilon)$-\emph{molecules} associated, respectively, with balls
$\{B_{i,j}\}_{i\in\zz,j\in\nn}$ such that, for any
$i\in\zz$, $\sum_{j\in\nn}\mathbf{1}_{cB_{i,j}}\le A$ with $c\in(0,1]$ and
$A$ being a positive constant independent of $f$
and $i$, and,
for any $i\in\zz$ and $j\in\nn$, $\lambda_{i,j}:=\widetilde A2^i\|\mathbf{1}_{B_{i,j}}\|_{L^{\vec{p}}(\rn)}$
with $\widetilde A$ being a positive constant independent of $f$, $i$ and $j$.

Moreover, for any $f\in WH^{\vec{p}}(\rn)$,
$$
\|f\|_{WH^{\vec{p}}(\rn)}\sim\inf\lf[\sup_{i\in\zz}\lf\|\sum_{j\in\nn}
\frac{\lambda_{i,j}\mathbf{1}_{B_{i,j}}}{\|\mathbf{1}_{B_{i,j}}\|_{L^{\vec{p}}(\rn)}}
\r\|_{L^{\vec{p}}(\rn)}\r],
$$
where the infimum is taken over all decompositions of $f$ as above
and the positive equivalence constants are independent of $f$.
\end{theorem}

Now, we recall the following notion of the mixed-norm Hardy space.

\begin{definition}\label{mixMhardy}
Let $\vec{p}\in(0,\infty)^n$.
The \emph{mixed-norm Hardy
space $H^{\vec{p}}(\rn)$} is defined to be the set of all $f\in\cs'(\rn)$
such that $\|f\|_{H^{\vec{p}}(\rn)}:=\
\|M_N^0(f)\|_{L^{\vec{p}}(\rn)}<\infty$,
where $M_N^0(f)$ is as in \eqref{EqMN0} with $N$ sufficiently large.
\end{definition}

To discuss the boundedness of Clader\'on--Zygmund operators from $H^{\vec{p}}(\rn)$ to $WH^{\vec{p}}(\rn)$,
we need the following vector-valued inequality of the Hardy--Littlewood maximal operator $\cm$ in \eqref{mm}
from  $L^{\vec{p}}(\rn)$ to $WL^{\vec{p}}(\rn)$.

\begin{proposition}\label{mixPro}
Let $\vec{p}\in[1,\infty)^n$ and $r\in(1,\infty)$. Then there exists a positive constant $C$
such that,
for any $\{f_j\}_{j\in\nn}\subset \mathscr{M}(\rn)$,
$$
\lf\|\lf\{\sum_{j\in\nn}\lf[\cm(f_j)\r]^r\r\}^\frac1r\r\|_{WL^{\vec{p}}(\rn)}\le
C\lf\| \lf \{ \sum_{j\in\nn}|f_j|^r \r\}^\frac1r \r\|_{L^{\vec{p}}(\rn)}.
$$
\end{proposition}

To prove Proposition \ref{mixPro}, we need the following extrapolation theorem,
which is a slight variant of a special case of \cite[Theorem 4.6]{CMP} via
replacing Banach function spaces by ball Banach function spaces.
Recall that an \emph{$A_1(\rn)$-weight $\omega$}
(see, for instance, \cite[Definition 7.1.1]{G1}) is a
locally integrable and nonnegative function satisfying that
$$
[\omega]_{A_1(\rn)}:=\sup_{B\in\BB}\frac1{|B|}\int_B\omega(x)\,dx\lf[\lf\|\omega^{-1}\r\|_{L^\infty(B)}\r]<\infty,
$$
where $\BB$ is as in \eqref{Eqball}.

\begin{lemma}\label{waicha}
Let $X$ be a ball Banach function space and $p_0\in(0,\infty)$.
Let $\mathcal{F}$ be the set of all
pairs of nonnegative measurable
functions $(F,G)$ such that, for any given $\omega\in A_1(\rn)$,
$$
\int_{\rn}[F(x)]^{p_0}\omega(x)\,dx\leq C_{(p_0,[\omega]_{A_1(\rn)})}
\int_{\rn}[G(x)]^{p_0}\omega(x)\,dx,
$$
where $C_{(p_0,[\omega]_{A_1(\rn)})}$ is a positive constant independent
of $(F,G)$, but depends on $p_0$ and $[\omega]_{A_1(\rn)}$.
Assume that there exists $q_0\in[p_0,\infty)$ such that $X^{1/q_0}$ is
a Banach function space and $\cm$ is bounded on $(X^{1/q_0})'$. Then there
exists a positive constant $C$ such that, for any $(F,G)\in\mathcal{F}$,
$$
\|F\|_{X}\leq C\|G\|_{X}.
$$
\end{lemma}

\begin{proof} We observe that a key fact used in the proof of \cite[Theorem 4.6]{CMP}
is that, if $X$ is a Banach function space as in Definition \ref{bn},
then $X=X''$ with the same norms. However, if $X$ is just a ball Banach function space
as in the lemma, by Lemma \ref{Lesdual}, we know that this fact also holds true.
Thus, using this fact and repeating the proof of \cite[Theorem 4.6]{CMP},
we then complete the proof of Lemma \ref{waicha}.
\end{proof}

We still
need the following weak-type weighted Fefferman--Stein vector-valued inequality of
the Hardy--Littlewood maximal operator $\cm$ in \eqref{mm} from \cite[Theorem 3.1(a)]{AJ}.

\begin{lemma}\label{ruom}
Let $\omega\in A_1(\rn)$ and $r\in(1,\infty)$. Then there exists a positive constant $C$,
depending on $p_0$ and $[\omega]_{A_1(\rn)}$,
such that,
for any $\alpha\in(0,\infty)$ and $\{f_j\}_{j\in\nn}\subset \mathscr{M}(\rn)$,
$$
\alpha\int_{\rn}
\mathbf{1}_{\{y\in\rn:\ \{\sum_{j\in\nn}[\cm(f_j)(y)]^r\}^\frac1r>\alpha\}}(x)
\omega(x)\,dx\le
C\int_{\rn} \lf [ \sum_{j\in\nn}\lf|f_j(x)\r|^r \r]^\frac1r \omega(x)\,dx.
$$
\end{lemma}

\begin{proof}[Proof of Proposition \ref{mixPro}]
For any given $r\in(1,\infty)$, let
$$
\mathcal{F}:=
\lf\{\lf(\alpha\mathbf{1}_{\{y\in\rn:\ \{\sum_{j\in\nn}[\cm(f_j)(y)]^r\}^\frac1r>\alpha\}},
\lf [ \sum_{j\in\nn}\lf|f_j\r|^r \r]^\frac1r\r):\ \alpha\in(0,\infty),\ \{f_j\}_{j\in\nn}\subset
\mathscr{M}(\rn)\r\}.
$$
Then, by Lemma \ref{ruom}, we conclude that, for any given $\omega\in A_1(\rn)$
and any $(F,G)\in\mathcal{F}$,
\begin{equation}\label{hha}
\int_{\rn}F(x)\omega(x)\,dx\lesssim\int_{\rn}G(x)\omega(x)\,dx.
\end{equation}
Let $\vec{p}\in[1,\infty)^n$. From \cite[Theorem 1.a]{bp} and \cite[Lemma 3.5]{hlyy},
it follows that $\cm$ as in \eqref{mm} is bounded on $(L^{\vec{p}}(\rn))'$.
By this and \eqref{hha}, applying Lemma \ref{waicha} with $p_0:=1$
and the fact that $L^{\vec{p}}(\rn)$ is a Banach function space,
we conclude that, for any $(F,G)\in\mathcal{F}$, $\|F\|_{L^{\vec{p}}(\rn)}\lesssim\|G\|_{L^{\vec{p}}(\rn)}$.
Thus, for any $\{f_j\}_{j\in\nn}\subset \mathscr{M}(\rn)$,
$$
\lf\|\lf\{\sum_{j\in\nn}\lf[\cm(f_j)\r]^r\r\}^\frac1r\r\|_{WL^{\vec{p}}(\rn)}\lesssim
\lf\| \lf \{ \sum_{j\in\nn}|f_j|^r \r\}^\frac1r \r\|_{L^{\vec{p}}(\rn)},
$$
which completes the proof of Proposition \ref{mixPro}.
\end{proof}

Applying Proposition \ref{mixPro}, Lemmas \ref{mixmM} and \ref{mixProfs}, Theorems \ref{Thcz} and \ref{Thcz1},
we immediately obtain the boundedness from
$H^{\vec{p}}(\rn)$ to $WH^{\vec{p}}(\rn)$
of both convolutional $\delta$-type and $\gamma$-type
Calder\'on--Zygmund operators, respectively, as follows.

\begin{theorem}
Let $\vec{p}\in(0,\infty)^n$ and $\delta\in(0,1]$.
Let $T$ be a convolutional $\delta$-type Calder\'on--Zygmund
operator. If $p_-\in[\frac{n}{n+\delta},1]$, then $T$ has a unique extension on $H^{\vec{p}}(\rn)$ and,
moreover, there exists a positive constant $C$ such that,
for any $f\in H^{\vec{p}}(\rn)$,
$$
\|Tf\|_{WH^{\vec{p}}(\rn)}\le C\|f\|_{H^{\vec{p}}(\rn)}.
$$
\end{theorem}

\begin{theorem}
Let $\vec{p}\in(0,2)^n$ and $\gamma\in(0,\infty)$.
Let $T$ be a $\gamma$-type Calder\'on--Zygmund
operator and have the vanishing moments up to order $\lceil\gamma\rceil-1$.
If $\lceil\gamma\rceil-1\le n(\frac1{p_-}-1)\le\gamma$, then $T$ has a unique extension on
$H^{\vec{p}}(\rn)$ and, moreover, there exists a
positive constant $C$ such that,
for any $f\in H^{\vec{p}}(\rn)$,
$$
\|Tf\|_{WH^{\vec{p}}(\rn)}\le C\|f\|_{H^{\vec{p}}(\rn)},
$$
where $C$ is a positive constant independent of $f$.
\end{theorem}

\subsection{Orlicz-slice spaces\label{s7.3}}

We begin with the notions of both Orlicz functions and Orlicz spaces (see, for instance, \cite{RR}).

\begin{definition}
A function $\Phi:\ [0,\infty)\ \rightarrow\ [0,\infty)$ is called an \emph{Orlicz function} if it is
non-decreasing and satisfies $\Phi(0)= 0$, $\Phi(t)>0$ whenever $t\in(0,\infty)$ and $\lim_{t\rightarrow\infty}\Phi(t)=\infty$.
\end{definition}

An Orlicz function $\Phi$  is said to be
of \emph{lower} (resp., \emph{upper}) \emph{type} $p$ with
$p\in(-\infty,\infty)$ if
there exists a positive constant $C_{(p)}$, depending on $p$, such that, for any $t\in[0,\infty)$
and $s\in(0,1)$ [resp., $s\in [1,\infty)$],
\begin{equation*}
\Phi(st)\le C_{(p)}s^p \Phi(t).
\end{equation*}
A function $\Phi:\ [0,\infty)\ \rightarrow\ [0,\infty)$ is said to be of
 \emph{positive lower} $p$ (resp., \emph{upper}) \emph{type} if it is of lower
 (resp., upper) type $p$ for some $p\in(0,\infty)$.

\begin{definition}\label{DeO}
Let $\Phi$ be an Orlicz function with positive lower type $p_{\Phi}^-$ and positive upper type $p_{\Phi}^+$.
The \emph{Orlicz space $L^\Phi(\rn)$} is defined
to be the set of all measurable functions $f$ such that
 $$\|f\|_{L^\Phi(\rn)}:=\inf\lf\{\lambda\in(0,\infty):\ \int_{\rn}\Phi\lf(\frac{|f(x)|}{\lambda}\r)\,dx\le1\r\}<\infty.$$
\end{definition}

\begin{remark}
\begin{itemize}
\item[(i)] Let $\Phi$ be an Orlicz function with positive lower type $p_{\Phi}^-$ and positive upper type $p_{\Phi}^+$.
In what follows, for any given $s\in(0,\infty)$, let
$\Phi_s(\tau):=\Phi(\tau^s)$ for any $\tau\in(0,\infty)$. Then $\Phi_s$ is also
an Orlicz function with lower type $sp_{\Phi}^-$ and upper type $sp_{\Phi}^+$. Moreover,
for any measurable function $f$ such that $|f|^s\in L^\Phi(\rn)$, we have
\begin{equation*}
\lf\|\lf|f\r|^s\r\|_{L^\Phi(\rn)}=\lf\|f\r\|_{L^{\Phi_s}(\rn)}^s.
\end{equation*}
\item[(ii)] Let $\Phi$ be as in (i) of this remark.
By \cite[Lemma 2.5]{ZYYW}, we may \emph{always assume} that $\Phi$
is continuous and strictly increasing. Let $\Phi^{-1}$ be the \emph{inverse function} of $\Phi$.
Observe that, for any $x\in\rn$ and $t\in(0,\infty)$,
\begin{equation}\label{Eq22}
\|\mathbf{1}_{B(x,t)}\|_{L^\Phi(\rn)}=\lf[\Phi^{-1}\lf(\lf|B(x,t)\r|^{-1}\r)\r]^{-1}=\lf[\Phi^{-1}
\lf(\varepsilon_nt^{-n}\r)\r]^{-1}=:\widetilde C_{(\Phi,t)}
\end{equation}
is independent of $x\in\rn$, where $\varepsilon_n$ denotes the \emph{volume of the unit ball of $\rn$}.
\end{itemize}
\end{remark}

We also recall some notions on the Young function.
A convex function $\Phi:\ [0,\infty)\  \rightarrow\  [0,\infty)$ is called a
\emph{Young function} if $\Phi$ is non-decreasing, $\Phi(0)= 0$ and
$\lim_{t\rightarrow\infty}\Phi(t)=\infty$. For any Young function $\Phi$,
its \emph{complementary function $\Psi:\ [0,\infty)\rightarrow[0,\infty)$}
is defined by setting, for any $y\in[0,\infty)$
$$
\Psi(y):=\sup\lf\{xy-\Phi(x):\ x\in[0,\infty)\r\}.
$$

\begin{remark}\label{Re83}
Let $\Phi$ be an Orlicz function with lower type $p_{\Phi}^-\in[1,\infty)$ and positive upper type $p_{\Phi}^+$.
By \cite[p.\,67, Theorem 10]{RR}, we know that
$L^\Phi(\rn)$ is a Banach space.
\end{remark}

The following notion of Orlicz-slice spaces
was introduced by Zhang et al. \cite{ZYYW}, which is a generalization of the slice spaces
proposed by Auscher and Mourgoglou \cite{AM2014} and  Auscher and  Prisuelos-Arribas \cite{APA}.

\begin{definition}\label{DeOs}
Let $t,\ q\in(0,\infty)$ and $\Phi$ be an Orlicz function with positive lower type $p_{\Phi}^-$ and
positive upper type $p_{\Phi}^+$. The \emph{Orlicz-slice space} $(E_\Phi^q)_t(\rn)$
is defined to be the set of all measurable functions $f$
such that
$$
\|f\|_{(E_\Phi^q)_t(\rn)}
:=\lf\{\int_{\rn}\lf[\frac{\|f\mathbf{1}_{B(x,t)}\|_{L^\Phi(\rn)}}
{\|\mathbf{1}_{B(x,t)}\|_{L^\Phi(\rn)}}\r]^q\,dx\r\}^{\frac{1}{q}}<\infty.
$$
\end{definition}

\begin{remark}\label{Re25}
Let $t,\ q\in(0,\infty)$ and $\Phi$ be an Orlicz function with positive lower type $p_{\Phi}^-$ and
positive upper type $p_{\Phi}^+$.
\begin{itemize}
\item[(i)] By \cite[Lemma 2.28]{ZYYW}, we know that the Orlicz-slice space $(E_\Phi^q)_t(\rn)$
 is a ball quasi-Banach space. It is worth pointing out that
$(E_\Phi^q)_t(\rn)$ with $q\in(1,\infty)$ and $p_{\Phi}^-\in(1,\infty)$ may not be a Banach function space.
For instance, let $t:=1$, $q:=1$, $n:=1$ and $\Phi(\tau):=\tau^{2}$
for any $\tau\in [0,\fz)$.
In this case, by \cite[Proposition 2.12]{ZYYW}, we know that $(E_\Phi^q)_t(\rr)$
and $\ell^1(L^2)(\rr)$ (see, for instance, \cite{AF}) coincide with equivalent norms.
Let
$$E:=\bigcup_{m\in\nn}[m,m+1/m^2].$$
Then it is easy to show that $|E|<\infty$, but
\begin{align*}
\|\mathbf{1}_{E}\|_{(E_\Phi^q)_t(\rr)}\sim\|\mathbf{1}_{E}\|_{\ell^1(L^2)(\rr)}
\sim\sum_{k\in \mathbb{Z}}\lf\|\mathbf{1}_{E}\r\|_{L^2(Q_{k})}
\sim\sum_{k\in \nn}1/k=\infty,
\end{align*}
where $Q_{k}:=k+[0,1)$ for any $k\in\mathbb{Z}$.
Thus, $\ell^1(L^2)(\rr)$ does not satisfy Definition \ref{bn}(iv), which means
that $\ell^1(L^2)(\rr)$ is not a Banach function space.

\item[(ii)] Let $\Phi(\tau):=\tau^r$ for any $\tau\in[0,\infty)$ with
any given $r\in(0,\infty)$. Then $(E_\Phi^q)_t(\rn)$
and $(E_r^q)_t(\rn)$ from \cite{AM2014,APA} coincide with equivalent quasi-norms. Moreover,
in this case, if $q\in(0,r]$,
for any $f\in(E_r^q)_t(\rn)$, then $f\in L^q(\rn)$ and
$\|f\|_{L^q(\rn)}\le\|f\|_{(E_r^q)_t(\rn)}$; if $r\in(0,q]$,
for any $f\in L^r(\rn)\cup L^q(\rn)$, then $f\in (E_r^q)_t(\rn)$ and
$\|f\|_{(E_r^q)_t(\rn)}\le\min\{\|f\|_{L^r(\rn)}, \|f\|_{L^q(\rn)}\}$.
Thus, $\|f\|_{L^p(\rn)}=\|f\|_{(E_p^p)_t(\rn)}$ for any $p\in(0,\infty)$;
see \cite[Proposition 2.11]{ZYYW}.
\end{itemize}
\end{remark}

\begin{definition}\label{DewOs}
Let $t,\ q\in(0,\infty)$ and $\Phi$ be an Orlicz function with positive lower type $p_{\Phi}^-$ and
positive upper type $p_{\Phi}^+$. The \emph{weak Orlicz-slice space} $(WE_\Phi^q)_t(\rn)$
is defined to be the set of all measurable functions $f$ such that
$$
\|f\|_{(WE_\Phi^q)_t(\rn)}:=\sup_{\alpha\in(0,\infty)}\lf\{\alpha\lf\|
\mathbf{1}_{\{x\in\rn:\ |f(x)|>\alpha\}}\r\|_{(E_\Phi^q)_t(\rn)}\r\}<\infty.
$$
\end{definition}

To establish a Fefferman--Stein vector-valued inequality on $(WE_\Phi^{q})_t(\rn)$,
we need first to establish an interpolation theorem, in the spirit of the Marcinkiewicz interpolation theorem.
To this end, we now establish the following Minkowski type inequality.

\begin{lemma}\label{LeMI}
Let $t\in(0,\infty)$ and $\Phi$ be an Orlicz function with lower type $p_{\Phi}^-\in(1,\infty)$
and positive upper type $p_{\Phi}^+$.
Suppose that a measurable function $F$
is defined on $\rn\times\rr^m$.
If, for almost every $x\in\rn$, $F(x,\cdot)\in L^1(\rr^m)$ and, for almost every
$y\in\rn$, $F(\cdot,y)\in L^\Phi(\rn)$, then
$$
\lf\|\int_{\rr^m}|F(\cdot,y)|\,dy\r\|_{L^\Phi(\rn)}\le\int_{\rr^m}\lf\|F(\cdot,y)\r\|_{L^\Phi(\rn)}\,dy.
$$
\end{lemma}
\begin{proof}
Let $\Phi$ be as in the lemma and $\Psi$ the complementary function of
$\Phi$. By \cite[p.\,61, Proposition 4 and p.\,81, Proposition 10]{RR}, we have
\begin{align*}
&\lf\|\int_{\rr^m}|F(\cdot,y)|\,dy\r\|_{L^\Phi(\rn)}\\
&\hspace*{12pt}\sim
\sup\lf\{\lf|\int_\rn\int_{\rr^m}|F(x,y)|\,dyg(x)\,dx\r|
:\  g\in L^\Psi(\rn)\text{ such that }\|g\|_{L^\Psi(\rn)}=1\r\}.
\end{align*}
From the Fubini theorem and \cite[p.\,58, Proposition 1]{RR}, it follows that
\begin{align*}
\lf|\int_\rn\int_{\rr^m}F(x,y)\,dyg(x)\,dx\r|&\le\int_\rn\int_{\rr^m}|F(x,y)||g(x)|\,dy\,dx
=\int_{\rr^m}\int_\rn|F(x,y)||g(x)|\,dx\,dy\\
&\lesssim\int_{\rr^m}\lf\|F(\cdot,y)\r\|_{L^\Phi(\rn)}\lf\|g\r\|_{L^\Psi(\rn)}\,dy
\sim\int_{\rr^m}\|F(\cdot,y)\|_{L^\Phi(\rn)}\,dy,
\end{align*}
which implies the desired conclusion. This finishes the proof of Lemma \ref{LeMI}.
\end{proof}

The interpolation theorem of operators on Orlicz-slice spaces is stated as follows.
\begin{theorem}\label{Th1}
Let $t\in(0,\infty)$, $q\in(1,\infty)$ and $\Phi$ be an Orlicz function with positive lower type $p_{\Phi}^-\in(1,\infty)$ and
positive upper type $p_{\Phi}^+$. Let $p_1\in(\frac1{\min\{p_\Phi^-,\, q\}},1)$ and $p_2\in(1,\infty)$.
Assume that $T$ is a sublinear operator defined on $(E_{\Phi_{p_1}}^{p_1q})_t(\rn)+(E_{\Phi_{p_2}}^{p_2q})_t(\rn)$ satisfying that there exist
positive constants $C_1$ and $C_2$, independent of $t$,
 such that, for any $i\in\{1,2\}$ and $f\in(E_{\Phi_{p_i}}^{p_iq})_t(\rn)$,
\begin{equation}\label{Eq1}
\lf\|T(f)\r\|_{(WE_{\Phi_{p_i}}^{p_iq})_t(\rn)}
\le C_i\|f\|_{(E_{\Phi_{p_i}}^{p_iq})_t(\rn)},
\end{equation}
where $\Phi_{p_i}(\tau):=\Phi(\tau^{p_i})$ for any $\tau\in[0,\infty)$ and $i\in\{1,2\}$.
Then $T$ is bounded on $(WE_{\Phi}^q)_t(\rn)$
and there exists a positive constant $C$,
independent of $t$, such that, for any $f\in(WE_{\Phi}^q)_t(\rn)$,
$$
\|T(f)\|_{(WE_{\Phi}^q)_t(\rn)}\le C\|f\|_{(WE_{\Phi}^q)_t(\rn)}.
$$
\end{theorem}
\begin{proof}
Let $f\in (WE_\Phi^q)_t(\rn)$ and
$$
\lambda:=\|f\|_{(WE_\Phi^q)_t(\rn)}=\sup_{\alpha\in(0,\infty)}\lf\{\alpha\lf
\|\mathbf{1}_{\{x\in\rn:\ |f(x)|>\alpha\}}\r\|_{(E_\Phi^q)_t(\rn)}\r\}.
$$
We need to show that, for any $\alpha\in(0,\infty)$,
$$
\alpha\lf
\|\mathbf{1}_{\{x\in\rn:\ |Tf(x)|>\alpha\}}\r\|_{(E_\Phi^q)_t(\rn)}
\lesssim\lambda
$$
with the implicit positive constant independent of $\alpha$, $f$ and $t$.

To this end, for any $\alpha\in(0,\infty)$, let
$$
f^{(\alpha)}:=f\mathbf{1}_{\{x\in\rn:\ |f(x)|>\alpha\}}\quad\text{and}\quad
f_{(\alpha)}:=f\mathbf{1}_{\{x\in\rn:\ |f(x)|\le\alpha\}}.
$$
We claim that
\begin{equation}\label{Eq2}
\lf\|f^{(\alpha)}\r\|_{(E_{\Phi_{p_1}}^{p_1q})_t(\rn)}\lesssim
\alpha\lf(\lambda/\alpha\r)^{1/p_1}
\end{equation}
and
\begin{equation}\label{Eq3}
\lf\|f_{(\alpha)}\r\|_{(E_{\Phi_{p_2}}^{p_2q})_t(\rn)}
\lesssim \alpha\lf(\lambda/\alpha\r)^{1/p_2}.
\end{equation}
Assuming that this claim holds true for the moment, then, by the condition
that $T$ is sublinear and \eqref{Eq1}, we conclude that, for any $\alpha\in(0,\infty)$,
\begin{align*}
&\alpha\lf\|\mathbf{1}_{\{x\in\rn:\ |T(f)(x)|>\alpha\}}\r\|_{(E_\Phi^q)_t(\rn)}\\
&\hspace*{12pt}\lesssim\alpha\lf\|\mathbf{1}_{\{x\in\rn:\ |T(f^{(\alpha)})(x)|>\alpha/2\}}
\r\|_{(E_\Phi^q)_t(\rn)}+\alpha\lf\|\mathbf{1}_{\{x\in\rn:\ |T(f_{(\alpha)})(x)|>\alpha/2\}}\r\|_{(E_\Phi^q)_t(\rn)}\\
&\hspace*{12pt}\sim\alpha\lf\|\mathbf{1}_{\{x\in\rn:\ |T(f^{(\alpha)})(x)|>\alpha/2\}}
\r\|_{(E_{\Phi_{p_1}}^{p_1q})_t(\rn)}^{p_1}+\alpha\lf\|\mathbf{1}_{\{x\in\rn:\
|T(f_{(\alpha)})(x)|>\alpha/2\}}\r\|_{(E_{\Phi_{p_2}}^{p_2q})_t(\rn)}^{p_2}\\
&\hspace*{12pt}\lesssim\alpha^{1-p_1}\lf\|f^{(\alpha)}\r\|_{(E_{\Phi_{p_1}}^{p_1q})_t(\rn)}^{p_1}
+\alpha^{1-p_2}\lf\|f^{(\alpha)}\r\|_{(E_{\Phi_{p_2}}^{p_2q})_t(\rn)}^{p_2}\lesssim\lambda.
\end{align*}
This implies that $\|T(f)\|_{(WE_\Phi^p)_t(\rn)}\lesssim \|f\|_{(WE_\Phi^p)_t(\rn)}$, which is the desired conclusion.

Therefore, it remains to prove the above claim.
To prove \eqref{Eq2}, by Lemma \ref{LeMI}, we have
\begin{align*}
&\lf\{\int_\rn\lf\|\frac{|f^{(\alpha)}|/\alpha}{(\lambda/\alpha)^{1/p_1}}
\mathbf{1}_{B(x,t)}\r\|_{L^{\Phi_{p_1}}(\rn)}^{p_1q}\,dx\r\}^\frac{1}{p_1q}\\
&\hspace*{12pt}=\lf\{\int_\rn\lf\|\int_0^{\frac{[|f^{(\alpha)}|/\alpha]^{p_1}}{[\lambda/\alpha]}}\,d\tau
\mathbf{1}_{{B(x,t)}}\r\|_{L^{\Phi}(\rn)}^{q}\,dx\r\}^\frac{1}{p_1q}\\
&\hspace*{12pt}\lesssim\lf\{\int_\rn\lf[\int_0^\infty\lf\|\mathbf{1}_{\{y\in\rn:\
[|f^{(\alpha)}(y)|/\alpha]^{p_1}>\frac{\lambda\tau}{\alpha}\}}\mathbf{1}_{B(x,t)}
\r\|_{L^{\Phi}(\rn)}\,d\tau\r]^{q}\,dx\r\}^\frac1{p_1q}\\
&\hspace*{12pt}\lesssim\lf\{\int_0^\infty\lf[\int_\rn\lf\|\mathbf{1}_{\{y\in\rn:\
[|f^{(\alpha)}(y)|/\alpha]^{p_1}>\frac{\lambda\tau}{\alpha}\}}\mathbf{1}_{B(x,t)}
\r\|^q_{L^{\Phi}(\rn)}\,dx\r]^\frac1{q}\, d\tau\r\}^\frac1{p_1}\\
&\hspace*{12pt}\lesssim\lf\{\int_0^{\alpha/\lambda}\lf[\int_\rn\lf\|\mathbf{1}_{\{y\in\rn:\
[|f^{(\alpha)}(y)|/\alpha]^{p_1}>\frac{\lambda\tau}{\alpha}\}}\mathbf{1}_{B(x,t)}
\r\|^q_{L^{\Phi}(\rn)}\,dx\r]^\frac1{q}\, d\tau\r\}^\frac1{p_1}+\lf\{\int_{\alpha/\lambda}^\infty\cdots d\tau\r\}^\frac1{p_1}\\
&\hspace*{12pt}
=:\mathrm{I_1}+\mathrm{I_2}.
\end{align*}
From the definition of $f^{(\alpha)}$, Definition \ref{DewOs} and \eqref{Eq22}, we deduce that
\begin{align*}
\mathrm{I_1}&\lesssim\lf\{\int_0^{\alpha/\lambda}\lf[\int_\rn\lf\|\mathbf{1}_{\{y\in\rn:\
|f(y)|>\alpha\}}\mathbf{1}_{B(x,t)}
\r\|^q_{L^{\Phi}(\rn)}\,dx\r]^\frac1{q}\, d\tau\r\}^\frac1{p_1}\\
&\lesssim\lf\{\frac{\alpha}{\lambda}\lf[\int_\rn\lf\|\mathbf{1}_{\{y\in\rn:\
|f(y)|>\alpha\}}\mathbf{1}_{B(x,t)}
\r\|^q_{L^{\Phi}(\rn)}\,dx\r]^\frac1{q}\r\}^\frac1{p_1}
\lesssim \widetilde C_{(\Phi,t)}^{\frac1{p_1}},
\end{align*}
here and hereafter, $\widetilde{C}_{(\Phi,t)}$ is the same as in \eqref{Eq22}.
As for $\mathrm{I_2}$, by the definition $f^{(\alpha)}$,
Definition \ref{DewOs}, \eqref{Eq22} and $\frac1{p_1}>1$, we conclude that
\begin{align*}
\mathrm{I_2}&\lesssim\lf\{\int_{\alpha/\lambda}^\infty\lf[\int_\rn\lf\|\mathbf{1}_{\{y\in\rn:\
|f(y)|>\alpha[\frac{\lambda\tau}{\alpha}]^{1/p_1}\}}\mathbf{1}_{B(x,t)}
\r\|^q_{L^{\Phi}(\rn)}\,dx\r]^\frac1{q}\, d\tau\r\}^\frac1{p_1}\\
&\lesssim\lf\{\int_{\alpha/\lambda}^\infty\lf[\alpha^{-1}\lf(\frac{\lambda\tau}{\alpha}\r)^{-\frac1{p_1}}\r]
\lf[\lambda\widetilde C_{(\Phi,t)}\r]\,d\tau\r\}^\frac1{p_1}\lesssim \widetilde C_{(\Phi,t)}^{\frac1{p_1}}.
\end{align*}
From \eqref{Eq22} and the estimates for $\mathrm{I_1}$ and $\mathrm{I_2}$, we then deduce \eqref{Eq2}.

To prove \eqref{Eq3}, by a proof similar to the estimation of \eqref{Eq2}, we have
\begin{align*}
&\lf\{\int_\rn\lf\|\frac{|f_{(\alpha)}|/\alpha}{(\lambda/\alpha)^{1/p_2}}
\mathbf{1}_{B(x,t)}\r\|_{L^{\Phi_{p_2}}(\rn)}^{p_2q}\,dx\r\}^\frac{1}{p_2q}\\
&\hspace*{12pt}\lesssim\lf\{\int_0^{\alpha/\lambda}\lf[\int_\rn\lf\|\mathbf{1}_{\{y\in\rn:\
[|f_{(\alpha)}(y)|/\alpha]^{p_2}>\frac{\lambda\tau}{\alpha}\}}\mathbf{1}_{B(x,t)}
\r\|^q_{L^{\Phi}(\rn)}\,dx\r]^\frac1{q}\, d\tau\r\}^\frac1{p_2}
+\lf\{\int_{\alpha/\lambda}^\infty\cdots d\tau\r\}^\frac1{p_2}\\
&\hspace*{12pt}
=:\mathrm{II_1}+\mathrm{II_2}.
\end{align*}
From the definition $f_{(\alpha)}$,
Definition \ref{DewOs}, \eqref{Eq22} and $0<\frac1{p_2}<1$, we deduce that
\begin{align*}
\mathrm{II_1}&\lesssim\lf\{\int_0^{\alpha/\lambda}\lf[\int_\rn\lf\|\mathbf{1}_{\{y\in\rn:\
|f(y)|>\alpha[\frac{\lambda\tau}{\alpha}]^{1/p_2}\}}\mathbf{1}_{B(x,t)}
\r\|^q_{L^{\Phi}(\rn)}\,dx\r]^\frac1{q}\, d\tau\r\}^\frac1{p_2}\\
&\lesssim\lf\{\int_0^{\alpha/\lambda}\lf[\alpha^{-1}\lf(\frac{\lambda\tau}{\alpha}\r)^{-\frac1{p_2}}\r]
\lf[\lambda\widetilde C_{(\Phi,t)}\r]\,d\tau\r\}^\frac1{p_2}\lesssim \widetilde C_{(\Phi,t)}^{\frac1{p_2}}.
\end{align*}

Observe that, when $\tau\in(\frac{\alpha}{\lambda},\infty)$,
$(|f_{(\alpha)}|/\alpha)^{p_2}\le1<\frac{\tau\lambda}{\alpha}$ and hence $\mathrm{II_2}=0$,
which, together with the estimate for $\mathrm{II_1}$ and \eqref{Eq22}, implies \eqref{Eq3}.
Thus, we complete the proof of our above claim and hence of Theorem \ref{Th1}.
\end{proof}

Moreover, we can establish the following vector-valued inequality of
the Hardy--Littlewood operator $\cm$ in \eqref{mm} on $(WE_{\Phi}^q)_t(\rn)$, which
shows that $(E_\Phi^q)_t(\rn)$ satisfies Assumption \ref{xinm}.
\begin{proposition}\label{Profs}
Let $t\in(0,\infty)$, $q,\ s\in(1,\infty)$ and $\Phi$ be an Orlicz function with positive lower type $p_{\Phi}^-\in(1,\infty)$ and
positive upper type $p_{\Phi}^+$.
Then there exists a positive constant $C$, independent of $t$,
such that, for any sequence $\{f_j\}_{j\in\nn}\subset\mathscr{M}(\rn)$,
$$
\lf\|\lf\{\sum_{j\in\nn}[\cm(f_j)]^s\r\}^\frac1s\r\|_{(WE_{\Phi}^q)_t(\rn)}
\le C\lf\|\lf\{\sum_{j\in\nn}|f_j|^s\r\}^\frac1s\r\|_{(WE_{\Phi}^q)_t(\rn)}.
$$
\end{proposition}
\begin{proof}
Let $\{f_j\}_{j\in\nn}$ be a given arbitrary sequence of measurable functions and,
for any measurable function $g$ and $x\in\rn$, define
$$
A(g)(x):=\lf\{\sum_{j\in\nn}[\cm(g\eta_j)(x)]^s\r\}^\frac1s,
$$
where $s\in(1,\infty)$ and, for any $i\in\nn$ and $y\in\rn$,
$$
\eta_j(y):=\frac{f_j(y)}{[\sum_{j\in\nn}|f_j(y)|^s]^{1/s}}
\quad\text{when}\quad\lf[\sum_{j\in\nn}|f_j(y)|^s\r]^{1/s}\neq0,
$$
and $\eta_j(y):=0$ otherwise. It is easy to see that,
by the Minkowski inequality, for any $\lambda\in\mathbb{C}$
and $g_1,\ g_2\in\mathscr M(\rn)$,
$$
A(g_1+g_2)\le A(g_1)+A(g_2)\quad\text{and}\quad A(\lambda g)=|\lambda|A(g).
$$
Thus, $A$ is sublinear. For any $p_\Phi^-,\ q,\ s\in(1,\infty)$,
from \cite[Theorem 2.20]{ZYYW}, we deduce that
$$
\lf\|\lf\{\sum_{j\in\nn}[\cm(f_j)]^s\r\}^\frac1s\r\|_{(E_{\Phi}^q)_t(\rn)}\lesssim \lf\|\lf\{\sum_{j\in\nn}|f_j|^s\r\}^\frac1s\r\|_{(E_{\Phi}^q)_t(\rn)}.
$$
Using this, we know that, for any given $p_1\in(\frac1{\min\{p_\Phi^-,q\}},1)$ and
$p_2\in(1,\infty)$ and any $h\in\mathscr M(\rn)$,
\begin{align*}
\|A(h)\|_{(WE_{\Phi_{p_i}}^{p_iq})_t(\rn)}&=\lf\|\lf\{\sum_{j\in\nn}[\cm(h\eta_j)]^s
\r\}^\frac1s\r\|_{(WE_{\Phi_{p_i}}^{p_iq})_t(\rn)}
\leq\lf\|\lf\{\sum_{j\in\nn}[\cm(h\eta_j)]^s
\r\}^\frac1s\r\|_{(E_{\Phi_{p_i}}^{p_iq})_t(\rn)}\\
&\lesssim\lf\|\lf\{\sum_{j\in\nn}|h\eta_j|^s\r\}^\frac1s\r\|_{(E_{\Phi_{p_i}}^{p_iq})_t(\rn)}
\sim\|h\|_{(E_{\Phi_{p_i}}^{p_iq})_t(\rn)},
\end{align*}
which implies that the operator $A$ is bounded on $(WE_{\Phi_{p_i}}^{p_iq})_t(\rn)$, where $i\in\{1,2\}$.
Now, taking $g:=[\sum_{j\in\nn}|f_j(y)|^s]^{1/s}$, then, by Theorem \ref{Th1}, we conclude that
\begin{align*}
\lf\|\lf\{\sum_{j\in\nn}[\cm(f_j)]^s\r\}^\frac1s
\r\|_{(WE_{\Phi}^q)_t(\rn)}&=\|A(g)\|_{(WE_{\Phi}^q)_t(\rn)}
\lesssim\|g\|_{(WE_{\Phi}^q)_t(\rn)}\sim
\lf\|\lf\{\sum_{j\in\nn}|f_j|^s\r\}^\frac1s\r\|_{(WE_{\Phi}^q)_t(\rn)},
\end{align*}
which completes the proof of Proposition \ref{Profs}.
\end{proof}

Now we introduce the notion of weak Orlicz-slice Hardy spaces.
\begin{definition}
Let $t,\ q\in(0,\infty)$, $N\in\nn$ and $\Phi$ be an Orlicz function with positive lower type $p_{\Phi}^-$ and
positive upper type $p_{\Phi}^+$. The \emph{weak Orlicz-slice Hardy space}
$(WHE_\Phi^q)_t(\rn)$ is defined to be the set of all $f\in\cs'(\rn)$
such that $M_N^0(f)\in(WE_\Phi^{q})_t(\rn)$ and, for any $f\in(WHE_\Phi^q)_t(\rn)$, let
$$
\|f\|_{(WHE_\Phi^q)_t(\rn)}:=\lf\|M_N^0(f)\r\|_{(WE_\Phi^{q})_t(\rn)},
$$
where $M_N^0(f)$ is as in \eqref{EqMN0} with $N$ sufficiently large.
\end{definition}

\begin{remark}
Let $t\in(0,\infty)$, $q\in(1,\infty)$ and
$\Phi$ be an Orlicz function with positive lower type $p_{\Phi}^-\in(1,\infty)$
and positive upper type $p_{\Phi}^+$. By Proposition \ref{Profs},
we conclude that, for any $r\in(1,\min\{q,p_{\Phi}^-\})$,
$\cm$ in \eqref{mm} is bounded on $((WE_\Phi^{q})_t(\rn))^{1/r}$,
which, combined with Theorem \ref{dayu2}, implies that
$(WHE_\Phi^q)_t(\rn)=(WE_\Phi^{q})_t(\rn)$ with equivalent norms.
\end{remark}

Applying Proposition \ref{Profs} and Theorem \ref{Thmc}(ii),
we directly obtain the following maximal function
characterizations of the weak Orlicz-slice Hardy space $(WHE_\Phi^q)_t(\rn)$.
\begin{theorem}\label{mdj}
Let $t,\ a,\ b,\ q\in(0,\infty)$. Let $\Phi$ be an Orlicz function with positive lower type $p_{\Phi}^-$ and positive upper type $p_{\Phi}^+$. Let $\varphi\in\mathcal{S}(\rn)$ satisfy
$\int_{\rn}\varphi(x)\,dx\neq0.$
Assume that $b\in(\frac{n}{\min\{p_\Phi^-,\,q\}},\infty)$ and $N\ge\lfloor b+1\rfloor$. For any $f\in\mathcal{S}'(\rn)$,
if one of the following quantities
$$
\lf\|M_N^0(f)\r\|_{(WE_\Phi^q)_t(\rn)},\  \lf\|M(f,\varphi)\r\|_{(WE_\Phi^q)_t(\rn)},\  \lf\|M_a^*(f,\varphi)\r\|_{(WE_\Phi^q)_t(\rn)},\ \lf\|M_N(f)\r\|_{(WE_\Phi^q)_t(\rn)},
$$
$$
\lf\|M_b^{**}(f,\varphi)\r\|_{(WE_\Phi^q)_t(\rn)},\ \lf\|M_{b,\ N}^{**}(f)\r\|_{(WE_\Phi^q)_t(\rn)}
\quad and \quad\lf\|\cn(f)\r\|_{(WE_\Phi^q)_t(\rn)}
$$
is finite, then the others are also finite and mutually equivalent with the
positive equivalence constants independent of $f$ and $t$.
\end{theorem}

To establish the atomic characterization of weak Orlicz-slice Hardy spaces,
 although \cite[Lemma 4.3]{ZYYW} and Proposition \ref{Profs}
ensure that $(E_\Phi^q)_t(\rn)$ satisfies Assumption \ref{a2.15} and Assumption \ref{xinm}, we still
need the following three lemmas, which are just, respectively,
\cite[Lemmas 4.3, 4.4 and 5.4]{ZYYW}.

\begin{lemma}\label{Le817}
Let $t$, $q\in(0,\infty)$ and $\Phi$ be an Orlicz function with positive lower type
$p_{\Phi}^-$ and positive upper type $p_{\Phi}^+$. Let $\vartheta\in (0,\min\{p^-_{\Phi},q\}]$.
Then $(E_\Phi^q)_t(\rn)$ is a strictly $\vartheta$-convex
ball quasi-Banach function space as in Definition \ref{Debf}(ii).
\end{lemma}

\begin{lemma}\label{Le818}
Let $t,\ q\in(0,\infty)$ and $\Phi$ be an Orlicz function with positive lower type $p_{\Phi}^-$
and positive upper type $p_{\Phi}^+$. Let $r\in(\max\{q,\ p_{\Phi}^+\},\infty]$ and $s\in (0,\min\{p^-_{\Phi},q\})$. Then
there exists a positive constant $C_{(s,r)}$, depending on $s$ and $r$, but independent of $t$,
such that, for any $f\in\mathscr{M}(\rn)$,
\begin{equation}\label{jdjd}
\lf\|\mathcal{M}^{((r/s)')}(f) \r\|_{([(E_\Phi^q)_t(\rn)]^{1/s})'}\le C_{(s,r)}\lf\|f \r\|_{([(E_\Phi^q)_t(\rn)]^{1/s})'},
\end{equation}
here and hereafter, $[(E_\Phi^q)_t(\rn)]^{1/s}$ denotes the
$\frac1s$-convexification of $(E_\Phi^q)_t(\rn)$ as in Definition \ref{Debf}(i)
with $X:=(E_\Phi^q)_t(\rn)$ and $p:=1/s$,
and $([(E_\Phi^q)_t(\rn)]^{1/s})'$ denotes its dual space.
\end{lemma}

\begin{lemma}\label{Le819}
Let $t\in(0,\infty)$, $q\in(0,1]$ and $\Phi$ be an Orlicz function with positive lower type $p_{\Phi}^-$ and positive upper type $p_{\Phi}^+\in(0,1]$.
Then there exists a nonnegative constant $C$ such that, for any sequence $\{f_j\}_{j\in\nn}\subset
(E_\Phi^{q})_t(\rn)$ of nonnegative functions
such that $\sum_{j\in\mathbb{N}}f_j$ converges
in $(E_\Phi^{q})_t(\rn),$
$$
\lf\|\sum_{j\in\mathbb{N}}f_j\r\|_{(E_\Phi^{q})_t(\rn)}
\ge C\sum_{j\in\mathbb{N}}\lf\|f_j\r\|_{(E_\Phi^{q})_t(\rn)}.
$$
\end{lemma}

Using Proposition \ref{Profs}, Lemmas \ref{Le817}, \ref{Le818} and \ref{Le819}
and Theorems \ref{Thad}, \ref{Thar}, \ref{Thmolcha0} and \ref{Thmolcha},
we immediately obtain the following atomic characterization of $(WHE_\Phi^q)_t(\rn)$
(see Theorem \ref{yuanzi} below) and
the following molecular characterization of $(WHE_\Phi^q)_t(\rn)$
(see Theorem \ref{fenzi} below).

\begin{theorem}\label{yuanzi}
Let $t,\ q\in(0,\infty)$  and
$\Phi$ be an Orlicz function with positive lower type $p_{\Phi}^-$ and
positive upper type $p_{\Phi}^+$.
Let $p_+:=\max\{1,p_\Phi^+,q\}$ and
assume that $r\in(p_+,\infty)$
and $d\in\zz_+$ with $d\geq\lfloor n(\frac{1}{\min\{p_\Phi^-/p_+,q/p_+,1\}}-1)\rfloor$.
Then $f\in(WHE_\Phi^q)_t(\rn)$ if and only if
$$
f=\sum_{i\in \zz}\sum_{j\in\nn}\lambda_{i,j}a_{i,j}\quad\text{in}\quad\cs'(\rn)
\quad and\quad
\sup_{i\in\zz}\lf\|\sum_{j\in\nn}
\frac{\lambda_{i,j}\mathbf{1}_{B_{i,j}}}
{\|\mathbf{1}_{B_{i,j}}\|_{(E_\Phi^q)_t(\rn)}}\r\|_{(E_\Phi^q)_t(\rn)}<\infty,
$$
where $\{a_{i,j}\}_{i\in\zz,j\in\nn}$ is a sequence of $((E_\Phi^q)_t,r,d)$-atoms
supported, respectively, in balls
$\{B_{i,j}\}_{i\in\zz,j\in\nn}$ such that, for any
$i\in\zz$, $\sum_{j\in\nn}\mathbf{1}_{cB_{i,j}}\le A$ with $c\in(0,1]$ and $A$
being a positive constant independent of $f$ and $i$, and,
for any $i\in\zz$ and $j\in\nn$,
$\lambda_{i,j}:=\widetilde A2^i\|\mathbf{1}_{B_{i,j}}\|_{(E_\Phi^q)_t(\rn)}$
with $\widetilde A$ being a positive constant independent of $f$ and $i$.

Moreover, for any $f\in(WHE_\Phi^q)_t(\rn)$,
$$
\|f\|_{(WHE_\Phi^q)_t(\rn)}\sim\inf\lf\{\sup_{i\in\zz}\lf\|\sum_{j\in\nn}
\frac{\lambda_{i,j}\mathbf{1}_{B_{i,j}}}
{\|\mathbf{1}_{B_{i,j}}\|_{(E_\Phi^q)_t(\rn)}}\r\|_{(E_\Phi^q)_t(\rn)}\r\},
$$
where the infimum is taken over all decompositions of $f$ as above
and the positive equivalence constant is independent of $f$ and $t$.
\end{theorem}

We also have the following molecular characterization of $(WHE_\Phi^q)_t(\rn)$.
\begin{theorem}\label{fenzi}
Let $t,\ q$,
$\Phi$, $r$ and $d$ be the same as in Theorem \ref{yuanzi}, and $\epsilon\in(n+d+1,\infty)$.
Then $f\in(WHE_\Phi^q)_t(\rn)$ if and only if
$$
f=\sum_{i\in \zz}\sum_{j\in\nn}\lambda_{i,j}m_{i,j}\quad\text{in}\quad\cs'(\rn)
\quad and\quad
\sup_{i\in\zz}\lf\|\sum_{j\in\nn}
\frac{\lambda_{i,j}\mathbf{1}_{B_{i,j}}}
{\|\mathbf{1}_{B_{i,j}}\|_{(E_\Phi^q)_t(\rn)}}\r\|_{(E_\Phi^q)_t(\rn)}<\infty,
$$
where $\{m_{i,j}\}_{i\in\zz,j\in\nn}$ is a sequence of $((E_\Phi^q)_t,r,d,\epsilon)$-molecules associated, respectively, with balls
$\{B_{i,j}\}_{i\in\zz,j\in\nn}$ such that, for any
$i\in\zz$, $\sum_{j\in\nn}\mathbf{1}_{cB_{i,j}}\le A$ with $c\in(0,1]$ and
$A$ being a positive constant independent of $f$
and $i$, and,
for any $i\in\zz$ and $j\in\nn$, $\lambda_{i,j}:=\widetilde A2^i\|\mathbf{1}_{B_{i,j}}\|_{(E_\Phi^q)_t(\rn)}$
with $\widetilde A$ being a positive constant independent of $f$, $i$ and $j$.

Moreover, for any $f\in(WHE_\Phi^q)_t(\rn)$,
$$
\|f\|_{(WHE_\Phi^q)_t(\rn)}\sim\inf\lf[\sup_{i\in\zz}\lf\|\sum_{j\in\nn}
\frac{\lambda_{i,j}\mathbf{1}_{B_{i,j}}}{\|\mathbf{1}_{B_{i,j}}\|_{(E_\Phi^q)_t(\rn)}}
\r\|_{(E_\Phi^q)_t(\rn)}\r],
$$
where the infimum is taken over all decompositions of $f$ as above
and the positive equivalence constants are independent of $f$ and $t$.
\end{theorem}

We now recall the notion of Orlicz-slice Hardy spaces introduced in \cite{ZYYW}.
\begin{definition}
Let $t$, $q\in(0,\infty) $ and $\Phi$ be an Orlicz function with positive lower type $p_{\Phi}^-$ and positive upper type $p_{\Phi}^+$.
Then the \emph{Orlicz-slice Hardy space $(HE_\Phi^q)_t(\rn)$} is defined by setting
$$
(HE_\Phi^q)_t(\rn):=\lf\{f\in\mathcal{S}'(\rn):\ \|f\|_{(HE_\Phi^q)_t(\rn)}:=
\|M(f,\varphi)\|_{(E_\Phi^q)_t(\rn)}<\infty \r\},
$$
where $\varphi\in\mathcal{S}(\rn)$ satisfies
$
\int_{\rn}\varphi(x)\,dx\neq0.
$
In particular, when $\Phi(s):=s^r$ for any $s\in[0,\fz)$ with any given $r\in(0,\fz)$,
the Hardy-type space $(HE_r^q)_t(\rn):=(HE_{\Phi}^q)_t(\rn)$ is called
the \emph{slice Hardy space}.
\end{definition}

Recall that the \emph{centered Hardy--Littlewood maximal operator}
$\mathcal{M}_c$ is defined by setting, for
any locally integrable function $f$ and $x\in\rn$,
\begin{equation}\label{cmm}
\mathcal{M}_c(f)(x):=\sup_{r\in(0,\infty)}\dashint_{B(x,r)}|f(x)|\,dy.
\end{equation}
In what follows, for any $r\in(0,\infty)$, $f\in L^1_\loc(\rn)$ and $x\in\rn$, let
$$
\dashint_{B(x,\,r)}f(y)\,dy:=\frac1{|B(x,r)|}\int_{B(x,r)}f(y)\,dy.
$$
To obtain the boundedness of Calder\'on--Zygmund operators from $(HE_{\Phi}^q)_t(\rn)$
to $(WHE_{\Phi}^q)_t(\rn)$,
we need to establish the following Fefferman--Stein vector-valued inequality
from $(E_{\Phi}^q)_t(\rn)$
to $(WE_{\Phi}^q)_t(\rn)$.

\begin{proposition}\label{Pro822}
Let $t\in(0,\infty)$, $q\in[1,\infty),\ r\in(1,\infty)$ and
$\Phi$ be an Orlicz function with positive lower type $p_{\Phi}^-\in[1,\infty)$ and
positive upper type $p_{\Phi}^+$. Then there exists a positive constant $C$,
independent of $t$,
such that,
for any $\{f_j\}_{j\in\zz}\subset\mathscr{M}(\rn)$,
$$
\lf\|\lf\{\sum_{j\in\zz}\lf[\cm(f_j)\r]^r\r\}^\frac1r\r\|_{(WE_\Phi^{q})_t(\rn)}\le
C\lf\|\lf\{\sum_{j\in\zz}|f_j|^r\r\}^\frac1r\r\|_{(E_\Phi^{q})_t(\rn)}.
$$
\end{proposition}
\begin{proof}
Let $\alpha\in(0,\infty)$ and $r\in(1,\infty)$. For any sequence
$\{f_j\}_{j\in\zz}\subset\mathscr{M}(\rn)$ and $x\in\rn$, we claim that
\begin{align}\label{Eq27}
&\lf\|\mathbf{1}_{\{y\in B(x,t):\ \{\sum_{j\in\zz}[\cm_c(f_j)(y)]^r\}^\frac1r>\alpha\}}\r\|_{L^\Phi(\rn)}\\\noz
&\hs\lesssim\alpha^{-1}\lf\|\lf\{\sum_{j\in\zz}|f_j|^r\r\}^\frac1r\mathbf{1}_{B(x,2t)}\r\|_{L^\Phi(\rn)}\\\noz
&\hs\hs+\lf\|\mathbf{1}_{B(x,t)}\r\|_{L^\Phi(\rn)}\mathbf{1}_{\{y\in\rn:\
\{\sum_{j\in\zz}[\cm(\dashint_{B(\cdot,t)}|f_j(z)|\,dz)(y)]^r\}^\frac1r>\frac{\alpha}2\}}(x),
\end{align}
where $\cm_c$ is as in \eqref{cmm} and the implicit
positive constant is independent of $\{f_j\}_{j\in\zz}$, $x$,
$\alpha$ and $t\in(0,\infty)$.

To show this, we write
\begin{align*}
&\lf\|\mathbf{1}_{\{y\in B(x,t):\ \{\sum_{j\in\zz}[\cm_c(f_j)(y)]^r\}^\frac1r>\alpha\}}
\r\|_{L^\Phi(\rn)}\\
&\hs\lesssim\lf\|\mathbf{1}_{\{y\in B(x,t):\ \{\sum_{j\in\zz}[\sup_{s\in(0,t]}
\dashint_{B(y,s)}|f_j(z)|\,dz]^r\}^\frac1r>\frac{\alpha}2\}}\r\|_{L^\Phi(\rn)}\\
&\hs\hs+\lf\|\mathbf{1}_{\{y\in B(x,t):\ \{\sum_{j\in\zz}[\sup_{s\in(t,\infty)}
\dashint_{B(y,s)}|f_j(z)|\,dz]^r\}^\frac1r>\frac{\alpha}2\}}\r\|_{L^\Phi(\rn)}=:\mathrm{I}+\mathrm{II}.
\end{align*}
For $\mathrm{I}$, since $B(y,s)\subset B(x,2t)$ whenever $s\in(0,t]$ and $y\in B(x,t)$,
from the Orlicz Fefferman--Stein vector-valued
inequality in \cite[Theorem 1.3.1]{KK} or in \cite[Theorem 2.1.4]{YLK}, it follows that
\begin{align*}
\mathrm{I}&\sim\lf\|\mathbf{1}_{\{y\in B(x,t):\ \{\sum_{j\in\zz}[\sup_{s\in(0,t]}
\dashint_{B(y,s)}|f_j(z)|\mathbf{1}_{B(x,2t)}(z)\,dz]^r\}^\frac1r
>\frac{\alpha}2\}}\r\|_{L^\Phi(\rn)}\\
&\lesssim\lf\|\mathbf{1}_{\{y\in\rn:\ \{\sum_{j\in\zz}[\cm(f_j\mathbf{1}_{B(x,2t)})(y)|]^r\}^\frac1r>\frac{\alpha}2\}}\r\|_{L^\Phi(\rn)}
\lesssim\alpha^{-1}\lf\|\lf\{\sum_{j\in\zz}|f_j|^r\r\}^\frac1r\mathbf{1}_{B(x,2t)}\r\|_{L^\Phi(\rn)}.
\end{align*}
As for $\mathrm{II}$, observe that, for any $\xi,\ z\in\rn$, $\xi\in B(z,t)$ if and only if $z\in B(\xi,t)$ and, moreover,
if $z\in B(y,s)$ and $\xi\in B(z,t)$ with $s\in(t,\infty)$, then $\xi\in B(y,2s)$. Besides, note that $y\in B(x,t)$
and $s\in(t,\infty)$ imply that $x\in B(y,2s)$. From these observations, we deduce that
\begin{align*}
\mathrm{II}&\sim\lf\|\mathbf{1}_{\{y\in B(x,t):\ \{\sum_{j\in\zz}[\sup_{s\in(t,\infty)}
\dashint_{B(y,s)}\dashint_{B(z,t)}|f_j(z)|\,d\xi\,dz]^r\}
^\frac1r>\frac{\alpha}2\}}\r\|_{L^\Phi(\rn)}\\
&\lesssim\lf\|\mathbf{1}_{\{y\in B(x,t):\ \{\sum_{j\in\zz}[\sup_{s\in(t,\infty)}
\dashint_{B(y,2s)}\dashint_{B(\xi,t)}|f_j(z)|\,dz\,d\xi]^r\}
^\frac1r>\frac{\alpha}2\}}\r\|_{L^\Phi(\rn)}\\
&\lesssim\lf\|\mathbf{1}_{\{y\in B(x,t):\
\{\sum_{j\in\zz}[\cm(\dashint_{B(\cdot,t)}|f_j(z)|\,dz)(x)]^r\}
^\frac1r>\frac{\alpha}2\}}\r\|_{L^\Phi(\rn)}\\
&\lesssim\lf\|\mathbf{1}_{B(x,t)}\r\|_{L^\Phi(\rn)}\mathbf{1}_{\{y\in\rn:\
\{\sum_{j\in\zz}[\cm(\dashint_{B(\xi,t)}|f_j(z)|\,dz)(y)]^r\}^\frac1r>\frac{\alpha}2\}}(x).
\end{align*}
This proves the above claim.

Using \eqref{Eq27}, for any $t\in(0,\infty)$ and any given $q\in[1,\infty)$,
we further obtain
\begin{align*}
&\int_\rn\lf[\frac{1}{\|\mathbf{1}_{B(x,t)}\|_{L^\Phi(\rn)}}\lf\|\mathbf{1}_{\{y\in B(x,t):\ \{\sum_{j\in\zz}[\cm_c(f_j)(y)]^r\}^\frac1r>\alpha\}}
\r\|_{L^\Phi(\rn)}\r]^q\,dx\\
&\hs\lesssim\alpha^{-q}\int_\rn\lf[\frac{1}
{\|\mathbf{1}_{B(x,t)}\|_{L^\Phi(\rn)}}\lf\|\lf\{\sum_{j\in\zz}|f_j|^r\r\}^\frac1r
\mathbf{1}_{B(x,2t)}\r\|_{L^\Phi(\rn)}\r]^q\,dx\\
&\hs\hs+\int_\rn\mathbf{1}_{\{y\in\rn:\
\{\sum_{j\in\zz}[\cm(\dashint_{B(\cdot,t)}|f_j(z)|\,dz)(y)]^r\}^\frac1r>\frac{\alpha}2\}}(x)\,dx\\
&\hs=:\mathrm{III}+\mathrm{IV}.
\end{align*}
Since the closures of
both $B(\vec{0}_n,2t)$ and $B(\vec{0}_n,t)$ are compact subsets of $\rn$ with nonempty interiors, it
follows that there exist $N\in\mathbb{N}$ and $\{x_1,\ \ldots,\ x_N\}\subset\rn$, independent of $t$, such that
$N\lesssim1$ and $B(\vec{0}_n,2t)\subseteq\bigcup_{m=1}^NB(x_m,t)$.
Thus, by this, (\ref{Eq22}) and the translation invariance of the Lebesgue measure, we conclude that
\begin{align*}
\mathrm{III}&\sim\frac{\alpha^{-q}}{\widetilde C_{(\Phi,t)}^q}\int_\rn\lf\|\lf\{\sum_{j\in\zz}|f_j|^r\r\}^\frac1r
\mathbf{1}_{B(x,2t)}\r\|_{L^\Phi(\rn)}^q\,dx\\
&\lesssim\frac{\alpha^{-q}}{\widetilde C_{(\Phi,t)}^q}\sum_{m=1}^N\int_\rn
\lf\|\lf\{\sum_{j\in\zz}|f_j|^r\r\}^\frac1r
\mathbf{1}_{B(x+x_m,t)}\r\|_{L^\Phi(\rn)}^q\,dx\\
&\lesssim\frac{\alpha^{-q}}{\widetilde C_{(\Phi,t)}^q}\int_\rn\lf\|\lf\{\sum_{j\in\zz}|f_j|^r\r\}^\frac1r
\mathbf{1}_{B(x,t)}\r\|_{L^\Phi(\rn)}^q\,dx\\
&\lesssim\alpha^{-q}\int_\rn\lf[\frac{1}{\|\mathbf{1}_{B(x,t)}\|_{L^\Phi(\rn)}}\lf\| \lf\{\sum_{j\in\zz}|f_j|^r\r\}^\frac1r
\mathbf{1}_{B(x,t)}\r\|_{L^\Phi(\rn)}\r]^q\,dx,
\end{align*}
where $\widetilde{C}_{(\Phi,t)}$ is the same as in \eqref{Eq22},
which further implies that
\begin{equation}\label{Eq28}
\mathrm{III}^\frac1q\lesssim\alpha^{-1}\lf\|
\lf\{\sum_{j\in\zz}|f_j|^r\r\}^\frac1r\r\|_{(E_\Phi^q)_t(\rn)}.
\end{equation}
It turns to estimate $\mathrm{IV}$. By the Fefferman--Stein vector-valued
inequality from $L^q(\rn)$ to $WL^q(\rn)$ with $q\in[1,\infty)$
(see \cite[(1) and (2) of Theorem 1]{FS}), for any $\alpha\in(0,\infty)$, we have
$$
\mathrm{IV}^\frac1q\lesssim\alpha^{-1}\lf\|
\lf\{\sum_{j\in\zz}\lf[\dashint_{B(\cdot,t)}|f_j(z)|\,dz\r]^r\r\}^\frac1r\r\|_{L^q(\rn)}.
$$
Let $r':=\frac{r}{r-1}$. Then there exists $\{b_j\}_{j\in\mathbb{Z}}\in \ell^{r'}$,
with $\|\{b_j\}_{j\in\mathbb{Z}}\|_{\ell^{r'}}=1,$ such that
$$\int_{\rn}\lf\{\sum_{j\in\mathbb{Z}}\lf[\dashint_{B(x,t)}|f_j(z)|\,dz\r]^r\r\}^{\frac{q}{r}}dx
=\int_{\rn}\lf[\sum_{j\in\mathbb{Z}}b_j\dashint_{B(x,t)}|f_j(z)|\,dz\r]^{q}\,dx.$$
From \cite[p.\,13, Proposition 1]{RR}, we deduce that, for any ball $B(x,t)$, $\Phi^{-1}(|B(x,t)|)\Psi^{-1}(|B(x,t)|)\sim|B(x,t)|$,
where the positive equivalence constants are independent of $x$ and $t$.
This, together with the H\"older inequality and \eqref{Eq22}, further implies that
\begin{align*}
\int_{\rn}\lf[\sum_{j\in\mathbb{Z}}b_j\dashint_{B(x,t)}|f_j(z)|\,dz\r]^{q}\,dx
&\lesssim\int_{\rn}\lf\{\dashint_{B(x,t)}\lf[\sum_{j\in\mathbb{Z}}|f_j(z)|^r\r]^{\frac{1}{r}}
\lf(\sum_{j\in\mathbb{Z}}b_j^{r'}\r)^{\frac{1}{r'}}\,dz\r\}^{q}\,dx\\
&\lesssim\int_{\rn}\lf\{\lf\|\lf[\sum_{j\in\mathbb{Z}}\lf|f_j\r|^r\r]^{\frac{1}{r}}
\mathbf{1}_{B(x,t)}\r\|_{L^\Phi(\rn)}\frac{\|\mathbf{1}_{B(x,t)}\|_{L^\Psi(\rn)}}{|B(x,t)|}\r\}^q\, dx\\
&\lesssim\int_{\rn}\lf\{\frac{1}{\|\mathbf{1}_{B(x,t)}\|_{L^\Phi(\rn)}}\lf\|
\lf[\sum_{j\in\mathbb{Z}}|f_j|^r\r]^{\frac{1}{r}}
\mathbf{1}_{B(x,t)}\r\|_{L^\Phi(\rn)}\r\}^q\, dx.
\end{align*}
Thus,
$$
\mathrm{IV}^\frac1q
\lesssim\alpha^{-1}\lf\|\lf\{\sum _{j\in\mathbb{Z}}|f_j|^r\r\}^{\frac{1}{r}}\r\|_{(E_\Phi^q)_t(\rn)},
$$
which, combined with \eqref{Eq27} and \eqref{Eq28},
then completes the proof of Proposition \ref{Pro822}.
\end{proof}

Applying Proposition \ref{Pro822}, Theorems \ref{Thcz} and \ref{Thcz1},
we directly obtain the following boundedness from
$(HE_{\Phi}^q)_t(\rn)$ to $(WHE_{\Phi}^q)_t(\rn)$
of both convolutional $\delta$-type and $\gamma$-type
Calder\'on--Zygmund operators, respectively, as follows.
\begin{theorem}\label{Thcz81}
Let $t\in(0,\infty)$, $q\in(0,\infty),\ \delta\in(0,1]$ and $\Phi$ be an Orlicz function
with positive lower type $p_{\Phi}^-$ and
positive upper type $p_{\Phi}^+$. Let $T$ be a convolutional $\delta$-type Calder\'on--Zygmund
operator. If $\min\{p_\Phi^-,\ q\}\in[\frac{n}{n+\delta},1]$, then $T$ has a unique extension on $(HE_\Phi^q)_t(\rn)$ and, moreover, there exists a positive constant $C$, independent of
$t$, such that,
for any $f\in (HE_\Phi^q)_t(\rn)$,
$$
\|Tf\|_{(WHE_{\Phi}^q)_t(\rn)}\le C\|f\|_{(HE_\Phi^q)_t(\rn)}.
$$
\end{theorem}

\begin{theorem}\label{Thcz82}
Let $t\in(0,\infty)$,  $q\in(0,2),\ \gamma\in(0,\infty)$ and $\Phi$ be an Orlicz function
with positive lower type $p_{\Phi}^-$ and
positive upper type $p_{\Phi}^+\in(0,2)$.  Let $T$ be a $\gamma$-type Calder\'on--Zygmund
operator and have the vanishing moments up to order $\lceil\gamma\rceil-1$.
If $\lceil\gamma\rceil-1\le n(\frac1{\min\{p_\Phi^-,\ q\}}-1)\le\gamma$, then $T$ has a unique extension on
$(HE_\Phi^q)_t(\rn)$ and, moreover, there exists a
positive constant $C$, independent of
$t$, such that,
for any $f\in(HE_\Phi^q)_t(\rn)$,
$$
\|Tf\|_{(WHE_{\Phi}^q)_t(\rn)}\le C\|f\|_{(HE_\Phi^q)_t(\rn)}.
$$
\end{theorem}

%%%%%%%%%%%%%%%%%%%%%%%%%%%%%%%%%%%%%%%%%%%%%%%%%%%%%%%%%%%%%%%%%%%%%%%%%%%%%%%%%%

%%%%%%%%%%%%%%%%%%%%%%%%%%%%%%%%%%%%%%%%%%%%%%%%%%%%%%%%%%%%%%%%%%%%%%%%%%%%%%%%%%

\bigskip

\noindent Yangyang Zhang, Dachun Yang (Corresponding author) and Wen Yuan

\smallskip

\noindent  Laboratory of Mathematics and Complex Systems
(Ministry of Education of China),
School of Mathematical Sciences, Beijing Normal University,
Beijing 100875, People's Republic of China

\smallskip

\noindent {\it E-mails}: \texttt{yangyzhang@mail.bnu.edu.cn} (Y. Zhang)

\noindent\phantom{{\it E-mails:}} \texttt{dcyang@bnu.edu.cn} (D. Yang)

\noindent\phantom{{\it E-mails:}} \texttt{wenyuan@bnu.edu.cn} (W. Yuan)

\bigskip

\noindent Songbai Wang

\smallskip

\noindent College of Mathematics and Statistics, Hubei Normal University,
Huangshi 435002, People's Republic of China

\smallskip

\noindent{\it E-mail}: \texttt{haiyansongbai@163.com}

\end{document}